\documentclass[11pt,a4paper]{amsart}
\usepackage{graphicx}
\usepackage{amssymb,amsmath}
\usepackage{amsthm}
\usepackage{epstopdf}
\usepackage{mathtools}
\usepackage{slashed}
\usepackage{xstring}
\usepackage{hyperref}
\usepackage{tikz-cd}
\usepackage[parfill]{parskip}
\usepackage{graphicx}
\usepackage{amssymb}
\usepackage{epstopdf}
\usepackage{enumerate}
\newcommand{\infop}[0]{\bar\partial+\bar\partial^*}
\newcommand{\hook}[0]{\lrcorner\,}

\newtheorem{thm}{Theorem}
\newtheorem*{thm*}{Theorem}
\newtheorem{lem}[thm]{Lemma}
\newtheorem{prop}[thm]{Proposition}
\newtheorem*{prop*}{Proposition}
\newtheorem{cor}[thm]{Corollary}
\newtheorem*{cor*}{Corollary}

\theoremstyle{definition}
\newtheorem*{rem}{Remark}
\newtheorem{defn}{Definition}
\newtheorem*{ex}{Example}

\numberwithin{thm}{section}
\numberwithin{equation}{section}
\numberwithin{defn}{section}

\title{Deformations of conically singular Cayley submanifolds}
\author{Kim Moore}
\address{Department of Mathematics, University College London, Gower Street, London, WC1E  6BT}
\email{kim.moore@ucl.ac.uk}

\begin{document}
\begin{abstract}
In this article we study the deformation theory of conically singular Cayley submanifolds. In particular, we prove a result on the expected dimension of a moduli space of Cayley deformations of a conically singular Cayley submanifold. Moreover, when the Cayley submanifold is a two-dimensional complex submanifold of a Calabi--Yau four-fold we show by comparing Cayley and complex deformations that in this special case the moduli space is a smooth manifold. We also perform calculations of some of the quantities discussed for some examples.
\end{abstract}

\maketitle

\section{Introduction}

Cayley submanifolds are calibrated submanifolds that arise naturally in manifolds with exceptional holonomy $Spin(7)$. Introduced by Harvey and Lawson \cite{MR666108}, calibrated submanifolds are by construction volume minimising, and hence minimal submanifolds. Cayley submanifolds exist in abundance, with the simplest examples being any two-dimensional complex submanifold of a Calabi--Yau four-fold.

The deformation theory of compact calibrated submanifolds in manifolds with special holonomy was studied by McLean \cite{MR1664890}. A major obstruction to generalising these results to noncompact manifolds is the failure of an elliptic operator on a noncompact manifold to be Fredholm. However, by introducing a weighted norm on spaces of sections of a given vector bundle on a particular type of noncompact manifold, it is possible to overcome this difficulty, as long as one is careful about the choice of weight. It is therefore possible to study certain moduli spaces of noncompact calibrated submanifolds.

In this article, the noncompact submanifolds that we study are \emph{conically singular}. Motivated by the SYZ conjecture, an interesting problem is whether a $Spin(7)$-manifold can be fibred by Cayley submanifolds with some singular fibres. Conically singular Cayley submanifolds are natural candidates for these singular fibres. Deformations of conically singular special Lagrangian submanifolds in Calabi--Yau manifolds are coassociative submanifolds of $G_2$-manifolds have been studied by Joyce \cite{MR2015549} and Lotay \cite{MR2403190} respectively. Deformations of compact Cayley submanifolds with boundary and asymptotically cylindrical Cayley submanifolds have been studied by Ohst \cite{ohst2014deformations,ohst2015deformations}.

We say that a manifold with a singular point is conically singular if a neighbourhood of the singular point is diffeomorphic to a cone $C\cong L\times (0,\epsilon)$, and moreover the metric approaches the cone metric like $r^{\mu-1}$ as $r\to 0$. We will prove a series of results on the moduli space of Cayley deformations of a Cayley submanifold, conically singular with cone $C$ and rate $\mu$, that also have a conical singularity at the same point with cone $C$ and rate $\mu$.

In Theorem \ref{thm:cscaydefs}, we prove that the expected dimension of this moduli space is given by the index of a first order linear elliptic operator acting on smooth normal vector fields that decay like $r^\mu$ close to the singular point. Motivated by other work of the author \cite{compactcay}, we pay special attention to Cayley deformations of a conically singular complex surface $N$ inside a Calabi--Yau four-fold $M$. In Theorem \ref{thm:cscxcaydefs}, we will show that the elliptic operator in Theorem \ref{thm:cscaydefs} is 
\[
\infop:C^\infty_\mu( \nu^{1,0}_M(N)\oplus \Lambda^{0,2}N\otimes \nu^{1,0}_M(N))\to C^\infty_{\textnormal{loc}}(\Lambda^{0,1}N\otimes \nu^{1,0}_M(N)).
\]
We will then study the moduli space of complex deformations of $N$ in $M$ that are conically singular at the same point with the same rate and cone as $N$. We will show in Theorem \ref{thm:cxcsdefs} that this moduli space is a smooth manifold, and moreover that there are no infinitesimal Cayley deformations of $N$ that are not infinitesimal complex deformations of $N$ in Corollary \ref{cor:cscxcaysame}.

In the later sections of this article we will focus on the elliptic operators whose indices we are interested in. In particular, we will characterise the exceptional weights for which these operators are not Fredholm. We will also apply the Atiyah--Patodi--Singer Index Theorem \cite{MR0397797} to write down an expression for the index of these operators in terms of topological and spectral invariants of the manifold. 

We will conclude this article by performing a series of calculations, including the dimension of the space of infinitesimal Cayley and complex deformations of three complex cones in $\mathbb{C}^4$, that are themselves cones, motivated by the work of Kawai \cite{MR3672213} on deformations of associative submanifolds of the seven-sphere.

\textbf{Layout.} In Section \ref{sec:prelim} we will discuss some background results which the reader may find useful on Cayley submanifolds, conically singular manifolds and Fredholm theory on noncompact manifolds. Section \ref{sec:csdefs} contains the results on the deformation theory of conically singular Cayley and complex submanifolds. In Section \ref{sec:index} we characterise the set $\mathcal{D}$ of exceptional weights for which the operators we discuss in this article, before deducing a version of the Atiyah--Patodi--Singer theorem for these operators. In Section \ref{sec:calc} we perform calculations of some of the quantities discussed in this article for some examples.

\textbf{Notation.} When $M$ is a complex manifold, we denote by $\Lambda^{p,q}M$ the bundle of $(p,q)$-forms $\Lambda^pT^{*1,0}M\otimes \Lambda^qT^{*0,1}M$. If $N$ is a submanifold of $M$, we denote the normal bundle of $N$ in $M$ by $\nu_M(N)$. Moreover, if $N$ is a complex submanifold of $M$ then we denote by $\nu^{1,0}_M(N)$ and $\nu^{0,1}_M(N)$ the holomorphic and anti-holomorphic normal bundles of $N$ in $M$ respectively. Submanifolds will be taken to be embedded unless otherwise stated.
\section{Preliminaries}\label{sec:prelim}
\subsection{Cayley submanifolds}\label{ss:cay}

We follow Joyce \cite[Defn 11.4.2]{MR2292510} to define $Spin(7)$-manifolds.
\begin{defn}\label{defn:spin7}
 Let $(x_1,\dots,x_8)$ be coordinates on $\mathbb{R}^8$ with the Euclidean metric $g_0=dx_1^2+\dots+dx_8^2$. Define a four-form on $\mathbb{R}^8$ by
\begin{align}\nonumber
 \Phi_0:=&dx_{1234}-dx_{1256}-dx_{1278}-dx_{1357}+dx_{1368}-dx_{1458}-dx_{1467} \\\label{eqn:phi0}
-&dx_{2358}-dx_{2367}+dx_{2457}-dx_{2468}-dx_{3456}-dx_{3478}+dx_{5678},
\end{align}
where $dx_{ijkl}:=dx_i\wedge dx_j\wedge dx_k\wedge dx_l$.

Let $X$ be an eight-dimensional oriented manifold. For each $p\in X$ define the subset $\mathcal{A}_pX\subseteq \Lambda^4T^*_pX$ to be the set of four-forms $\Phi$ for which there exists an oriented isomorphism $T_pX\to \mathbb{R}^8$ identifying $\Phi$ and $\Phi_0$ given in \eqref{eqn:phi0}, and define the vector bundle $\mathcal{A}X$ to be the vector bundle with fibre $\mathcal{A}_pX$.

A four-form $\Phi$ on $X$ that satisfies $\Phi|_p\in \mathcal{A}_pX$ for all $p\in X$ defines a metric $g$ on $X$. We call $(\Phi,g)$ a $Spin(7)$-\emph{structure} on $X$. If $\nabla$ denotes the Levi-Civita connection of $g$ then say $(\Phi,g)$ is a \emph{torsion-free} $Spin(7)$-structure on $X$ if $\nabla \Phi=0$.

Then $(X,\Phi,g)$ is a $Spin(7)$-\emph{manifold} if $X$ is an eight-dimensional oriented manifold and $(\Phi,g)$ is a torsion-free $Spin(7)$-structure on $X$.
\end{defn}

Given a $Spin(7)$-manifold $(X,\Phi,g)$ then $\Phi$ is a calibration on $X$, known as the \emph{Cayley calibration}. An oriented, four-dimensional submanifold $Y$ of $X$ is said to be Cayley if 
\[
\Phi|_Y=\text{vol}_Y,
\]
i.e., $Y$ is calibrated by $\Phi$.

\begin{defn}
Let $(M^m,J,\omega')$ be a compact K\"{a}hler manifold with trivial canonical bundle $K_M:= \Lambda^{m,0}M$, i.e., with nowhere vanishing section $\alpha$ with $\bar\partial\alpha=0$. Then by Yau's proof of the Calabi conjecture, there exists a Ricci-flat K\"{a}hler form $\omega\in [\omega']$. Choose $\Omega\in \Omega^{m,0}(M)$ so that
\begin{equation}\label{eqn:omOm}
\frac{\omega^m}{m!}=\left(\frac{i}{2}\right)^m(-1)^{m(m-1)/2}\Omega\wedge \overline{\Omega}.
\end{equation}
Say that $(M,J,\omega,\Omega)$ is a Calabi--Yau manifold.
\end{defn}
Given a Calabi--Yau four-fold $(M,J,\omega,\Omega)$, we can define a Cayley form on $M$ by
\begin{equation}\label{eqn:caycy}
\Phi=\frac{1}{2}\omega\wedge \omega+\text{Re }\Omega,
\end{equation}
and so with the choice of constant in \eqref{eqn:omOm} we can view $M$ as a $Spin(7)$-manifold. Examining expression \eqref{eqn:caycy} we see that complex surfaces and special Lagrangians in a Calabi--Yau four-fold are Cayley submanifolds.

We can decompose bundles of forms on $Spin(7)$-manifolds into irreducible representations of $Spin(7)$. The following proposition is taken from \cite[Prop 11.4.4]{MR2292510}.
\begin{prop}\label{prop:spin7decomp}
Let $X$ be a $Spin(7)$-manifold. Then the bundle of two-forms $M$ admits the following decomposition into irreducible representations of $Spin(7)$:
\begin{align*}
 \Lambda^2X&\cong \Lambda^2_7\oplus \Lambda^2_{21},
\end{align*}
where $\Lambda^k_l$ denotes the irreducible representation of $Spin(7)$ on $k$-forms of dimension $l$.
\end{prop}
\begin{rem}
If $Y$ is a Cayley submanifold of $X$, then we have that
\begin{equation}\label{eqn:l27decomp}
\Lambda^2_7|_Y=\Lambda^2_+Y\oplus E,
\end{equation}
where $E$ is a rank four vector bundle on $Y$.
\end{rem}

The following result allows us to characterise Cayley submanifolds of a $Spin(7)$-manifold $(X,\Phi,g)$ by finding a four-form that vanishes exactly when restricted to a Cayley submanifold of $X$.
\begin{prop}[{\cite[Lem 10.15]{salamon2010notes}}]\label{prop:tau}
Let $X$ be a eight-dimensional manifold with $Spin(7)$-structure $(\Phi,g)$ and let $Y$ be an oriented four-dimensional submanifold of $X$. Then there exists $\tau\in C^\infty(\Lambda^4 X\otimes \Lambda^2_7)$ so that $Y$ is a Cayley submanifold of $X$ if, and only if, $\tau|_Y\equiv 0$.

If $x,u,v,w$ are orthogonal, then
\[
\tau(x,u,v,w)=\pi_7(\Phi(\,\cdot \, , u,v,w)\wedge x^\flat),
\]
where  $\pi_7(x^\flat\wedge y^\flat)=\frac{1}{2}(x^\flat \wedge y^\flat+\Phi(x,y,\cdot,\cdot))$ and $\flat$ denotes the musical isomorphism $TX\to T^*X$. Moreover, if $e_1,\dots , e_8$ is an orthonormal frame for $TX$ then
\begin{equation}\label{eqn:taunice}
\tau=\sum_{i=2}^8 \left[e^i\wedge (e_1\hook \Phi)-e^1\wedge (e_i\hook \Phi)\right]\otimes \pi_7(e^1\wedge e^i)
\end{equation}
\end{prop}
\subsection{Deformation theory of compact Cayley submanifolds}\label{ss:cptdefs}
We begin by studying a compact Cayley submanifold $Y$ of a $Spin(7)$-manifold $X$. The results here are due to McLean \cite[\S 6]{MR1664890}, although are taken in this form from a paper of the author \cite{compactcay}. We first use the tubular neighbourhood theorem to identify the moduli space of Cayley deformations of $Y$ in $X$ with the kernel of a partial differential operator.
\begin{prop}[{\cite[Thm 6.3]{MR1664890}}] \label{prop:caylin}
 Let $(X,g,\Phi)$ be a $Spin(7)$-manifold with compact Cayley submanifold $Y$.  Let $\exp$ denote the exponential map and for a normal vector field $v$ define $Y_v:=\exp_v(Y)$. The moduli space of Cayley deformations of $Y$ in $X$ is isomorphic near $Y$ to the kernel of the following partial differential operator
\begin{align}\nonumber
 F:C^\infty(V)&\to C^\infty(E), \\ \label{eqn:caypde}
v&\mapsto \pi(*_Y\exp^*_v(\tau|_{Y_v})),
\end{align}
where $\tau$ is defined in Proposition \ref{prop:tau}, $V$ is an open neighbourhood of the zero section in $\nu_X(Y)$ and 
\begin{equation}\label{eqn:27split}
\Lambda^2_7|_Y=\Lambda^2_+Y\oplus E,
\end{equation}
with $\pi:\Lambda^2_7|_Y\to E$ the projection map.

Moreover, we have that the the linearisation of $F$ at zero is the operator
\begin{align}\nonumber
D:C^\infty(\nu_X(Y))&\to C^\infty(E), \\ \label{eqn:caylin}
 v&\mapsto\sum_{i=1}^4 \pi_7(e^i\wedge (\nabla^\perp_{e_i}v)^\flat),
\end{align}
where $\{e_1,e_2,e_3,e_4\}$ is a frame for $TY$ with dual coframe $\{e^1,e^2,e^3,e^4\}$, $\nabla^\perp:TY\otimes \nu_X(Y)\to  \nu_X(Y)$ denotes the connection on $\nu_X(Y)$ induced by the Levi-Civita connection of $X$ and $\pi_7$ denotes the projection of two-forms onto $\Lambda^2_7$ as in Proposition \ref{prop:tau}.
\end{prop}
\begin{proof}
We will recap the proof that the linearisation takes the claimed form since McLean's version of this result is presented differently.

We have that
\[
dF|_0(v)=\frac{d}{dt}F(tv)|_{t=0}=*_N\mathcal{L}_v\tau|_N.
\]
Take an orthonormal frame $\{e_1,\dots , e_8\}$ for $TM$ so that $TN=\{e_1,e_2,e_3,e_4\}$, and in this frame
\begin{align*}
\Phi&=e^{1234}-e^{1256}-e^{1278}-e^{1357}+e^{1368}-e^{1458}-e^{1467} \\
&-e^{2358}-e^{2367}+e^{2457}-e^{2468}-e^{3456}-e^{3478}+e^{5678},
\end{align*}
where $e^i=g(e_i,\, \cdot \,)$ and $e^{ijkl}:=e^i\wedge e^j\wedge e^k\wedge e^l$.
We have that
\[
*_N\mathcal{L}_v\tau|_N=(\mathcal{L}_v\tau)(e_1,e_2,e_3,e_4).
\]
Using a formula such as \cite[Eqn (4.3.26)]{MR2829653}, we find that
\begin{align*}
(\mathcal{L}_v\tau)(e_1,e_2,e_3,e_4)&=(\nabla_v\tau)(e_1, e_2,e_3,e_4)+\tau(\nabla^\perp_{e_1} v,e_2,e_3,e_4) \\
&-\tau(\nabla^\perp_{e_2}v,e_1,e_3,e_4)+\tau(\nabla^\perp_{e_3}v,e_1,e_2,e_4) \\
&-\tau(\nabla^\perp_{e_4}v,e_1,e_2,e_3),
\end{align*}
where we have used that $\tau$ vanishes on four tangent vectors to a Cayley submanifold. By definition of $\tau$ given in Proposition \ref{prop:tau}, we have that
\[
(\mathcal{L}_v\tau)(e_1,e_2,e_3,e_4)=(\nabla_v\tau)(e_1, e_2,e_3,e_4)+\sum_{i=1}^4 \pi_7(e^i\wedge (\nabla^\perp_{e_i}v)^\flat),
\]
and so it remains to show that if $\Phi$ is parallel then so is $\tau$. From Equation \eqref{eqn:taunice}, we see that
\begin{align*}
\nabla_v\tau&=\sum_{i=2}^8\nabla_v\left[e^i\wedge (e_1\hook \Phi)-e^1\wedge (e_i\hook\Phi)\right]\otimes \pi_7(e^1\wedge e^i) \\
&+\sum_{i=2}^8\left[e^i\wedge (e_1\hook \Phi)-e^1\wedge (e_i\hook\Phi)\right]\otimes \nabla_v\pi_7(e^1\wedge e^i).
\end{align*}
We can see that the second sum in the above expression will vanish when evaluated on $e_1,e_2,e_3,e_4$, so it remains to compute
\[
\nabla_v\left[e^i\wedge (e_1\hook \Phi)-e^1\wedge (e_i\hook\Phi)\right],
\]
for $i=2,\dots ,8$. Since 
\[
\nabla_v(e^1\wedge (e_i\hook \Phi))=(\nabla_v e^1)\wedge (e_i\hook \Phi)+e^1\wedge (\nabla_ve_i\hook \Phi)+e^1\wedge (e_i\hook\nabla_v\Phi),
\]
we find that
\begin{align*}
\nabla_v(e^1\wedge (e_i\hook \Phi))(e_1,e_2,e_3,e_4)&=e^2(e_i)(\nabla_ve^1)(e_2)+e^3(e_i)(\nabla_v e^1)(e_3) \\
&+e^4(e_i)(\nabla_ve^1)(e_4)+\Phi(\nabla_v e_i,e_2,e_3,e_4) \\
&+(\nabla_v\Phi)(e_i,e_2,e_3,e_4).
\end{align*}
Similarly,
\begin{align*}
\nabla_v(e^i\wedge (e_1\hook \Phi))(e_1,e_2,e_3,e_4)&=(\nabla_v e^i)(e_1)-e^i(e_2)\Phi(\nabla_ve_1,e_1,e_3,e_4) \\
&+e^i(e_3)\Phi(\nabla_ve_1,e_1,e_2,e_4) \\
&-e^i(e_4)\Phi(\nabla_ve_1,e_1,e_2,e_3).
\end{align*}
Using the explicit expression for $\Phi$, we have that
\begin{align*}
(\nabla_v\tau)(e_1,e_2,e_3,e_4)&=\sum_{i=2}^8\left[e^i(e_2)e^2(\nabla_ve_1)-e^2(e_i)(\nabla_ve^1)(e_2) \right. \\
+e^i(e_3)e^3(\nabla_ve_1)&-e^3(e_i)(\nabla_ve^1)(e_3)+e^i(e_4)e^4(\nabla_ve_1)-e^4(e_i)(\nabla_ve^1)(e_4) \\
-e^1(\nabla_v e_i)&\left.+(\nabla_v e^i)(e_1)-(\nabla_v\Phi)(e_i,e_2,e_3,e_4)\right]\otimes \pi_7(e^1\wedge e^i).
\end{align*}
Finally, note that since the metric $g$ on $X$ is parallel with respect to the Levi-Civita connection,
\begin{align*}
(\nabla_ve^j)(e_k)&=-e^j(\nabla_v e_k)=-g(\nabla_v e_k,e_j)=g(e_k,\nabla_v e_j)=e^k(\nabla_v e_j) \\
&=-(\nabla_v e^k)(e_j),
\end{align*}
and so we find that
\[
(\nabla_v\tau)(e_1,e_2,e_3,e_4)=\sum_{i=5}^8-(\nabla_v\Phi)(e_i,e_1,e_2,e_3,e_4)\otimes\pi_7(e^1\wedge e^i),
\]
which vanishes since $\Phi$ is parallel.
\end{proof}
Now suppose that $M$ is a four-dimensional Calabi--Yau manifold and $N$ is a two-dimensional complex submanifold of $M$. We can apply the above results to study the Cayley deformations of $N$ in $M$, but we will exploit the complex structure of $N$ and $M$ to give these results nicer forms. The following results are due to the author, and proofs can be found in \cite{compactcay}.

To begin with, we identify the normal bundle and $E$ with natural vector bundles on $N$.
\begin{prop}[{\cite[Prop 3.2 and 3.3]{compactcay}}]\label{prop:bundleisom}
Let $N$ be a two-dimensional submanifold of a Calabi-Yau four-fold $M$. Then
\begin{equation}\label{eqn:normisom}
\nu_M(N)\otimes \mathbb{C}\cong \nu^{1,0}_M(N)\oplus \Lambda^{0,2}N\otimes \nu^{1,0}_M(N),
\end{equation}
via the map
\[
v\mapsto \frac{1}{4}(v\hook \overline{\Omega})^\sharp, 
\]
where $\Omega$ is the holomorphic volume form on $M$ and $\sharp:\nu^{*0,1}_M(N)\to \nu^{1,0}_M(N)$, and
\begin{equation}\label{eqn:eisom}
E\otimes \mathbb{C} \cong \Lambda^{0,1}N\otimes \nu^{1,0}_M(N),
\end{equation}
where $\nu_M^{1,0}(N)$ denotes the holomorphic normal bundle of $N$ in $M$.
\end{prop}
With these isomorphisms in place, we can modify Proposition \ref{prop:caylin}.
\begin{prop}[{\cite[Prop 3.5]{compactcay}}]\label{prop:caylin2}
Let $N$ be a complex surface in a Calabi--Yau manifold $M$. Then the infinitesimal Cayley deformations of $N$ in $M$ can be identified with the kernel of the operator
\begin{equation}\label{eqn:infopcs}
\infop:C^\infty(\nu^{1,0}_M(N)\oplus \Lambda^{0,2}N\otimes \nu^{1,0}_M(N))\to C^\infty(\Lambda^{0,1}N\otimes \nu^{1,0}_M(N)).
\end{equation}
\end{prop}

We will now apply McLean's method to study the complex deformations of $N$ in $M$. We begin by finding a form which vanishes exactly when restricted to a two-dimensional complex submanifold.

\begin{prop}[{\cite[Prop 4.2]{compactcay}}]\label{prop:compdefs}
Let $Y$ an oriented four-dimensional submanifold of a four-dimensional Calabi--Yau manifold $M$. Then $Y$ is a complex submanifold of $M$ iff for all vector fields $u,v,w$ on $Y$,
\[
\sigma(u,v,w)= 0,
\]
where $\sigma(u,v,w):=\textnormal{Re }\Omega(u,v,w,\cdot)$, where $\Omega$ is the holomorphic volume form of $M$.
\end{prop}

We can now define a partial differential operator whose kernel can be identified with the moduli space of complex deformations of $N$ in $M$.
\begin{prop}[{\cite[Prop 4.3]{compactcay}}]\label{prop:comppde}
 Let $N$ be a compact complex surface inside a four-dimensional Calabi--Yau manifold $M$. Let $\exp$ denote the exponential map and for a normal vector field $v$ define $N_v:=\exp_v(N)$. Then the moduli space of complex deformations of $N$ is isomorphic near $N$ to the kernel of
\begin{align}\nonumber
 G:C^\infty(V\otimes \mathbb{C})&\to C^\infty(\Lambda^1N\otimes T^*M|_N\otimes \mathbb{C}), \\ \label{eqn:cxop}
v&\mapsto *_N\exp_v^*(\sigma|_{N_v}),
\end{align}
where $\sigma$ was defined in Proposition \ref{prop:compdefs} and $V$ is an open neighbourhood of the zero section in $\nu_M(N)$.
\end{prop}

We can find the linear part of $G$.
\begin{prop}[{\cite[Prop 4.4]{compactcay}}]\label{prop:glin}
 Let $N$ be a compact complex surface in a Calabi--Yau four-fold $M$. Let $G$ be the partial differential operator defined in Proposition \ref{prop:compdefs}. Then the linearisation of $G$ at zero is equal to the operator
\[
 v\mapsto -\partial^*(v\hook\Omega)-\bar\partial^*(v\hook\overline{\Omega}),
\]
where $v\in C^\infty(\nu_M(N)\otimes \mathbb{C})$. Therefore $v$ is an infinitesimal complex deformation of $N$ if, and only if,
\[
 \partial^*(v\hook \Omega)=0=\bar\partial^*(v\hook \overline{\Omega}).
\]

Moreover, we have that, if $v=v_1\oplus v_2$ where $v_1\in \nu^{1,0}_M(N)$ and $v_2\in \nu^{0,1}_M(N)$
\[
 \partial^*(v_1\hook\Omega)=0 \iff \bar\partial v_1=0.
\]
\end{prop}

So we can see from this result (in combination with the explicit isomorphism given in Proposition \ref{prop:bundleisom}) that an infinitesimal Cayley deformation of $N$, $v\oplus w \in C^\infty(\nu^{1,0}_M(N)\oplus \Lambda^{0,2}N\otimes \nu^{1,0}_M(N))$ such that $\bar\partial v +\bar\partial^* w=0$ is a complex deformation of $N$ if and only if $\bar\partial v=0=\bar\partial^*w$. Therefore an infinitesimal Cayley deformation of $N$ in $M$ that is not complex would satisfy $\bar\partial v=-\bar\partial^* w$.
With a little bit more work, we can prove the following theorem.

\begin{thm}[{\cite[Thm 4.9]{compactcay}}]\label{thm:maincomp}
 Let $N$ be a compact complex surface inside a four-dimensional Calabi--Yau manifold $M$. Then the moduli space of Cayley deformations of $N$ in $M$ near $N$ is isomorphic to the moduli space of complex deformations of $N$ in $M$, which near $N$ is a smooth manifold of dimension
\[
 \textnormal{dim}_\mathbb{C} \textnormal{Ker }\bar\partial+\textnormal{dim}_\mathbb{C} \textnormal{Ker }\bar\partial^*=2 \,\textnormal{dim}_{\mathbb{C}} \textnormal{Ker }\bar\partial,
\]
where 
\begin{align*}
 \bar\partial:C^\infty(\nu^{1,0}_M(N))&\to C^\infty(\Lambda^{0,1}N\otimes \nu^{1,0}_M(N)), \\
\bar\partial^*:C^\infty(\Lambda^{0,2}N\otimes\nu^{1,0}_M(N))&\to C^\infty(\Lambda^{0,1}N\otimes \nu^{1,0}_M(N)).
\end{align*}
\end{thm}
\begin{rem}
Comparing this to Kodaira's theorem \cite[Thm 1]{MR0133841} on the deformation theory of compact complex submanifolds, we see that we agree with the infinitesimal deformation space, but in this special case where the ambient manifold is Calabi--Yau we can integrate all infinitesimal complex deformations to true complex deformations.
\end{rem}

\subsection{Conically singular and asymptotically cylindrical manifolds}\label{ss:csacyl}
We will now give some facts about closely related conically singular and asymptotically cylindrical manifolds that we will require later.
\subsubsection{Conically singular manifolds}
Heuristically speaking, a conically singular manifold can be thought of as a compact topological space that is a smooth Riemannian manifold away from a point. If the manifold near this point is diffeomorphic to a product $L\times (0,\epsilon)$, and the metric on the manifold is close to the cone metric on $L\times (0,\epsilon)$, then we call the manifold conically singular. This idea is made formal in the following definition, taken from \cite[Defn 3.1]{MR2403190}.
\begin{defn}\label{defn:csman}
 Let $M$ be a connected Hausdorff topological space and let $\hat{x}\in M$. Suppose that $\hat{M}:=M\backslash\{\hat{x}\}$ is a smooth Riemannian manifold with metric $g$. Then we say that $M$ is conically singular at $\hat{x}$ with cone $C$ and rate $\lambda$ if there exist $\epsilon>0$, $\lambda>1$, a closed Riemannian manifold $(L,g_L)$ of dimension one less than $M$, an open set $\hat{x}\in U\subseteq M$ and a diffeomorphism 
\[
 \Psi:(0,\epsilon)\times L\to U\backslash\{\hat{x}\},
\]
such that
\begin{equation}\label{eqn:csmetric}
 |\nabla_C^j(\Psi^*g-g_C)|_{g_C}=O(r^{\lambda-1-j}) \quad \text{for } j\in \mathbb{N} \text{ as } r\to 0,
\end{equation}
where $r$ is the coordinate on $(0,\infty)$ on the cone $C=(0,\infty)\times L$, $g_C=dr^2+r^2g_L$ is the cone metric on $C$ and $\nabla_C$ is the Levi-Civita connection of $g_C$.
\end{defn}
\begin{defn}\label{defn:radfunc}
 Let $M$ be a conically singular manifold at $\hat{x}$ with cone $(0,\infty)\times L$. Use the notation of Definition \ref{defn:csman}. We say that a smooth function $\rho:\hat{M}\to (0,1]$ is a radius function for $M$ if $\rho$ is bounded below by a positive constant on $M\backslash U$, while on $U\backslash\{\hat{x}\}$ there exist constants $0<c<1$ and $C>1$ such that
\[
 cr<\Psi^*\rho<Cr,
\]
on $(0,\epsilon)\times L$.
\end{defn}

We will now define weighted Sobolev spaces for conically singular manifolds. The definition given here may be deduced from \cite[Defn 4.1]{MR879560}.

\begin{defn}\label{defn:weightcs}
 Let $M$ be an $m$-dimensional conically singular manifold at $\hat{x}$ with metric $g$ on $\hat{M}:=M\backslash \{\hat{x}\}$. Let $\rho$ be a radius function for $M$. For a vector bundle $E$ define the weighted Sobolev space $L^p_{k,\mu}(E)$ to be the set of sections $\sigma\in L^p_{k,\text{loc}}(E)$ such that
\begin{equation}\label{eqn:weightsobcs}
 \|\sigma\|_{L^p_{k,\mu}}:=\left(\sum_{j=0}^k \int_{\hat{M}}|\rho^{j-\mu}\nabla^j\sigma|^p \rho^{-m}\,\text{vol}_g\right)^{1/p},
\end{equation}
is finite.
\end{defn}

\subsubsection{Asymptotically cylindrical manifolds}
An asymptotically cylindrical manifold is topologically the same as a conically singular manifold, but metrically they are conformally equivalent. Compare the following definition to Definition \ref{defn:csman}.
\begin{defn}\label{defn:acylman}
Suppose that $(\hat{M},g)$ is a Riemannian manifold. Then we say that $\hat{M}$ is asymptotically cylindrical if there exist $\delta>0$, a closed Riemannian manifold $(L,g_L)$ of dimension one less than $M$, an open set $U\subseteq M$ and a diffeomorphism 
\[
 \Psi:(0,\infty)\times L\to U,
\]
such that for all $j\in \mathbb{N}\cup \{0\}$
\begin{equation}\label{eqn:acmetric}
 |\nabla_\infty^j(\Psi^*g-g_\infty)|_{g_\infty}=O(e^{-\delta t}) \quad \text{ as } t\to \infty,
\end{equation}
where $t$ is the coordinate on $(0,\infty)$ on the cylinder $C=(0,\infty)\times L$, $g_\infty=dt^2+g_L$ is the cylindrical metric on $C$ and $\nabla_\infty$ is the Levi-Civita connection of $g_\infty$.
\end{defn}
Notice that if $(\hat{M},g)$ is a conically singular manifold with radius function $\rho$ then $(\hat{M},\rho^{-2}g)$ is asymptotically cylindrical.

We have the following weighted spaces on an asymptotically cylindrical manifold.
\begin{defn}\label{defn:weightac}
Let $(M,g)$ be an asymptotically cylindrical manifold. For a vector bundle $E$ over $M$, define the weighted Sobolev spaces $W^p_{k,\delta}(E)$ to be the space of sections $\sigma\in L^p_{k,\text{loc}}(E)$ so that
\[
\|\sigma\|_{W^p_{k,\delta}}:=\left(\sum_{j=0}^k\int_M |\rho^{-\delta}\nabla^j\sigma|^p\text {vol}_g\right)^{1/p}<\infty,
\]
where $\rho:M\to (0,1]$ is a smooth function satisfying $c e^{-t}\le \rho(t)\le Ce^{-t}$ on the cylindrical end on $M$ and is equal to one elsewhere.
\end{defn}

We have the following relationship between the weighted spaces $W^p_{k,\delta}$ and $L^p_{k,\mu}$.
\begin{lem}[{\cite[Prop and Defn 4.4]{MR879560}}]\label{lem:weightisom}
 Let $M$ be a conically singular manifold at $\hat{x}$ of dimension $m$ with metric $g$ on $\hat{M}:=M\backslash \{\hat{x}\}$. Let $\rho$ be a radius function for $M$. Let $T^q_s\hat{M}$ be the vector bundle of $(s,q)$-tensors on $\hat{M}$. Denote by $W^p_{k,\delta}(T^q_s\hat{M})$ the weighted space of Definition \ref{defn:weightac} with metric $\rho^{-2}g$ and denote by $L^p_{k,\mu}(T^q_s\hat{M})$ the weighted space of Definition \ref{defn:weightcs}. Then these spaces are isomorphic, with isomorphism given by
\begin{align*}
 L^p_{k,\mu}(T^q_s\hat{M})&\to W^p_{k,\delta}(T^q_s\hat{M}), \\
\sigma&\mapsto \rho^{\delta-\mu+s-q}\sigma.
\end{align*}
\end{lem}
\subsection{Fredholm theory on noncompact manifolds}\label{ss:lmco}

A key part in the argument for proving a result on the moduli space of Cayley deformations of compact manifolds is the observation that an elliptic operator on a compact manifold is Fredholm. Unfortunately, this result fails in general when the underlying manifold is not compact, even in the simplest of settings. However, when the noncompact manifold is topologically a compact manifold with a cylindrical end, a theory was developed for certain types of elliptic operators.
\begin{defn}
Let $X$ be a manifold with a cylindrical end $L\times (0,\infty)$. Let
\[
A:C^\infty_0(E)\to C^\infty_0(F),
\]
be a differential operator on compactly supported smooth sections of vector bundles. We say that $A$ is translation invariant if it is invariant under the natural $\mathbb{R}_+$-action on the cylindrical end $L\times (0,\infty)$ of $X$. If $X$ has an asymptotically cylindrical metric $g$, then we say that an operator
\[
A=\sum_{j=0}^m a_j\cdot \nabla^j,
\]
is asymptotically translation invariant if there exists a translation invariant operator
\[
A_\infty=\sum_{j=0}^m a_j ^\infty \cdot \nabla^j,
\]
and $\delta>0$ such that for all $j=0,\dots, m$ and $k\in \mathbb{N}\cup \{0\}$
\[
|\nabla^k(a_j-a_j^\infty)|_g=O(e^{-\delta t}) \quad \text{ as } t \to \infty,
\]
where $\nabla$ is the Levi-Civita connection of $g$. Here $a_j,a_j^\infty\in C^\infty(E^*\otimes F\otimes (TX)^{\otimes j})$ and `$\cdot$' denotes tensor product followed by contraction.
\end{defn}

The following result may be deduced from the work of Lockhart and McOwen \cite[Thm 6.2]{MR837256} in combination with Lemma \ref{lem:weightisom}.
\begin{prop}\label{prop:lmco}
Let $M$ be a conically singular manifold at $\hat{x}$, $\rho$ a radius function for $M$ and $T^q_s\hat{M}$ be the vector bundle of $(s,q)$-tensors on $\hat{M}:=M\backslash \{\hat{x}\}$. Let
\[
 A:C^\infty_0(T^q_s\hat{M})\to C^\infty_0(T^{q'}_{s'}\hat{M}),
\]
be a linear $m^\text{th}$-order elliptic differential operator with smooth coefficients such that there exists $\lambda\in \mathbb{R}$ so that
\[
\tilde{A}:=\rho^{\lambda+s'-q'} A \rho^{q-s},
\]
is asymptotically translation invariant.
Then
\begin{equation}\label{eqn:acs}
 A:L^p_{k+m,\mu}(T^q_s\hat{M})\to L^p_{k,\mu-\lambda}(T^{q'}_{s'}\hat{M}),
\end{equation}
is a bounded map and there exists a discrete set $\mathcal{D}_{A}\subseteq\mathbb{R}$ such that \eqref{eqn:acs} is Fredholm if, and only if, $\mu\in \mathbb{R}\backslash \mathcal{D}_{A}$.
\end{prop}
We will characterise the set $\mathcal{D}$ for the operators that feature in this article in Section \ref{sec:index}.

\section{Deformations of conically singular Cayley submanifolds}\label{sec:csdefs}

\subsection{Conically singular Cayley submanifolds}

The following definition gives a preferred choice of coordinates around any given point of $X$. This definition is analogous to \cite[Defn 3.6]{MR2053761} and \cite[Defn 3.3]{MR2403190}, which are coordinate systems for almost Calabi--Yau manifolds and $G_2$-manifolds respectively.

\begin{defn}\label{defn:spin7coord}
Let $(X,g,\Phi)$ be a $Spin(7)$-manifold. Then given $x\in X$, there exist $\eta>0$, an open set $x\in V\subseteq X$, $\eta>0$ and a diffeomorphism 
\begin{equation}\label{eqn:spin7coord}
 \chi:B_\eta(0)\to V,
\end{equation}
where $B_\eta(0)$ denotes the ball of radius $\eta$ around zero in $\mathbb{R}^8$, with $\chi(0)=x$ and so that $d\chi|_0:\mathbb{R}^8\to T_xX$ is an isomorphism identifying $(\Phi|_x,g|_x)$ with $(\Phi_0,g_0)$. Call $\chi$ a $Spin(7)$ \emph{coordinate system} for $X$ around $x$.

Call two $Spin(7)$-coordinate systems $\chi,\tilde{\chi}$ for $X$ around $x$ \emph{equivalent} if
\[
 d\chi|_0=d\tilde{\chi}|_0,
\]
as maps $\mathbb{R}^8\to T_xM$.
\end{defn}

In particular, when the $Spin(7)$-manifold $X$ is a four-dimensional Calabi--Yau manifold, we can choose a holomorphic volume form $\Omega$ for $X$ so that $\chi$ is a biholomorphism and $d\chi|_0$ identifies the Ricci-flat K\"ahler form $\omega$ with $\omega_0$ and $\Omega$ with $\Omega_0$, the Euclidean K\"ahler form and holomorphic volume form respectively.

We may now define conically singular submanifolds of $Spin(7)$-manifolds. This definition is again analogous to \cite[Defn 3.6]{MR2053761} and \cite[Defn 3.4]{MR2403190}.
\begin{defn}\label{defn:cs}
 Let $(X,g,\Phi)$ be a $Spin(7)$-manifold and $Y\subseteq X$ compact and connected such that there exists $\hat{x}\in Y$ such that $\hat{Y}:=Y\backslash\{\hat{x}\}$ is a smooth submanifold of $X$. Choose a $Spin(7)$-coordinate system $\chi$ for $X$ around $\hat{x}$. We say that $Y$ is \emph{conically singular} (CS) at $\hat{x}$ with rate $\mu$ and cone $C$ if there exist $1<\mu<2$, $0<\epsilon<\eta$, a compact Riemannian submanifold $(L,g_L)$ of $S^7$ of dimension one less than $Y$, an open set $\hat{x}\in U\subset X$ and a smooth map $\phi:(0,\epsilon)\times L\to B_\eta(0)\subseteq \mathbb{R}^8$ such that $\Psi=\chi\circ\phi:(0,\epsilon)\times L\to U\backslash\{\hat{x}\}$ is a diffeomorphism and $\phi$ satisfies
\begin{equation}
 |\nabla^j(\phi-\iota)|=O(r^{\mu-j})\text{ for }j\in\mathbb{N} \text{ as } r\to 0,
\end{equation}
where $\iota:(0,\infty)\times L\to \mathbb{R}^8$ is the inclusion map given by $\iota(r,l)=rl$, $\nabla$ is the Levi-Civita connection of the cone metric $g_C=dr^2+r^2g_L$ on $C$, and $|\cdot |$ is computed using $g_C$.
\end{defn}
\begin{rem}
 If the smooth, noncompact submanifold $\hat{Y}$ is a Cayley (complex) submanifold of the $Spin(7)$-manifold (Calabi--Yau four-fold) $X$ then we say that $Y$ is a CS Cayley (complex) submanifold of $X$.
\end{rem}
Conically singular submanifolds come with a rate $1<\mu<2$. We must have that $\mu>1$ to guarantee that a conically singular submanifold is a conically singular manifold (in the sense of Definition \ref{defn:csman}). The reason for asking that $\mu<2$ is so that $\mu$ does not depend on the choice of equivalent $Spin(7)$-coordinate system around the singular point of the conically singular submanifold. 
\begin{lem}\label{lem:indepcoor}
 Let $Y$ be a conically singular submanifold at $\hat{x}$ with rate $\mu$ and cone $C$ of a $Spin(7)$-manifold $(X,g,\Phi)$ with $Spin(7)$-coordinate system $\chi$ around $\hat{x}$. Then Definition \ref{defn:cs} is independent of choice of equivalent $Spin(7)$-coordinate system.
\end{lem}
\begin{proof}
Let $\tilde{\chi}$ be another $Spin(7)$-coordinate system for $X$ around $\hat{x}$ equivalent to $\chi$. Then $\chi$ and $\tilde{\chi}$ and their differentials agree at zero. Let $\phi:(0,\epsilon)\times L\to B_\eta(0)$ be the map from Definition \ref{defn:cs}. We will show that $Y$ is conically singular in $X$ with $Spin(7)$-coordinate system $\tilde{\chi}$ around $\hat{x}$. Taking $\tilde\phi:=\tilde\chi^{-1}\circ \chi\circ \phi$, we have that
\begin{align}\label{eqn:equivcoord}
 |\nabla^j(\tilde\phi-\iota)|=|\nabla^j(\tilde\chi^{-1}\circ \chi\circ \phi-\iota)|=|\nabla^j(\phi-\iota)|+O(r^{2-j}),
\end{align}
since $\tilde\chi^{-1}\circ\chi(x)=x+x^TAx+\dots$, and $\phi(r,l)=rl+O(r^{\mu})$. So we see that $Y$ is conically singular at $\hat{x}$ with cone $C$ in $(X,g,\Phi)$ with $Spin(7)$-coordinate system $\tilde{\chi}$, but in order for $Y$ to be CS with rate $\mu$ in this case, Equation \eqref{eqn:equivcoord} tells us that we must have that $\mu<2$.  
\end{proof}
The following definition is independent of choice of equivalent $Spin(7)$-coordinate system. It is analogous to \cite[Defn 3.5]{MR2403190}.
\begin{defn}\label{defn:tcone}
Let $Y$ be a conically singular submanifold at $\hat{x}$ with rate $\mu$ and cone $C$ of a $Spin(7)$-manifold $(X,g,\Phi)$ with $Spin(7)$-coordinate system $\chi$. Denote by $\zeta:=d\chi|_0:T_0\mathbb{R}^8\to T_{\hat{x}}X$. Define the \emph{tangent cone} of $Y$ at $\hat{x}$ to be 
\[
\hat{C}:=\zeta\circ\iota(C)\subseteq T_{\hat{x}}X,
\]
where $\iota:C\to\mathbb{R}^8$ is the inclusion map given in Definition \ref{defn:cs}.
\end{defn}

On a Calabi--Yau manifold $M$ we are given a Ricci-flat metric $\omega$ that we often have no explicit expression for. The following lemma tells us that Definition \ref{defn:cs} is independent of choice of K\"ahler metric on $M$.
\begin{lem}\label{lem:indepom}
 Let $M$ be a Calabi--Yau four-fold with Ricci-flat K\"ahler form $\omega$ and let $N$ be a CS submanifold of $M$ as in Definition \ref{defn:cs}. Then if $\omega'$ is any other K\"ahler form on $M$ then $N$ is still a conically singular submanifold of $M$ with the same rate and tangent cone.
\end{lem}
\begin{proof}
Suppose that $N$ is a CS submanifold of $M$ with respect to $\omega$ at $\hat{x}$. Choose a $Spin(7)$-coordinate system for $M$ around $\hat{x}$,
\[
 \chi:B_\eta(0)\to V\backslash\{\hat{x}\},
\]
for some $\eta>0$ and open $V\subseteq M$ containing $\hat{x}$, so that $\chi(0)=\hat{x}$ and
\[
 \chi^*\omega=\omega_0+O(|z|^2),
\]
where $\omega_0$ is the standard Euclidean K\"ahler form on $\mathbb{C}^4$. Let $\phi$, $\epsilon$, $C=(0,\infty)\times L$, $\iota$ and $\mu$ be as in Definition \ref{defn:cs}.

Now given any other K\"ahler form $\omega'$ on $M$, we can find by \cite[pg 107]{MR1288523} $\eta'>0$, an open set $x\in V'\subseteq M$ and a diffeomorphism
\[
 \chi':B_{\eta'}(0)\to V'\backslash\{x\},
\]
with $\chi'(0)=x$ and
\[
 \chi'^*\omega'=\omega_0+O(|z|^2).
\]
Since $\chi$ and $\chi'$ are diffeomorphisms, $d\chi|_0$ and $d\chi'|_0$ are isomorphisms $\mathbb{C}^4\to T_{\hat{x}}M$. Then $A:=(d\chi'|_0)^{-1}\circ d\chi|_0$ is an invertible linear map $\mathbb{C}^4\to \mathbb{C}^4$.  We will show that $N$ is conically singular in $(M,\omega')$ (taking $\chi'$ to be the coordinate system) with cone $C'=A\iota(C)$ and rate $\mu$.

Firstly note that since $A$ is a linear map, $C'=A\iota(C)=\{Av \,|\, v\in \iota(C)\}$ is also a cone. Denote by $L'$ the link of $C'$ (considered as a Riemannian submanifold of $S^7$), and for any $\epsilon'>0$ write $\iota':L'\times (0,\epsilon')\to \mathbb{C}^4$ for the inclusion map $(r',l')\mapsto r'l'$.

Define $\phi':(0,\epsilon')\times L'\to \mathbb{C}^4$ by $\phi'=\chi'^{-1}\circ \chi\circ\phi\circ A^{-1}$, where $\epsilon'=\epsilon\|A\|$. Then this map is well defined (taking $\epsilon'$ smaller if necessary) and moreover $\chi'\circ\phi'$ is a diffeomorphism onto its image. Moreover, by a similar argument to Lemma \ref{lem:indepcoor} we have that
\[
|\nabla^j(\phi'(r',l')-\iota(r',l'))|=O(r'^\mu),
\]
since $\mu<2$.

Finally, we have that
\[
\hat{C'}=d\chi'|_0(\iota'(C'))=d\chi|_0\circ (d\chi|_0)^{-1}\circ d\chi'|_0(A\iota(C))=d\chi|_0(A^{-1}A\iota(C))=\hat{C},
\]
and so the tangent cone to $N$ at $\hat{x}$ is the same with respect to each metric.
\end{proof}
\begin{rem}
Note that the proof Lemma \ref{lem:indepom} also shows that if $N$ is conically singular with respect to one $Spin(7)$-coordinate system, it is conically singular with respect to any other $Spin(7)$-coordinate system, although with a different cone in general, but the same tangent cone.
\end{rem}
We can now construct an example of a conically singular complex surface inside a Calabi--Yau four-fold.
\begin{ex}
We will model our conically singular complex surface on the following complex cone in $\mathbb{C}^4$. Define $C$ to be the set of $(z_1,z_2,z_3,z_4)\in \mathbb{C}^4$ satisfying
\begin{align*}
 z_1^4+z_2^4+z_3^4+z_4^4&=0, \\
z_1^3+z_2^3+z_3^3+z_4^3&=0.
\end{align*}
Clearly, if $z\in C$, then also $\lambda z\in C$ for any $\lambda\in\mathbb{R}\backslash\{0\}$, and so $C$ is a cone.

Checking the rank of the matrix
\[
 \begin{pmatrix}
  4z_1^3 & 4z_2^3 & 4z_3^3 & 4z_4^3 \\
3z_1^2 & 3z_2^2 & 3z_3^2 & 3z_4^2
 \end{pmatrix},
\]
at each point of $C$, we see that the only singular point of $C$ is zero.

As we will discuss in more detail in Section \ref{sec:index}, a complex cone $C$ in $\mathbb{C}^4$ has both a real link $L:=S^7\cap C$, and a complex link $\Sigma:=\pi(L)$, where $\pi:S^7\to \mathbb{C}P^3$ is the Hopf fibration. We can view the real link of a complex cone as a circle bundle over the complex link of the cone.

In this case, the complex link $\Sigma$ of $C$ is the Riemannian surface in $\mathbb{C}P^3$ is given by $[z_0:z_1:z_2:z_3]\in\mathbb{C}P^3$ satisfying
\begin{align*}
 z_0^4+z_1^4+z_2^4+z_3^4&=0, \\
z_0^3+z_1^3+z_2^3+z_3^3&=0.
\end{align*}
We can apply the adjunction formula to find that the canonical bundle of $\Sigma$ is given by
\[
 K_\Sigma=K_{\mathbb{C}P^3}|_\Sigma\otimes \mathcal{O}_{\mathbb{C}P^3}(4)|_{\Sigma}\otimes \mathcal{O}_{\mathbb{C}P^3}(3)|_\Sigma=\mathcal{O}_{\mathbb{C}P^3}(4+3-3-1)|_\Sigma=\mathcal{O}_{\mathbb{C}P^3}(3)|_\Sigma,
\]
where $\mathcal{O}_{\mathbb{C}P^3}(k)$ denotes the $-k^\text{th}$ (tensor) power of the tautological line bundle over $\mathbb{C}P^3$ if $k$ is a negative integer, the $k^\text{th}$ power of the dual of the tautological line bundle if $k$ is a positive integer, and the trivial line bundle if $k=0$.
Then it follows from the Hirzebruch--Riemann--Roch theorem \cite[Thm 5.1.1]{MR2093043} that the genus of $\Sigma$ is
\[
 g=\frac{2+\text{deg }\mathcal{O}_{\mathbb{C}P^3}(3)|_\Sigma}{2}=\frac{2+3\times \text{deg}(\Sigma)}{2}=(2+3\times 4\times 3)/2=19.
\]

Now consider the Calabi--Yau four-fold $M$ defined by 
\[
\{[z_0:z_1:z_2:z_3:z_4:z_5]\in \mathbb{C}P^5\, |\,z_0^6+z_1^6+z_2^6+z_3^6+z_4^6+z_5^6=0\}.
\]
Consider the singular submanifold $N$ of $M$ defined to be
the set of all $[z_0:z_1:z_2:z_3:z_4:z_5]\in \mathbb{C}P^5$ satisfying
\begin{align*}
 z_0^6+z_1^6+z_2^6+z_3^6+z_4^6+z_5^6&=0, \\
z_1^4+z_2^4+z_3^4+z_4^4 &=0, \\
z_1^3+z_2^3+z_3^3+z_4^3 &=0.
\end{align*}
The complex Jacobian matrix of the defining equations of $N$ is given by
\[
 \begin{pmatrix}
6z_0^5 & 6z_1^5 & 6z_2^5 & 6z_3^5 & 6z_4^5 & 6z_5^5 \\
  0 &  4z_1^3 & 4z_2^3 & 4z_3^3 & 4z_4^3 & 0 \\
0 & 3z_1^2 & 3z_2^2 & 3z_3^2 & 3z_4^2 & 0
 \end{pmatrix}. 
\]
It can be calculated that there are six singular points on $N$ of the form $[\omega:0:0:0:0:1]$, where $\omega$ is a $6^\text{th}$ root of $-1$.

We will now prove that $N$ satisfies Definition \ref{defn:cs}. We will exploit Lemma \ref{lem:indepom} and check the definition using the metric on $M$ induced from the Fubini--Study metric on $\mathbb{C}P^5$, denoted by $\omega$.

Denote the singular points of $N$ by $\{p_1,\dots,p_6\}$, where $p_k=[\omega_k:0:0:0:0:1]$ for $\omega_k:= e^{i(2k-1)\pi/6}$. We must construct maps $\chi_k$ so that there exist $\eta_k>0$ and open sets $p_k\in V_k\subseteq M$ and diffeomorphisms
\[
 \chi_k:B_{\eta_k}(0)\to V_k,
\]
with $\chi_k(0)=p_k$ and so that
\[
 \chi_k^*\omega=\omega_0+O(|z|^2),
\]
for $k=1,\dots, 6$.

For $k=1,\dots 6$, define $\chi_k:B_{\eta_k}(0)\to M$ by 
\begin{align}\nonumber
\chi_k(&w_1,w_2,w_3,w_4)= \\ \label{eqn:maps}
[&\omega_k:\sqrt{2}w_1:\sqrt{2}w_2:\sqrt{2}w_3:\sqrt{2}w_4:(1-8(w_1^6+w_2^6+w_3^6+w_4^6))^{1/6}],
\end{align}
where if $a=re^{i\theta}$ for $r>0$ and $-\pi<\theta\le \pi$, we define $a^{1/6}:=r^{1/6}e^{i\theta/6}$. It is clear that \eqref{eqn:maps} is a diffeomorphism onto its image. The induced Fubini--Study metric on $M$ pulls back under $\chi_k$ to the Euclidean metric on $\mathbb{C}^4$ at each $p_k=[\omega_k:0:0:0:0:1]$. Taking $\phi=\iota$, where $\iota:C\to \mathbb{C}^4$ is the inclusion map, we see that $\phi\circ\chi$ is a diffeomorphism $C$ to $N$, and so the definition of conically singular is trivially satisfied.
\end{ex}

\subsection{Tubular neighbourhood theorems}

In this section we will prove a tubular neighbourhood theorem for conically singular submanifolds so that we can identify deformations of conically singular submanifolds with normal vector fields. We will do this in two steps. Firstly, in Proposition \ref{prop:tubnbhdcone} we will construct a tubular neighbourhood of a cone in $\mathbb{R}^n$ using the well-known tubular neighbourhood theorem for compact submanifolds. We will use this to construct a tubular neighbourhood of a conically singular submanifold in Proposition \ref{prop:cstubnbhd}. Propositions \ref{prop:tubnbhdcone} and \ref{prop:cstubnbhd} use ideas of similar results proved by Joyce \cite[Thm 4.6]{MR2053761} for special Lagrangian cones and Lotay \cite[Prop 6.4]{MR2403190} for CS coassociative submanifolds.

\begin{prop}[Tubular neighbourhood theorem for cones]\label{prop:tubnbhdcone}
 Let $C$ be a cone in $\mathbb{R}^n$ with link $L$ and let $g$ be a Riemannian metric on $\mathbb{R}^n$ (not necessarily the Euclidean metric). There exists an action of $\mathbb{R}_+$ on $\nu_{\mathbb{R}^n}(C)$ 
\[
 t:\nu_{\mathbb{R}^n}(C)\to \nu_{\mathbb{R}^n}(C),
\]
so that
\begin{equation}\label{eqn:tact}
 |t\cdot v |=t| v |.
\end{equation}
We can construct open sets $V_C\subseteq \nu_{\mathbb{R}^n}(C)$, invariant under \eqref{eqn:tact}, containing the zero section and $T_C\subseteq \mathbb{R}^n$, invariant under multiplication by positive scalars, containing $C$ that grow like $r$ and a dilation equivariant diffeomorphism
\[
 \Xi_C:V_C\to T_C,
\]
in the sense that $\Xi_C(t\cdot v)=t\,\Xi_C(v)$ for all $v\in \nu_{\mathbb{R}^n}(C)$. Moreover, $\Xi_C$ maps the zero section of $\nu_{\mathbb{R}^n}(C)$ to $C$.
\end{prop}
\begin{proof}
We will first address the claim that there exists an $\mathbb{R}_+$-action on $\nu_{\mathbb{R}^n}(C)$ so that \eqref{eqn:tact} holds. First note that points in $\nu_{\mathbb{R}^n}(C)$ take the form
\[
(r,l,v(r,l)),
\]
where $r\in \mathbb{R}_+$, $l\in L$ and $v(r,l)\in \nu_{r,l}(C)$. Since any finite-dimensional inner product spaces of the same dimension are isometric, given any $r'\in \mathbb{R}_+$ we can think of $(r',l,v(r,l))$ as another point in $\nu_{\mathbb{R}^n}(C)$ with $|v(r,l)|_{r,l}=|v(r,l)|_{r',l}$, where $|\cdot|_{r,l}$ denotes the norm on $\nu_{r,l}(C)$ induced from $g$. Define an action of $\mathbb{R}_+$ on $\nu_{\mathbb{R}}(C)$ by
\begin{align}\nonumber
t:\nu_{\mathbb{R}^n}(C) & \to \nu_{\mathbb{R}^n}(C), \\ \label{eqn:tactpr}
(r,l,v(r,l))&\mapsto (tr,l,tv(r,l)).
\end{align}
Then $|t\cdot v(r,l)|_{tr,l}=|t v(r,l)|_{r,l}=t|v(r,l)|_{r,l}$ as claimed. Notice that $t\cdot (t'\cdot v)=(tt')\cdot v$ and so \eqref{eqn:tactpr} is a group action in the usual sense.

To prove the tubular neighbourhood part of this proposition, we first apply the usual tubular neighbourhood theorem to the compact submanifold $L$ of $S^{n-1}$. (Recall that we need a metric on $S^{n-1}$ to define the exponential map. We take this to be the standard round metric on $S^{n-1}$.) This gives us an open set $V_L\subseteq \nu_{S^{n-1}}(L)$ containing the zero section and an open set $T_L\subseteq S^7$ containing $L$ and a diffeomorphism
\[
 \Xi_L:V_L\to T_L,
\]
so that $\Xi_L$ maps the zero section of $\nu_{S^{n-1}}(L)$ to $L$. Again write points in $\nu_{\mathbb{R}^n}(C)$ as $(r,l,v(r,l))$, where $v\in \nu_{r,l}(C)$, and similarly points in $\nu_{S^{n-1}}(L)$ as $(l,v(l))$ where $v\in \nu_{l}(L)\cong \nu_{r,l}(C)$. Then define
\[ 
V_C:=\left\{(r,l,v(r,l))\in \nu_{\mathbb{R}^n}(C) \,\left|\right. \, \left(l,r^{-1}v(r,l)\right)\in V_L\right\}.
\]
It is clear that $V_C$ is invariant under the $\mathbb{R}_+$-action \eqref{eqn:tactpr} by construction of $V_C$ and the $\mathbb{R}_+$-action. We see that $V_C$ grows like $r$ in the sense that if $v=(r,l,v(r,l))\in V_C$ then
\[
|v(r,l)|_{rl}\le r |V|,
\]
where $V$ is the diameter of the set $V$.
Now define 
\[
T_C:=\{\lambda t \, |\, t\in T_L,\, \lambda\in \mathbb{R}_+\}. 
\]
Then it is clear that $T_C$ is dilation invariant, in the sense that it is clearly invariant under multiplication by positive scalars, and that $C\subseteq T_C$. We see that $T_C$ grows like $r$ in the sense that if $t\in T$, $l\in L$ and $r\in \mathbb{R}_+$ then
\[
|rt-rl|\le r |T|,
\]
where $|T|$ is the diameter of the set $T$.
Define 
\begin{align*}
 \Xi_C:V_C&\to T_C, \\
(r,l,v(r,l))&\mapsto r\,\Xi_L(l,r^{-1} v(r,l)).
\end{align*}
It is clear that $\Xi_C$ is well-defined, bijective and smooth. It is also clear that
\[
 \Xi_C(t\cdot(r,l,v(r,l))=t\, \Xi_C(r,l,v(r,l)).
\]
Finally we have that
\[
 \Xi_C(r,l,0)=r\, \Xi_L(l,0)=rl,
\]
by definition of $\Xi_L$ and so $\Xi_C$ maps the zero section of $\nu_{\mathbb{R}^n}(C)$ to $C$.
\end{proof}

We can use this result to prove a tubular neighbourhood theorem for a conically singular submanifold.
\begin{prop}\label{prop:cstubnbhd}
 Let $Y$ be a conically singular submanifold of $X$ at $\hat{x}$ with cone $C$ and rate $\mu$. Write $\hat{Y}:=Y\backslash \{\hat{x}\}$. Then there exist open sets $\hat{V}\subseteq \nu_X(\hat{Y})$ containing the zero section and $\hat{T}\subseteq X$ containing $\hat{Y}$ and a diffeomorphism
\[
 \hat{\Xi}:\hat{V}\to \hat{T},
\]
that takes the zero section of $\nu_X(\hat{Y})$ to $\hat{Y}$. Moreover, we can choose $\hat{V}$ and $\hat{T}$ to grow like $\rho$ as $\rho\to 0$.
\end{prop}
\begin{proof}
Notice that $K:=Y\backslash U$ is a compact submanifold of $X$. So by the compact tubular neighbourhood theorem we can find open sets $\hat{V}_1\subseteq \nu_X(K)$ containing the zero section and $\hat{T}_1\subseteq X$ containing $K$ and a diffeomorphism
\[
 \hat{\Xi}_1:\hat{V}_1\to\hat{T}_1.
\]
We will construct a tubular neighbourhood for $\hat{Y}$ near $\hat{x}$. Denote $C_\epsilon:=C\cap B_\epsilon(0)$. Use the notation of Definition \ref{defn:cs}. Choose $\phi:C_\epsilon\to \mathbb{R}^n$ uniquely by asking that 
\[
 \phi(r,l)-\iota(r,l)\in (T_{rl}\iota(C))^\perp.
\]
Then since
\[
 |\phi-\iota|=O(r^\mu),
\]
for $1<\mu<2$ as $r\to 0$, making $\epsilon$ smaller if necessarily, we can guarantee that $\phi(r,l)$ lies in the tubular neighbourhood of $C$ given by Proposition \ref{prop:tubnbhdcone}. We can therefore identify $\phi(C_\epsilon)$ with a normal vector field $v_\phi$ on $C$.

Applying Proposition \ref{prop:tubnbhdcone} gives us $V_C\subseteq \nu_{\mathbb{R}^n}(C)$, $T_C\subseteq \mathbb{R}^n$ and a diffeomorphism
\[
 \Xi_C:V_C\to T_C.
\]
Denote by $V_{C_\epsilon}$ the restriction of $V_C$ to $C_\epsilon$, and define
\[
 V_\phi:=\{v\in \nu_{B_\epsilon(0)}(C_\epsilon) \,|\, v+v_\phi\in V_{C_\epsilon}\},
\]
with
\[
 \Xi_\phi(v):=\Xi_C(v+v_\phi),
\]
for $v\in V_\phi$ and
\[
 T_\phi:=\Xi_C(V_\phi).
\]
Then $\Xi_C:V_\phi\to T_\phi$ is a diffeomorphism by construction.

Write $\hat{U}:=U\backslash \{\hat{x}\}$. Define $\hat{V}_2:=F(V_\phi)\subseteq \nu_X(\hat{U})$, where $F$ is the isomorphism $\nu_{B_\epsilon(0)}(C_\epsilon)\to \nu_X(\hat{U})$ induced from $\Psi$ and $\iota$ and $\hat{T}_2:=\chi(T_\phi)$. By definition, these sets grow with order $\rho$ as $\rho\to 0$. Then
\[
\chi\circ \Xi_\phi\circ F^{-1}: \hat{V}_2\to \hat{T}_2,
\]
is a diffeomorphism taking the zero section of $\nu_X(\hat{U})$ to $\hat{U}$. Define $\hat{V}$, $\hat{T}$ and $\hat{\Xi}$ by interpolating smoothly between $\hat{V}_1$ and $\hat{V}_2$, $\hat{T}_1$ and $\hat{T}_2$ and $\hat{\Xi}_1$ and $\hat{\Xi}_2$.
\end{proof}

\subsection{Deformation problem}
The moduli space that we will consider will be defined in Definition \ref{defn:cscaymodsp} below, and this moduli space will be identified with the kernel of a nonlinear partial differential operator in Proposition \ref{prop:cscayop}. First we will define a weighted norm on spaces of differentiable sections of a vector bundle.

\subsubsection{Weighted norms on spaces of differentiable sections}

Let $X$ will be an $n$-dimensional CS manifold with a radius function $\rho$, $E$ a vector bundle over $\hat{X}$ (the nonsingular part of $X$) with a metric and connection.
\begin{defn}\label{defn:weightck}
 Let $\lambda\in\mathbb{R}$ and $k\in\mathbb{N}$. Define the space $C^k_\lambda(E)$ to be the space of sections $\sigma \in C^k_\text{loc}(E)$ satisfying
\[
 \|\sigma\|_{C^k_\lambda}:=\sum_{j=0}^k\sup_{\hat X}|\rho^{j-\lambda}\nabla^j \sigma|<\infty.
\]
We say that $\sigma\in C^\infty_\lambda(E)$ if $\sigma\in C^k_\lambda(E)$ for all $k\in\mathbb{N}$.
\end{defn}
The space $C^k_\lambda(E)$ is a Banach space, but $C^\infty_\lambda(E)$ is not in general.

\subsubsection{Moduli space}

We will now formally define the moduli space of conically singular Cayley deformations of a Cayley submanifold that we will be studying in this article.
\begin{defn}\label{defn:cscaymodsp}
Let $Y$ be a conically singular Cayley submanifold at $\hat{x}$ with cone $C$ and rate $\mu$ of a $Spin(7)$-manifold $(X,g,\Phi)$ with respect to some $Spin(7)$-coordinate system $\chi$, and denote the tangent cone of $Y$ at $\hat{x}$ by $\hat{C}$. Write $\hat{Y}:=Y\backslash \{\hat{x}\}$. Define the \emph{moduli space of conically singular (CS) Cayley deformations} of $Y$ in $X$, $\hat{\mathcal{M}}_{\mu}(Y)$, to be the set of CS Cayley submanifolds $Y'$ at $\hat{x}$ with cone $C$, rate $\mu$ and tangent cone $\hat{C}$ of $X$ so that there exists a continuous family of topological embeddings $\iota_t:Y\to X$ with $\iota_0(Y)=Y$ and $\iota_1(Y)=Y'$, so that $\iota_t(\hat{x})=\hat{x}$ for all $t\in [0,1]$ and so that $\hat{\iota}_t:=\iota_t|_{\hat{Y}}$ is a smooth family of embeddings $\hat{Y}\to X$ with $\hat{\iota}_0(\hat{Y})=\hat{Y}$ and $\hat{\iota}_1({\hat{Y}})=\hat{Y}':=Y'\backslash\{\hat{x}\}$.
\end{defn}
We will now end this section by identifying the moduli space of Cayley CS deformations of a CS Cayley submanifold of a $Spin(7)$-manifold with the kernel of a nonlinear partial differential operator.
\begin{prop}\label{prop:cscayop}
Let $Y$ be a CS Cayley submanifold at $\hat{x}$ with cone $C$ and rate $\mu\in(1,2)$ of a $Spin(7)$-manifold $(X,g,\Phi)$. Let $\tau$ be the $\Lambda^2_7$-valued four-form defined in Proposition \ref{prop:tau}, $\pi:\Lambda^2_7\to E$ be the projection map for the splitting given in \eqref{eqn:l27decomp} and $\hat{V}\subseteq \nu_X(\hat{Y})$, $\hat{T}\subseteq{M}$ and $\hat{\Xi}$ be the open sets and diffeomorphism from the CS tubular neighbourhood theorem \ref{prop:cstubnbhd}. For $v\in C^\infty(\nu_X(\hat{Y}))$ taking values in $\hat{V}$ write $\hat{\Xi}_v$ for the diffeomorphism $\hat{\Xi}\circ v: \hat{Y}\to \hat{Y}_v:=\hat{\Xi}_v(\hat{Y})$.

Then we can identify the moduli space of CS Cayley deformations of $Y$ in $X$ near $Y$ with the kernel of the following differential operator
\begin{align}\nonumber
\hat{F}:C^\infty_\mu(\hat{V}) &\to C^\infty_\textnormal{loc}(E), \\ \label{eqn:cscaypde}
v&\mapsto \pi( *_{\hat{Y}}\,\hat{\Xi}_v^*(\tau|_{\hat{Y}_v})).
\end{align}
\end{prop}
\begin{proof}
The deformation $\hat{Y}_v$ is Cayley if, and only if, $\tau|_{\hat{Y}_v}\equiv 0$, which since $\hat{\Xi}_v$ is a diffeomorphism is equivalent to $v\in \text{Ker }\hat{F}$, and since $v,\tau, \hat{\Xi}_v$ are all smooth, we see that $\hat{F}$ takes values in $C^\infty_\text{loc}(E)$ at claimed.

It remains to show that $Y_v:=\hat{Y}_v\cup \{\hat{x}\}$ is a CS submanifold of $X$ at $\hat{x}$ with cone $C$ and rate $\mu$ (with respect to the same $Spin(7)$-coordinate system as $Y$) if, and only if, $v\in C^\infty_\mu(\hat{V})$.

Let $v$ be a smooth normal vector field on $\hat{Y}$, and let $\hat{Y}_v:=\hat{\Xi}_v(\hat{Y})$. Use the notation of Definition \ref{defn:cs}. Choose $\phi:(0,\epsilon)\times L\to B_\epsilon(0)$ uniquely by requiring that
\[
\phi(r,l)-\iota(r,l)\in (T_{rl}\iota(C))^\perp.
\]
Now we can use $\Psi$ and $\iota$ to identify $\nu_{X}(\hat{U})$ with $\nu_{B_\epsilon(0)}(\iota(C_\epsilon))$, where $\hat{U}:=U\backslash\{\hat{x}\}$ and $C_\epsilon:=(0,\epsilon)\times L$. Write $v_C$ for the section of $\nu_{B_\epsilon(0)}(\iota(C_\epsilon))$ corresponding to $v$ under this identification.

Making $\epsilon$ and $U$ smaller if necessary, by the definition of the tubular neighbourhood map in Proposition \ref{prop:cstubnbhd}, we can define a map $\phi_v:C_\epsilon\to B_\epsilon(0)$ by
\[
\phi_v(r,l)=\Xi_\phi(v_C(r,l)),
\]
where $\Xi_\phi$ was defined in the proof of Proposition \ref{prop:cstubnbhd}, so that $\chi\circ\phi_v:C_\epsilon\to \Xi_v(\hat{U})\subseteq \hat{Y}_v$ is a diffeomorphism. So we see that for $Y_v$ to be a CS submanifold of $X$ with rate $\mu$ and cone $C$ we must have that
\begin{equation}\label{eqn:defcs}
|\nabla^j(\phi_v(r,l)-\iota(r,l))|=O(r^{\mu-j}),
\end{equation}
for all $j\in \mathbb{N}$ as $r\to 0$. Now we can write
\[
|\nabla^j(\phi_v-\iota)|\le |\nabla^j(\phi_v-\phi)|+|\nabla^j(\phi-\iota)|,
\]
and so \eqref{eqn:defcs} holds if, and only if,
\[
 |\nabla^j(\phi_v-\phi)|=O(r^{\mu-j}),
\]
for $j\in \mathbb{N}$ as $r\to 0$. But examining the definition of $\phi_v$, we see that we can identify $\phi_v-\phi$ with the graph of $v_C$, and so \eqref{eqn:defcs}  holds if, and only if,
\[
|\nabla^j v_C|=O(r^{\mu-j}),
\]
for $j\in \mathbb{N}$ as $r\to 0$. But then by definition of $v_C$ this is equivalent to 
\[
|\nabla^j v|=O(\rho^{\mu-j}),
\]
for $j \in \mathbb{N}$ as $\rho\to 0$, that is, $v\in C^j_\mu(\hat{V})$ for all $j\in \mathbb{N}$. So we see that the moduli space of CS Cayley deformations of $Y$ in $X$ can be identified with the kernel of \eqref{eqn:cscaypde}.
\end{proof}
\subsection{Cayley deformations of a CS Cayley submanifold}\label{sec:cscay}

In this section we prove Theorem \ref{thm:cscaydefs} on the expected dimension of the moduli space of CS Cayley deformations of a conically singular Cayley submanifold $Y$ in a $Spin(7)$-manifold $X$.

The following lemma is similar to \cite[Thm 5.1]{MR2053761} and \cite[Prop 6.9]{MR2403190}.

\begin{lem}\label{lem:cscayweight}
Let $Y$ be a conically singular Cayley submanifold of a $Spin(7)$-manifold $X$. Let $\hat{F}$ be the operator defined in Proposition \ref{prop:cscayop}. Then we can write
\begin{equation}\label{eqn:qdefn}
 \hat{F}(v)(x)=Dv(x)+\hat{Q}(x,v(x),\nabla v(x)),
\end{equation}
for $x\in\hat{Y}$, where
\[
 \hat{Q}:\{(x,y,z) \,|\, (x,y)\in \hat{V}, z\in \nu_x(\hat{Y})\otimes T^*_x\hat{Y}\}\to E,
\]
is smooth, $D$ was defined in Proposition \ref{prop:caylin} and 
\[
\hat{Q}(v)(x):=\hat{Q}(x,v(x),\nabla v(x)),
\]
is a section of $E$. Let $\mu>1$. Then for each $k\in\mathbb{N}$, for $v\in C^{k+1}_\mu(\hat{V})$ with $\|v\|_{C^1_1}$ sufficiently small, there exist constants $C_k>0$ so that
\begin{equation}\label{eqn:qest}
 \|\hat{Q}(v)\|_{C^{k}_{2\mu-2}}\le C_k\|v\|_{C^{k+1}_\mu}^2,
\end{equation}
and if $v\in L^p_{k+1,\mu}(\hat{V})$ with $\|v\|_{C^1_1}$ sufficiently small, with $k>1+4/p$, there exist constants $D_k>0$ such that
\begin{equation}\label{eqn:qestpk}
\|\hat{Q}(v)\|_{p,k,2\mu-2}\le D_k\|v\|_{p,k+1,\mu}^2.
\end{equation}
Moreover, we may deduce that
\begin{equation}\label{eqn:cscaysob}
 \hat{F}:L^p_{k+1,\mu}(\hat{V})\to L^p_{k,\mu-1}(E),
\end{equation}
is a smooth map of Banach spaces for any $1<p<\infty$ and $k\in \mathbb{N}$ with $k>1+4/p$.
\end{lem}
\begin{proof}
 By definition of conically singular, we can split $Y$ into a compact piece $K$, where we can argue as in the proof of an analogous result for compact Cayley submanifolds (see, for example, \cite[Lem 3.4]{compactcay}) that the estimate \eqref{eqn:qest} holds, and a piece diffeomorphic to a cone, which is where we must check how $\hat{F}$ behaves as $\rho\to 0$, where $\rho$ is a radius function for $\hat{Y}$. Making the compact piece slightly larger, using the definition of $\hat{F}$, we may estimate $\hat{F}$ by estimating $\hat{F}_C$, the operator on the cone defined by
\[
 \hat{F}_C(v+v_\phi)(r,l)=\hat{F}(v)(\Psi(r,l)),
\]
where $v_\phi$ is the normal vector field on $C$ that describes $\phi(C)$ as described in the proof of Proposition \ref{prop:cstubnbhd}, where we are using the notation of Definition \ref{defn:cs}. 
Define $\hat{Q}_C$ analogously by
\[
 \hat{F}_C(v+v_\phi)(r,l)=D(v+v_\phi)(r,l)+\hat{Q}_C(r,l,(v+v_\phi)(r,l),\nabla(v+v_\phi)(r,l)).
\]
By definition of $\hat{Q}$ and $\hat{Q}_C$, we see that
\begin{equation}\label{eqn:hatqc}
 \hat{Q}(r,l,v,\nabla v)=\hat{Q}_C(r,l,v+v_\phi,\nabla(v+v_\phi))-\hat{Q}_C(r,l,v_\phi,\nabla v_\phi),
\end{equation}
and so to estimate $\hat{Q}$ and its derivatives, it suffices to estimate the right hand side of Equation \eqref{eqn:hatqc}. Notice that since for each $(r,l)\in C$ we can think of $\hat{Q}_C$ as a map 
\[
 \nu_{r,l}(C)\times \nu_{r,l}(C)\otimes T^*_{r,l}C \to E_{r,l},
\]
and so we can make sense of a Taylor expansion of $\hat{Q}$ around points of the form $(r,l,y,z)$, for $y\in \nu_{r,l}(C)$ and $z\in \nu_{r,l}(C)\otimes T^*_{r,l}C$. Abusing notation slightly, write
\[
 \frac{\partial \hat{Q}_C}{\partial y}(r,l,y,z),
\]
for derivative of $\hat{Q}_C$ in the $y$ direction at $(r,l,y,z)$, and adopt similar notation for the derivative in the $z$ direction and the higher derivatives. Then we have that
\begin{align} \nonumber
\hat{Q}_C(r,l,y+y_0,z+z_0)&=\hat{Q}_C(r,l,y_0,z_0)+\frac{\partial \hat{Q}_C}{\partial y}(r,l,y_0,z_0)y \\ \nonumber +\frac{\partial \hat{Q}_C}{\partial z}(r,l,y_0,z_0)z
&+\frac{1}{2}\frac{\partial^2 \hat{Q}_C}{\partial y^2}(r,l,ty+y_0,tz+z_0)(y,y) \\ \nonumber
&+\frac{\partial^2 \hat{Q}_C}{\partial y\partial z}(r,l,ty+y_0,tz+z_0)(y,z) \\ \label{eqn:hatqexp}
&+\frac{1}{2}\frac{\partial^2 \hat{Q}_C}{\partial z^2}(r,l,ty+y_0,tz+z_0)(z,z),
\end{align}
for some $t\in [0,1]$. We would like to estimate the derivatives of $\hat{Q}_C$. Since $\hat{Q}$ is smooth in all of its variables, this is possible as long as we restrict the domain of $\hat{Q}_C$ to a compact set. However, we are working on a cone (with its singular point removed) so this isn't possible. We may, however, fix $r=r_0$, for some $r_0\in (0,\epsilon)$, perform our estimates, and use the definition of $\hat{F}_C$ and $\hat{Q}_C$ to study the behaviour of the estimates we find as we let $r$ vary. Recall the action of $\mathbb{R}_+$ on $\nu_{\mathbb{R}^8}(C)$ that was defined in the proof of Proposition \ref{prop:tubnbhdcone}, and indeed the tubular neighbourhood map that we constructed in this proof, which forms part of the operator $\hat{F}_C$ that we are currently studying. By construction, we can see that
\[
 |\hat{F}_C(v(r,l))|_r=\left|\hat{F}_C\left(\frac{r_0}{r}\cdot v(r,l)\right)\right|_{r_0},
\]
where $|\cdot|_r$ means that we are taking the norm at the point $r$. We may deduce that
\begin{align*}
& |\nabla^k\hat{Q}_C(r,lv+v_\phi,\nabla(v+v_\phi))|_r \\
 =r^{-k}&\left|\hat{Q}_C\left(r_0,l,\frac{r_0}{r}\cdot(v+v_\phi),\nabla\frac{r_0}{r}\cdot(v+v_\phi)\right)\right|_{r_0}.
\end{align*}
We also have that by construction
\[
 \left|\nabla^k\frac{r_0}{r}\cdot v(r,l)\right|_{r_0}=\left(\frac{r}{r_0}\right)^{k-1}|\nabla^kv(r,l)|_{r}.
\]
Since $\hat{Q}_C$ has no linear parts, we have that by equation \eqref{eqn:hatqexp},
\begin{align}
\nonumber
&\hat{Q}_C(r,l,v+v_\phi,\nabla(v+v_\phi))-\hat{Q}_C(r,l,v_\phi,\nabla v_\phi) \\ \nonumber&\le
\frac{1}{2}\frac{\partial^2 \hat{Q}_C}{\partial y^2}(r,l,tv+v_\phi,\nabla(tv+v_\phi))(v,v) \\ \nonumber
&+\frac{\partial^2 \hat{Q}_C}{\partial y\partial z}(r,l,tv+v_\phi,\nabla(tv+v_\phi))(v,\nabla v) \\ \label{eqn:qexpv}
&+\frac{1}{2}\frac{\partial^2 \hat{Q}_C}{\partial z^2}(r,l,tv+v_\phi,\nabla(tv+v_\phi))(\nabla v,\nabla v),
\end{align}
for some $t\in[0,1]$.

Consider
\[
 \frac{\partial^2 \hat{Q}_C}{\partial y^2}(r_0,l,(tv+v_\phi)(r_0,l),\nabla(tv+v_\phi)(r_0,l)).
\]
Then by taking the supremum over the closed sets $l\in L$ and $v$ with $|v(r_0,l)|_{r_0,l}+|\nabla v(r_0,l)|_{r_0,l}\le \delta$, which is possible as long as we take $\delta$ sufficiently small, we may bound this expression, as well as the other coefficients of Equation \eqref{eqn:qexpv}. Using the scale equivariance properties of $\hat{Q}_C$ described above, we deduce that as long as $\|v\|_{C^1_1}$ is small, we have that
\[
 |\hat{Q}_C(r,l,v+v_\phi,\nabla(v+v_\phi))-\hat{Q}_C(r,l,v_\phi,\nabla v_\phi)|_r\le C_0(r^{-1}|v|_r+|\nabla v|_r)^2.
\]
Therefore
\begin{align}\nonumber
 r^{2-2\mu}&|\hat{Q}_C(r,l,v+v_\phi,\nabla(v+v_\phi))-\hat{Q}_C(r,l,v_\phi,\nabla v_\phi)|_r \\ \nonumber
 &\le C_0r^{2-2\mu}(r^{-1}|v|_r+|\nabla v|_r)^2 \\ \label{eqn:0est}
&= C_0\|v\|_{C^1_\mu}^2.
\end{align}
Finally, we can take $k$ derivatives of Equation \eqref{eqn:qexpv}, which will give us a polynomial quadratic in $v$ and its derivatives, whose coefficients depend on between two and $k+2$ derivatives of $\hat{Q}_C$, the $C^1_1$-norm of $v_\phi$ and $\delta$ as above. We can estimate these coefficients as we did above. We will find that
\begin{align*}
 &|\nabla^k(\hat{Q}_C(r,l,v+v_\phi,\nabla(v+v_\phi))-\hat{Q}_C(r,l,v_\phi,\nabla v_\phi))| \\
 &\le C\left(\sum_{j_1,j_2}\frac{r^{j_1-1}r^{j_2-1}}{r^k}|\nabla^{j_1}v||\nabla^{j_2}v|\right),
\end{align*}
and since
\[
 r^{k-(2\mu-2)}r^{j_1-1}r^{j_2-1}r^{-k}=r^{j_1-\mu}r^{j_2-\mu},
\]
we may deduce that
\[
 r^{k-(2\mu-2)}|\nabla^k(\hat{Q}_C(r,l,v+v_\phi,\nabla(v+v_\phi))-\hat{Q}_C(r,l,v_\phi,\nabla v_\phi))|\le C_k \|v\|_{C^{k+1}_\mu}^2,
\]
and so we see that the estimate \eqref{eqn:qest} holds.

Finally, as long as $\mu>1$, we have that $C^k_{2\mu-2}(E)\subseteq C^k_{\mu-1}(E)$. We can apply Minkowski's inequality and the smallness of $\|v\|_{C^1_1}$ to \eqref{eqn:qest} to deduce \eqref{eqn:qestpk} and that \eqref{eqn:cscaysob} is a smooth map of Banach spaces, as $\hat{Q}$ is smooth.
\end{proof}

Now that we have described the behaviour of the operator $\hat{F}$ close to the singular point of the conically singular manifold $\hat{Y}$, we will prove a weighted elliptic regularity result for normal vector fields in the kernel of $\hat{F}$.

\begin{prop}\label{prop:weightelliptreg}
 Let $Y$ be a conically singular Cayley submanifold of a $Spin(7)$-manifold $X$. Let $\hat{F}$ be the map defined in Proposition \ref{prop:cscayop}. Then 
\[
 \{v\in C^\infty_\mu(\hat{V}) \, |\, \hat{F}(v)=0\} \cong \{v\in L^p_{k+1,\mu}(\hat{V})\, | \, \hat{F}(v)=0\},
\]
for any $\mu\in (1,2)\backslash\mathcal{D}$, $1<p<\infty$ and $k\in\mathbb{N}$ satisfying $k>1+4/p$. Here $\mathcal{D}$ is the set of exceptional weights given by applying Proposition \ref{prop:lmco} to the linear part of $\hat{F}$.
\end{prop}
\begin{proof}
 We will first show that if $v\in C^\infty_\mu(\hat{V})$ satisfying $\hat{F}(v)=0$, then $v\in L^p_{k+1,\mu}$. This is a little trickier than it seems, since we have that for any $\epsilon>0$, $C^\infty_\mu(\hat{V})\subseteq L^p_{k,\mu-\epsilon}(\hat{V})$, which is weaker than what we require. We will show that if $v\in L^p_{k+1,\mu-\epsilon}(\hat{V})$, for $\epsilon>0$ sufficiently small, satisfies $\hat{F}(v)=0$, then we may deduce that $v\in L^p_{k+1,\mu}(\hat{V})$. Recall that in Lemma \ref{lem:cscayweight}, we saw that we could write
\[
 \hat{F}(v)=Dv+\hat{Q}(v),
\]
where $D$ was defined in Proposition \ref{prop:caylin}, and $\hat{Q}$ is nonlinear. By Proposition \ref{prop:lmco} there exists a discrete set $\mathcal{D}$ so that 
\begin{equation}\label{eqn:ellipregd}
 D:L^p_{k+1,\lambda}(\nu_X(\hat{Y}))\to L^p_{k,\lambda-1}(E),
\end{equation}
is Fredholm as long as $\lambda\notin \mathcal{D}$. Take $0<\epsilon<(\mu-1)/2$ small enough so that $[\mu-\epsilon,\mu]\cap\mathcal{D}=\emptyset$. Let $v\in L^p_{k+1,\mu-\epsilon}(\hat{V})$ and suppose that $\hat{F}(v)=0$. Since \eqref{eqn:ellipregd} is Fredholm when $\lambda=\mu-\epsilon$, we can write
\[
 L^p_{k,\mu-\epsilon-1}(E)=D(L^p_{k+1,\mu-\epsilon}(\nu_X(\hat{Y})))\oplus \hat{\mathcal{O}}_{\mu-\epsilon},
\]
where $\hat{\mathcal{O}}_{\mu-\epsilon}$ is finite-dimensional and
\[
\hat{\mathcal{O}}_{\mu-\epsilon}\cong \text{Coker}_{\mu-\epsilon}D,
\]
where $\text{Coker}_{\lambda}D$ denotes the cokernel of \eqref{eqn:ellipregd}.
Since $[\mu-\epsilon,\mu]\cap \mathcal{D}=\emptyset$, we know that (see \cite[Lem 7.1]{MR837256})
\begin{equation}\label{eqn:cokereq}
 \text{Coker}_{\mu-\epsilon}D=\text{Coker}_\mu D.
\end{equation}
Now since $\hat{F}(v)=0$, we have that $Dv=-\hat{Q}(v)$, and so $\hat{Q}(v)$ is orthogonal to $\text{Coker}_{\mu-\epsilon} D$. Also $\hat{Q}(v)\in L^p_{k,2\mu-2-2\epsilon}(E)\subseteq L^p_{k,\mu-1}(E)$ by Lemma \ref{lem:cscayweight} since $v\in L^p_{k+1,\mu-\epsilon}(\hat{V})$ and by our choice of $\epsilon$. Therefore we have that $Dv=\hat{Q}(v)\in L^p_{k,\mu-1}(\hat{V})$, and it is orthogonal to $\text{Coker}_{\mu}D$ by \eqref{eqn:cokereq}. Therefore there exists $\bar{v}\in L^p_{k+1,\mu}(\hat{V})$ with $Dv=D\bar{v}$. But then we must have that $v-\bar{v}\in \text{Ker}_{\mu-\epsilon}D=\text{Ker}_\mu D$, since $[\mu-\epsilon,\mu]\cap \mathcal{D}=\emptyset$, and so $v\in L^p_{k+1,\mu}(\hat{V})$, as required.

Conversely, let $v\in L^p_{k+1,\mu}(\hat{V})$ satisfy $\hat{F}(v)=0$. Here we perform a trick similar to that in \cite[Prop 4.6]{MR2140999}. Taylor expanding $\hat{F}(v)$ around zero we can write $\hat{F}(v)$ as a polynomial in $v$ and $\nabla v$. Differentiating and gathering terms we can write
 \[
  \nabla\hat{F}(v)=L(x,v(x),\nabla v(x))\nabla^2v+E(x,v(x),\nabla v(x)).
 \]
Consider the second order elliptic linear operator 
\begin{align*}
 L_v: \nu_X(\hat{Y})&\to E, \\
 w&\mapsto L(x,v(x),\nabla v(x))\nabla^2w.
\end{align*}
By Sobolev embedding, we know that $v\in C^l_\mu(\hat{V})$, for $l\ge 2$ by choice of $p$ and $k$, and therefore the coefficients of the linear operator $L_v$ lie in $C^{l-1}_\text{loc}(\hat{V})$. Local regularity for linear elliptic operators with coefficients in H\"{o}lder spaces (a nice statement is given in \cite[Thm 1.4.2]{MR2292510}, taken from \cite[Thm 6.2.5]{MR0202511}) tells us that $v\in C^{l+1}_\text{loc}(\hat{V})$ which is an improvement on the regularity of $v$, and so bootstrapping we may deduce that $v\in C^\infty_\text{loc}(\hat{V})$. (This is why we must differentiate $\hat{F}(v)$, to ensure that the coefficients of the linear operator have enough regularity to improve the regularity of $v$.) Therefore the coefficients of the operator $L_v$ are smooth and so we may apply an estimate of Lockhart and McOwen \cite[Eq. 2.4]{MR837256} in combination with a change of coordinates which tells us that
\begin{equation}\label{eqn:linellipweight}
 \|v\|_{p,k+2,\mu}\le C(\|L_v v\|_{p,k,\mu-2}+\|v\|_{p,0,\mu}).
\end{equation}
Since $\hat{F}(v)=0=\nabla\hat{F}(v)$, we have that 
\[
L_vv=-E(x,v(x),\nabla v(x)).
\]
Since $E(x,v(x),\nabla v(x))$ is a polynomial in $v$ and $\nabla v$ with coefficients that depend on the $C^1_1$-norm of $v$, and $v\in C^1_\mu(\hat{V})$ and $L^p_{k+1,\mu}(\hat{V})$, we have that $E(x,v(x),\nabla v(x))\in L^p_{k,\mu-1}(E)\subseteq L^p_{k,\mu-2}(E)$. Therefore Equation \eqref{eqn:linellipweight} tells us that $v\in L^p_{k+2,\mu}(\hat{V})$, from which we may deduce that $v$ is in fact in $C^\infty_\mu(\hat{V})$.
\end{proof}

We may finally deduce the main theorem of this section, on the expected dimension of the moduli space of Cayley CS deformations of a CS Cayley submanifold of a $Spin(7)$-manifold $X$.
\begin{thm}\label{thm:cscaydefs}
 Let $Y$ be a CS Cayley submanifold at $\hat{x}$ with cone $C$ and rate $\mu\in(1,2)\backslash\mathcal{D}$ of a $Spin(7)$-manifold $X$. Let $D$ denote the first order elliptic differential operator defined in \eqref{eqn:caylin}. Then there exist a smooth manifold $\hat{K}_0$, which is an open neighbourhood of $0$ in the kernel of \eqref{eqn:dpkmu}, and a smooth map $\hat{g}_2$ from $\hat{K}_0$ into the cokernel of \eqref{eqn:dpkmu} with $\hat{g}_2(0)=0$ so that an open neighbourhood of $Y$ in the moduli space of CS Cayley deformations of $Y$ in $X$, $\hat{\mathcal{M}}_\mu(Y)$ from Definition \ref{defn:cscaymodsp}, is homeomorphic to an open neighbourhood of $0$ in $\textnormal{Ker }\hat{g}_2$. 

Moreover, the expected dimension of $\hat{\mathcal{M}}_\mu(Y)$ is given by the index of the linear elliptic operator
\begin{equation}\label{eqn:dpkmu}
 D:L^p_{k+1,\mu}(\nu_X(\hat{Y}))\to L^p_{k,\mu-1}(E).
\end{equation}
If the cokernel of \eqref{eqn:dpkmu} is $\{0\}$ then $\hat{\mathcal{M}}_\mu(Y)$ is a smooth manifold near $Y$ of the same dimension as the kernel of \eqref{eqn:dpkmu}.
 Here $\mathcal{D}$ is the set of weights $\mu\in\mathbb{R}$ for which \eqref{eqn:dpkmu} is not Fredholm from Proposition \ref{prop:lmco}.
\end{thm}
\begin{proof}
 By Propositions \ref{prop:cscayop} and \ref{prop:weightelliptreg}, we can identify $\hat{\mathcal{M}}_\mu(Y)$ near $Y$ with the kernel of the operator
\[
 \hat{F}:L^p_{k+1,\mu}(\hat{V})\to L^p_{k,\mu-1}(E).
\]
The linearisation of $\hat{F}$ at zero is the operator
\begin{equation}\label{eqn:dmuthm}
 D:L^p_{k+1,\mu}(\nu_X(\hat{Y}))\to L^p_{k,\mu-1}(E),
\end{equation}
which is elliptic. Since $\mu\notin \mathcal{D}$, \eqref{eqn:dmuthm} is Fredholm. Therefore we may decompose
\[
L^p_{k+1,\mu}(\nu_X(\hat{Y}))=\hat{K}'\oplus \hat{X}',
\]
where $\hat{K}'$ is the kernel of \eqref{eqn:dmuthm} and $\hat{X}'$ is closed, and
\[
L^p_{k,\mu-1}(E)=D(L^p_{k+1,\mu}(\nu_X(\hat{Y})))\oplus \hat{\mathcal{O}}_\mu,
\]
where $\hat{\mathcal{O}}_\mu$ is the finite-dimensional obstruction space and
\[
 \hat{\mathcal{O}}_\mu\cong L^p_{k,\mu-1}(E)/D(L^p_{k+1,\mu}(\nu_X(\hat{Y})))=:\text{Coker}_\mu D.
\]
Then the map
\begin{align*}
 \hat{\mathcal{F}}:L^p_{k+1,\mu}(\hat{V})\times \hat{\mathcal{O}}_\mu&\to L^p_{k,\mu-1}(E), \\
(v,w)&\mapsto \hat{F}(v)+w,
\end{align*}
has
\begin{equation}\label{eqn:obsmaphat}
 d\hat{\mathcal{F}}|_{(0,0)}(v,w)=Dv+w,
\end{equation}
which is surjective. Write $\hat{K}=\hat{K}\times \{0\}$ for the kernel of \eqref{eqn:obsmaphat}. We then have that
\[
 L^p_{k+1,\mu}(\nu_X(\hat{Y}))\times \hat{\mathcal{O}}_\mu=\hat{K}\oplus (\hat{X}'\times \hat{\mathcal{O}}_\mu).
\]
Now we may apply the Banach space implicit function theorem to find $\hat{K}_0\subseteq \hat{K}$ containing zero, $\hat{X}'_0\subseteq \hat{X}'$, $\hat{\mathcal{O}}_0\subseteq \hat{\mathcal{O}}_\mu$ and a smooth map $\hat{g}=(\hat{g}_1,\hat{g_2}):\hat{K}_0\to \hat{X}'_0\times \hat{\mathcal{O}}_0$ so that
\[
 \hat{\mathcal{F}}^{-1}(0)\cap (\hat{K}_0\times \hat{X}'_0\times \hat{\mathcal{O}}_0)=\{(x,\hat{g}_1(x),\hat{g}_2(x)) \, | \, x\in \hat{K}_0\}.
\]
So we may identify the kernel of $\hat{F}$, and therefore $\hat{M}_\mu(Y)$ with the kernel of $\hat{g}_2:\hat{K}_0\to \hat{\mathcal{O}}_0$, a smooth map between finite-dimensional spaces (since \eqref{eqn:dmuthm} is Fredholm). Sard's theorem tells us that the expected dimension of the kernel of $\hat{g}_2$ is given by the index of the operator \eqref{eqn:dmuthm}. 
\end{proof}

\subsection{Cayley deformations of a CS complex surface}\label{ss:cscompcay}

In this section we prove Theorem \ref{thm:cscxcaydefs} which gives the expected dimension of the moduli space of CS Cayley deformations of a two-dimensional conically singular complex submanifold $N$ of a Calabi--Yau four-fold $M$ in terms of the index of the operator $\infop$ acting on weighted sections of a vector bundle over $\hat{N}$ (the nonsingular part of $N$). 

\subsubsection{Deformation problem}\label{ss:cscxdefprob}

We would like to study the moduli space given in Definition \ref{defn:cscaymodsp} for the CS Cayley submanifold $N$ that is a complex submanifold of a Calabi--Yau four-fold $M$. We will now identify this moduli space with the kernel of a nonlinear partial differential operator.
\begin{prop}\label{prop:cscaylin2}
 Let $N$ be a CS complex surface at $\hat{x}$ with cone $C$ and rate $\mu\in(1,2)$ inside a Calabi-Yau four-fold $M$. Write $\hat{N}:=N\backslash \{\hat{x}\}$. Then the moduli space of CS Cayley deformations of $N$ in $M$, $\hat{\mathcal{M}}_\mu(N)$, can be identified with the kernel of the operator
\begin{equation*}
\hat{F}^\textnormal{cx}:C^\infty_\mu(\hat{U})\to C^\infty_\textnormal{loc}(\Lambda^{0,1}\hat{N}\otimes \nu^{1,0}_M(\hat{N})),
\end{equation*}
where $\hat{U}\subseteq \nu^{1,0}_M(\hat{N})\oplus \Lambda^{0,2}\hat{N}\otimes \nu^{1,0}_M(\hat{N})$ is the image of $\hat{V}\otimes \mathbb{C}$ from the tubular neighbourhood theorem under the isomorphism given in Proposition \ref{prop:bundleisom}, and $\hat{F}^\textnormal{cx}$ is defined so that the following diagram commutes
\[
 \begin{tikzcd}
 C^\infty(U) \arrow{r}{F^\textnormal{cx}}\arrow[swap]{d}{} & C^\infty(\Lambda^{0,1}N\otimes \nu^{1,0}_M(N)) \arrow{d}{}\\
C^\infty(V\otimes \mathbb{C}) \arrow{r}{F}  & C^\infty(E\otimes \mathbb{C})
\end{tikzcd}
\]
where $\hat{F}$ is the operator defined in Proposition \ref{prop:cscayop} and we use the isomorphisms given in Proposition \ref{prop:bundleisom}.

Moreover, the linearisation of $\hat{F}^\textnormal{cx}$ at zero is the operator
\begin{equation}\label{eqn:lincscxcay}
 \bar\partial+\bar\partial^*:C^\infty_\mu(\nu_M^{1,0}(\hat{N})\oplus \Lambda^{0,2}\hat{N}\otimes \nu^{1,0}_M(\hat{N}))\to C^\infty_\textnormal{loc}(\Lambda^{0,1}\hat{N}\otimes \nu^{1,0}_M(\hat{N})).
\end{equation}
\end{prop}
\begin{proof}
By Proposition \ref{prop:cscayop} we can identify the moduli space of CS Cayley deformations of $N$ in $M$ with the kernel of $\hat{F}$, which is the same as the kernel of $\hat{F}^\text{cx}$. 

Since the linearisation of the operator of $\hat{F}$ is given by the operator $D$ defined in Proposition \ref{prop:caylin}, the local argument of Proposition \ref{prop:caylin2} still holds, and so we see that the linearisation of $\hat{F}^\text{cx}$ at zero is given by the operator \eqref{eqn:lincscxcay} as claimed.
\end{proof}

\subsubsection{Cayley deformations of a CS complex surface}\label{ss:caydefscx}

In this section, we will give analogies of the results of Section \ref{sec:cscay}, which were on analytic properties of the operator $\hat{F}$ defined in Proposition \ref{prop:cscayop}, for the operator $\hat{F}^\text{cx}$ defined in Proposition \ref{prop:cscaylin2}. Due to the relation between the operators $\hat{F}$ and $\hat{F}^\text{cx}$ noted in the proof of Proposition \ref{prop:cscaylin2}, these results follow immediately from their counterparts.

\begin{lem}\label{lem:cscxcayweight}
Let $N$ be a conically singular complex surface inside a Calabi--Yau four-fold $M$. Let $\hat{F}^\textnormal{cx}$ be the operator defined in Proposition \ref{prop:cscaylin2}. Then we can write
\begin{equation}\label{eqn:cxqdefn}
 \hat{F}^\textnormal{cx}(w)(x)=(\infop)w(x)+\hat{Q}^\textnormal{cx}(x,w(x),\nabla w(x)),
\end{equation}
for $x\in\hat{N}$, where
\begin{align*}
 \hat{Q}^\textnormal{cx}&:\{(x,y,z) \,|\, (x,y)\in \hat{U}, z\in \left[\nu_x^{1,0}(\hat{N})\oplus \Lambda^{0,2}_x\hat{N}\otimes \nu^{1,0}_x(\hat{N})\right]\otimes T^*_x\hat{N}\} \\
&\to \Lambda^{0,1}\hat{N}\otimes \nu^{1,0}_M(\hat{N}),
\end{align*}
is smooth and $\hat{Q}^\textnormal{cx}(w)(x):=\hat{Q}^\textnormal{cx}(x,w(x),\nabla w(x))$ is a section of $\Lambda^{0,1}\hat{N}\otimes \nu_M^{1,0}(\hat{N})$. Let $\mu>1$. Then for each $k\in\mathbb{N}$, for $w\in C^{k+1}_\mu(\hat{U})$ with $\|w\|_{C^1_1}$ sufficiently small, there exist constants $C_k>0$ so that
\begin{equation}\label{eqn:cxqest}
 \|\hat{Q}^\textnormal{cx}(w)\|_{C^{k}_{2\mu-2}}\le C_k\|w\|_{C^{k+1}_\mu}^2,
\end{equation}
and if $w\in L^p_{k+1,\mu}(\hat{U})$ with $\|w\|_{C^1_1}$ sufficiently small, there exist constants $D_k>0$ such that
\begin{equation} \label{eqn:cxqpk}
\|\hat{Q}^\textnormal{cx}(w)\|_{p,k,2\mu-2}\le D_k\|w\|_{p,k+1,\mu}^2.
\end{equation}
Moreover, we may deduce that
\begin{equation}\label{eqn:cscxcaysob}
 \hat{F}^\textnormal{cx}:L^p_{k+1,\mu}(\hat{U})\to L^p_{k,\mu-1}(E),
\end{equation}
is a smooth map of Banach spaces for any $1<p<\infty$ and $k\in \mathbb{N}$ with $k>1+4/p$.
\end{lem}
\begin{proof}
 Since $\hat{F}^\text{cx}$ is defined by composing the operator $\hat{F}$ defined in Proposition \ref{prop:cscayop} with isomorphisms of vector bundles, the estimates \eqref{eqn:cxqest} and \eqref{eqn:cxqpk} follow from the estimates \eqref{eqn:qest} and \eqref{eqn:qestpk} respectively since the isomorphisms defined in Proposition \ref{prop:bundleisom} are isometries.
 
Moreover, since these isomorphisms are smooth, the claim that \eqref{eqn:cscxcaysob} is a smooth map of Banach spaces follows from the corresponding fact for $\hat{F}$ from Lemma \ref{lem:cscayweight}.
\end{proof}

We may now give a weighted elliptic regularity result for $\hat{F}^\textnormal{cx}$.
\begin{prop}\label{prop:cxweightelliptreg}
 Let $N$ be a conically singular complex surface inside a Calabi--Yau four-fold $M$. Let $\hat{F}^\textnormal{cx}$ be the map defined in Proposition \ref{prop:cscaylin2}. Then 
\[
 \{w\in C^\infty_\mu(\hat{U}) \, |\, \hat{F}^\textnormal{cx}(w)=0\} \cong \{w\in L^p_{k+1,\mu}(\hat{U})\, | \, \hat{F}^\textnormal{cx}(w)=0\},
\]
for any $\mu\in (1,2)\backslash\mathcal{D}$, $1<p<\infty$ and $k\in\mathbb{N}$. Here $\mathcal{D}$ is the set of exceptional weights given by applying Proposition \ref{prop:lmco} to the linear part of $\hat{F}^\textnormal{cx}$.
\end{prop}
\begin{proof}
 This follows from Proposition \ref{prop:weightelliptreg} in combination with the fact that the kernels of $\hat{F}$, defined in Proposition \ref{prop:cscayop}, and $\hat{F}^\text{cx}$ are isomorphic by definition, and the isomorphism given in Proposition \ref{prop:bundleisom} is an isometry.
\end{proof}

We deduce the following theorem on the moduli space of CS Cayley deformations of a CS complex surface inside a Calabi--Yau four-fold. This theorem can be proved by an identical argument to the proof of Theorem \ref{thm:cscaydefs}, but we will deduce it as a corollary of Theorem \ref{thm:cscaydefs}.
\begin{thm}\label{thm:cscxcaydefs}
 Let $N$ be a CS complex surface at $\hat{x}$ with cone $C$ and rate $\mu\in(1,2)\backslash\mathcal{D}$ of a Calabi--Yau four-fold $M$. Then the expected dimension of $\hat{\mathcal{M}}_\mu(N)$ is given by the index of the linear elliptic operator
\begin{equation}\label{eqn:dpkmucx}
 \infop:L^p_{k+1,\mu}(\nu_M^{1,0}(\hat{N})\oplus \Lambda^{0,2}\hat{N}\otimes \nu_M^{1,0}(\hat{N}))\to L^p_{k,\mu-1}(\Lambda^{0,1}\hat{N}\otimes \nu^{1,0}_M(\hat{N})).
\end{equation}
Moreover if the cokernel of \eqref{eqn:dpkmucx} is $\{0\}$ then $\hat{\mathcal{M}}_\mu(N)$ is a smooth manifold near $N$ of the same dimension as the (complex) dimension of the kernel of \eqref{eqn:dpkmucx}.
Here $\mathcal{D}$ is the set of weights $\mu\in\mathbb{R}$ for which \eqref{eqn:dpkmu} is not Fredholm from Proposition \ref{prop:lmco}.
\end{thm}
\begin{proof}
 By Theorem \ref{thm:cscaydefs}, the expected dimension of $\hat{\mathcal{M}}_\mu(N)$ is given by the index of the operator \eqref{eqn:dpkmu}. Since, by Proposition \ref{prop:caylin2} we can consider the operator \eqref{eqn:dpkmucx} as the composition of the operator \eqref{eqn:dpkmu} with the isomorphisms from Proposition \ref{prop:bundleisom}, which are isometries, we may deduce that the index of \eqref{eqn:dpkmu} and \eqref{eqn:dpkmucx} are equal.
\end{proof}

\subsection{Complex deformations of a CS complex surface}\label{sec:cscomp}
In this section, we will compare the CS complex and Cayley deformations of a CS complex surface inside a four-dimensional Calabi--Yau manifold.

\begin{defn}\label{defn:cxcsmodsp}
 Let $N$ be a CS complex surface at $\hat{x}$ with rate $\mu$ and cone $C$ inside a Calabi--Yau manifold $M$ with respect to some $Spin(7)$-coordinate system $\chi$, and denote by $\hat{C}$ the tangent cone of $N$. Write $\hat{N}:=N\backslash \{\hat{x}\}$. Define the \emph{moduli space of conically singular (CS) complex deformations of $N$ in $M$}, $\hat{\mathcal{M}}^\text{cx}_\mu(N)$, to be the set of CS complex surfaces $N'$ at $\hat{x}$ with cone $C$, rate $\mu$ and tangent cone $\hat{C}$ of $M$ so that there exists a continuous family of topological embeddings $\iota_t:N\to M$ with $\iota_0(N)=N$ and $\iota_1(N)=N'$, so that $\iota_t(\hat{x})=\hat{x}$ for all $t\in [0,1]$ and so that $\hat{\iota}_t:=\iota_t|_{\hat{N}}$ is a smooth family of embeddings $\hat{N}\to X$ with $\hat{\iota}_0(\hat{N})=\hat{N}$ and $\hat{\iota}_1({\hat{N}})=\hat{N}':=N'\backslash\{\hat{x}\}$.
\end{defn}
We will now identify the moduli space of CS complex deformations of a CS complex surface in a Calabi--Yau manifold $M$ with the kernel of a nonlinear partial differential operator.
\begin{prop}\label{prop:cscxop}
 Let $N$ be a conically singular complex surface at $\hat{x}$ with rate $\mu$ and cone $C$ inside a Calabi--Yau four-fold $M$. Write $\hat{N}:=N\backslash\{\hat{x}\}$. Let $\hat{V}\subseteq \nu_M(\hat{N})\otimes\mathbb{C}$ be the open set and $\hat{\Xi}:\hat{V}\to \hat{T}$ the diffeomorphism defined in the tubular neighbourhood theorem \ref{prop:cstubnbhd}. For $v\in C^\infty_\textnormal{loc}(\hat{V})$ write  $\Xi_v:=\Xi\circ v$, and define $\hat{N}_v:=\Xi_v(\hat{N})$. Then the moduli space of CS complex deformations of $N$ in $M$, $\hat{\mathcal{M}}_\mu^\textnormal{cx}(N)$, is isomorphic near $N$ to the kernel of
\begin{align}\nonumber
 \hat{G}:C^\infty_\mu(\hat{V}\otimes\mathbb{C})&\to C^\infty_{\textnormal{loc}}(\Lambda^1\hat{N}\otimes T^*M|_{\hat{N}}\otimes \mathbb{C}), \\
v&\mapsto *_{\hat{N}}\,\Xi_v^* (\sigma|_{\hat{N}_v}),
\end{align}
where $\sigma$ was defined in Proposition \ref{prop:compdefs}. Moreover, the kernel of $\hat{G}$ is isomorphic to the kernel of its linear part given by the map 
\begin{align}\nonumber
C^\infty_\mu(\nu_M(\hat{N})\otimes \mathbb{C})&\to C^\infty_\textnormal{loc}(\Lambda^{1,0}\hat{N}\otimes \nu^{*1,0}_M(\hat{N})\oplus \Lambda^{0,1}\hat{N}\otimes \nu^{*0,1}_M(\hat{N})), \\ \label{eqn:csglin}
v&\mapsto -\partial^*(v\hook\Omega)-\bar\partial^*(v\hook \overline{\Omega}).
\end{align}
The kernel of \eqref{eqn:csglin} is isomorphic to
\begin{equation}\label{eqn:kergisom}
 \{v\in C^\infty_\mu(\nu^{1,0}_M(\hat{N})\oplus \Lambda^{0,2}\hat{N}\otimes \nu^{1,0}_M(\hat{N})) \, | \, \bar\partial v=0=\bar\partial^* w\}.
\end{equation}
\end{prop}
\begin{proof}
 By definition of $\sigma$ we see that normal vector fields in the kernel of $\hat{G}$ correspond to complex deformations of $\hat{N}$, and a similar argument to Proposition \ref{prop:cscayop} shows that weighted smooth sections of $\nu_M(\hat{N})\otimes \mathbb{C}$ give conically singular deformations of $\hat{N}$ as required. The linear part of $\hat{G}$ follows from Proposition \ref{prop:glin}, which was a local argument, and similarly that the kernel of $\hat{G}$ is equal to the kernel of its linear part follows from the local argument \cite[Lem 4.7]{compactcay}. Finally, that the kernel of \eqref{eqn:csglin} is equal to \eqref{eqn:kergisom} follows from Proposition \ref{prop:glin}, where we proved that
\[
 \partial^*(v\hook \Omega)=0 \iff \bar\partial (\pi_{1,0}(v))=0,
\]
where $\pi_{1,0}:\nu_M(\hat{N})\otimes \mathbb{C}\to \nu^{1,0}_M(\hat{N})$ and the isomorphism of Proposition \ref{prop:bundleisom}.
\[
 \nu^{0,1}_M(\hat{N})\cong \Lambda^{0,2}\hat{N}\otimes \nu^{1,0}_M(\hat{N}).
\]
\end{proof}
This proposition allows us to prove that the CS complex deformations of a conically singular complex surface are unobstructed. This theorem is a generalisation of Theorem \ref{thm:maincomp} to conically singular submanifolds.
\begin{thm}\label{thm:cxcsdefs}
 Let $N$ be a conically singular complex surface at $\hat{x}$ with rate $\mu\in(1,2)$ and cone $C$ inside a Calabi--Yau four-fold $M$. The moduli space of CS complex deformations of $N$ in $M$, $\hat{\mathcal{M}}_\mu^\textnormal{cx}(N)$ given in Definition \ref{defn:cxcsmodsp}, is a smooth manifold of dimension
\begin{equation}\label{eqn:dimmodsp}
 \textnormal{dim}_\mathbb{C}\textnormal{Ker }\bar\partial+\textnormal{dim}_{\mathbb{C}}\textnormal{Ker }\bar\partial^*=2\textnormal{dim}_{\mathbb{C}}\textnormal{Ker }\bar\partial,
\end{equation}
where
\begin{align}\label{eqn:bp}
 \bar\partial:C^\infty_\mu(\nu^{1,0}_M(\hat{N}))&\to C^\infty_\textnormal{loc}(\Lambda^{0,1}\hat{N}\otimes \nu^{1,0}_M(\hat{N})), \\ \label{eqn:bps}
\bar\partial^*:C^\infty_\mu(\Lambda^{0,2}\hat{N}\otimes\nu^{1,0}_M(\hat{N}))&\to C^\infty_\textnormal{loc}(\Lambda^{0,1}\hat{N}\otimes \nu^{1,0}_M(\hat{N})).
\end{align}
\end{thm}
\begin{proof}
 By Proposition \ref{prop:cscxop} the moduli space of CS complex deformations of $N$ in $M$ can be identified with the kernels of the operators \eqref{eqn:bp} and \eqref{eqn:bps}. Equation \eqref{eqn:dimmodsp} follows since the kernels of the operators \eqref{eqn:bp} and \eqref{eqn:bps} are isomorphic \cite[Cor 4.6]{compactcay}.
\end{proof}

To compare CS complex and Cayley deformations of a CS complex surface, we require the following result.
\begin{prop}\label{prop:nocaydefs}
Let $N$ be a CS complex surface at $\hat{x}$ with cone C and rate $\mu\in(1,2)$ in a Calabi--Yau four-fold $M$. Write $\hat{N}:=N\backslash \{\hat{x}\}$. Then $w\in L^2_{k+1,\mu}(\nu^{1,0}_M(\hat{N})\oplus \Lambda^{0,2}\hat{N}\otimes \nu^{1,0}_M(\hat{N}))$ is an infinitesimal CS Cayley deformation of $\hat{N}$ if, and only if, it is an infinitesimal complex deformation of $\hat{N}$. That is, $(\infop)w=0$ if, and only if, $\bar\partial w=0=\bar\partial^*w$.
\end{prop}
\begin{proof}
Suppose that $w\in L^2_{k+1,\mu}(\nu^{1,0}_M(\hat{N})\oplus \Lambda^{0,2}\hat{N}\otimes \nu^{1,0}_M(\hat{N}))$ satisfies $\bar\partial w=-\bar\partial^* w$ for $\mu\in (1,2)$. Then $\bar\partial^*\bar\partial w=0$. We will check whether
\[
 \int_{\hat{N}}\langle \bar\partial u,v \rangle \text{ vol}_{\hat{N}}=\int_{\hat{N}}\langle u, \bar\partial^* v\rangle \text{ vol}_{\hat{N}},
\]
holds for $u\in L^2_{1,\mu}(\nu^{1,0}_M(\hat{N})\oplus \Lambda^{0,2}\hat{N}\otimes \nu^{1,0}_M(\hat{N}))$ and $v\in L^2_{1,\mu-1}(\nu^{1,0}_M(\hat{N})\oplus \Lambda^{0,2}\hat{N}\otimes \nu^{1,0}_M(\hat{N}))$, that is, whether the integrals on both sides converge. Let $\rho$ be a radius function for $N$. We have that
\[
  \int_{\hat{N}}\langle \bar\partial u,v \rangle \text{ vol}_{\hat{N}}=\int_{\hat{N}} \langle \rho^{1-\mu-2}\bar\partial u, \rho^{\mu+3-2}v\rangle\text{ vol}_{\hat{N}}\le \|\bar\partial u\|_{2,\mu-1}\|v\|_{2,-\mu-3},
\]
by H\"older's inequality. This is finite since 
\[
 |\rho^{\mu+3}v|\le|\rho^{1-\mu}v|,
\]
since $\mu\in (1,2)$. Similarly,
\[
  \int_{\hat{N}}\langle u,\bar\partial^*v \rangle \text{ vol}_{\hat{N}}=\int_{\hat{N}} \langle \rho^{-\mu-2} u, \rho^{\mu+4-2}\bar\partial^* v\rangle\text{ vol}_{\hat{N}}\le \|u\|_{2,\mu}\|\bar\partial^*v\|_{2,-\mu-4},
\]
which again is finite since
\[
 |\rho^{\mu+4}\bar\partial^*v |\le |\rho^{2-\mu}\bar\partial^* v|,
\]
for $\mu\in (1,2)$. Therefore
\[
 \|\bar\partial w\|^2_{L^2}=\int_{\hat{N}} \langle \bar\partial w, \bar\partial w \rangle \text{ vol}_{\hat{N}}=\int_{\hat{N}} \langle w, \bar\partial^*\bar\partial w\rangle \text{ vol}_{\hat{N}}=0,
\]
and so $\bar\partial w=0$.
\end{proof}
This allows us to find that CS complex and Cayley deformations of a CS complex surface in a Calabi--Yau four-fold are the same.
\begin{cor}\label{cor:cscxcaysame}
 Let $N$ be a CS complex surface inside a Calabi--Yau four-fold $M$. Then the moduli space of CS Cayley deformations of $N$ in $M$ is isomorphic to the moduli space of CS complex deformations of $N$ in $M$. 
\end{cor}
\begin{proof}
 There are no infinitesimal CS Cayley deformations of $N$ by Proposition \ref{prop:nocaydefs}, i.e., no $w\in C^\infty_\mu(\nu^{1,0}_M(\hat{N})\oplus \Lambda^{0,2}\hat{N}\otimes \nu^{1,0}_M(\hat{N}))$ satisfying
\[
 \bar\partial w=-\bar\partial^*w,
\]
where $\bar\partial w \ne 0$. Comparing the expected dimension of the moduli space of CS Cayley deformations of $N$ in $M$, computed in Theorem \ref{thm:cscxcaydefs}, to the dimension of the moduli space of CS complex deformations of $N$ in $M$, computed in Theorem \ref{thm:cxcsdefs} we see that these spaces must be the same, since any CS complex deformation of $N$ is a Cayley deformation of $N$.
\end{proof}

\section[Index theory]{Index theory}\label{sec:index}

Let $Y$ be a CS Cayley submanifold of a $Spin(7)$-manifold $X$ with nonsingular part $\hat{Y}$ and let $N$ be a CS complex surface inside a four-dimensional Calabi--Yau manifold $M$ with nonsingular part $\hat{N}$. In this section, we will be interested in the index of the operators
\begin{equation}\label{eqn:cayopind}
 D:L^p_{k+1,\mu}(\nu_X(\hat{Y}))\to L^p_{k,\mu-1}(E),
\end{equation}
from Proposition \ref{prop:caylin} on sections with compact support and extended by density to the above spaces, and
\begin{equation}\label{eqn:cxopind}
 \infop:L^p_{k+1,\mu'}(\nu_M^{1,0}(\hat{N})\oplus \Lambda^{0,2}\hat{N}\otimes \nu_M^{1,0}(\hat{N}))\to L^p_{k,\mu'-1}(\Lambda^{0,1}\hat{N}\otimes \nu^{1,0}_M(\hat{N})).
\end{equation}

We will first characterise the set of exceptional weights $\mathcal{D}$ for which \eqref{eqn:cayopind} and \eqref{eqn:cxopind} are not Fredholm. We will then explain how we can apply the Atiyah--Patodi--Singer index theorem to operators on conically singular manifolds, before applying this result to the operator \eqref{eqn:cxopind}.

\subsection[Finding the exceptional weights]{Finding the exceptional weights for the operators \texorpdfstring{$D$}{D} and \texorpdfstring{$\infop$}{infop}}\label{sec:badweights}

In this section we will find the set $\mathcal{D}$ of exceptional weights for which the linear elliptic operators \eqref{eqn:cayopind} and \eqref{eqn:cxopind} that appeared in Section \ref{sec:csdefs} are not Fredholm. To do this we will study these operators acting on Cayley and complex cones in $\mathbb{R}^8$. We will see that the exceptional weights are actually eigenvalues for differential operators on the links of these cones. 

\subsubsection{Nearly parallel \texorpdfstring{$G_2$}{g2} structure on \texorpdfstring{$S^7$}{s7}}
We can consider $\mathbb{R}^8$ as a cone with link $S^7$. Let $(\Phi_0,g_0)$ be the Euclidean $Spin(7)$-structure (as given in Definition \ref{defn:spin7}). Define a three-form $\varphi$ on $S^7$ by the following relation:
\begin{equation}\label{eqn:nearpar}
 \Phi_0|_{(r,p)}=r^3dr\wedge \varphi|_p +r^4 *\varphi|_p.
\end{equation}
Then $(\varphi,g)$ is a $G_2$-structure on $S^7$ (here $g$ is the standard round metric on $S^7$). Notice that this $G_2$-structure is not torsion-free, however, since $\Phi_0$ is closed we have that
\begin{equation}\label{eqn:s7g2}
 d\varphi=4*\varphi.
\end{equation}
$G_2$-structures $(\varphi,g)$ satisfying \eqref{eqn:s7g2} are called \emph{nearly parallel}.

\subsubsection{Exceptional weights for the operator \texorpdfstring{$D$}{D}}\label{ss:badweightd}
Let $Y$ be a CS Cayley submanifold at $\hat{x}$ with rate $\mu$ and cone $C$ of a $Spin(7)$-manifold $X$ and write $\hat{Y}:=Y\backslash\{\hat{x}\}$. Consider the linear elliptic operator on $\hat{Y}$ given by
\begin{align}\nonumber
D:C^\infty_0(\nu_X(\hat{Y}))&\to C^\infty_0(E), \\    \label{eqn:dbadweight}
v&\mapsto \sum_{i=1}^4\pi_7(e^i\wedge (\nabla_{e_i}^\perp v)^\flat),
\end{align}
where $\{e_1,e_2,e_3, e_4\}$ is an orthonormal frame for $T\hat{Y}$ with dual coframe $\{e^1,e^2,e^3, e^4\}$, $\Lambda^2_7$ is the  seven dimensional irreducible representation of $Spin(7)$ on two-forms with $\pi_7:\Lambda^2X\to \Lambda^2_7$ and $\Lambda^2_7|_{\hat{Y}}=\Lambda^2_+\hat{Y}\oplus E$.

We will now describe the set of exceptional weights for $D$ in terms of an eigenvalue problem on the link of $C$.
\begin{prop}\label{prop:evalprobcay}
 Let $Y$ be a CS Cayley submanifold at $\hat{x}$ with cone $C$ and rate $\mu$ of a $Spin(7)$-manifold $X$. Write $\hat{Y}:=Y\backslash \{\hat{x}\}$. Let $\mathcal{D}_D$ denote the set of $\lambda\in \mathbb{R}$ for which
\[
 D:L^p_{k+1,\lambda}(\nu_X(\hat{Y}))\to L^p_{k,\lambda-1}(E),
\]
defined in \eqref{eqn:dbadweight} is not Fredholm.

Let $L:=C\cap S^7$ be the link of the cone $C$, a submanifold of $S^7$. Then $\lambda\in \mathcal{D}_D$ if, and only if, there exists $v\in C^\infty(\nu_{S^7}(L))$ so that
\begin{equation}\label{eqn:eprobcay}
 D_L v =-\lambda v,
\end{equation}
where for $\{e_1,e_2,e_3\}$ an orthonormal frame for $TL$ and $\nabla^\perp$ the connection on the normal bundle of $L$ in $S^7$ induced by the Levi-Civita connection of the round metric on $S^7$,
\begin{align}\nonumber
D_L:C^\infty(\nu_{S^7}(L))&\to C^\infty(\nu_{S^7}(L)), \\ \label{eqn:asslin}
v&\mapsto \sum_{i=1}^3 e_i\times \nabla_{e_i}^\perp v,
\end{align}
where $\times$ is the cross product on $S^7$ induced from the nearly parallel $G_2$-structure $(\varphi,g)$ defined by
\[
 g(u\times v,w)=\varphi(u,v,w),
\]
for any vector fields $u,v,w$ on $S^7$.
\end{prop}
\begin{rem}
 The operator $D_L$ can be defined on any associative submanifold of a $G_2$-manifold, that is, a manifold with torsion-free $G_2$-structure. Normal vector fields in its kernel correspond to infinitesimal associative deformations of the associative submanifold. This can be deduced from the work of McLean \cite[Thm 5-2]{MR1664890}, however the operator first appears in this form in \cite[Eqn 14]{MR2388088}. Infinitesimal associative deformations of an associative submanifold of $S^7$ with its nearly parallel $G_2$-structure, however, satisfy \eqref{eqn:eprobcay} with $\lambda=1$ as shown by Kawai \cite[Lem 3.5]{MR3672213}. Proposition \ref{prop:evalprobcay} can be considered as a different proof of this fact.
\end{rem}
\begin{proof}
 We can apply Proposition \ref{prop:lmco} to the operator $D$. Suppose that $\rho$ is a radius function for $Y$. Then since the given $Spin(7)$-structure on $X$ approaches the Euclidean $Spin(7)$-structure as we move close to the singular point of $Y$,
\[
 \rho^{-1}D\rho^{-1}
\]
 is asymptotic to the translation invariant differential operator
\[
 r^{-1}D_0 r^{-1},
\]
where $D_0$ is defined similarly to $D$ but using the Euclidean $Spin(7)$-structure pulled back to $X$ by a $Spin(7)$-coordinate system $\chi$ for $X$ around $\hat{x}$ (see Definition \ref{defn:spin7coord}).

By Proposition \ref{prop:lmco} in combination with the discussion in \cite[pg 416]{MR837256}, we see that $\lambda\in \mathcal{D}_D$ if, and only if, there exists a normal vector field $v\in C^\infty(\nu_L(S^7))$ satisfying
\[
 r^{-1}D_0(r^{\lambda-1} v)=0,
\]
where since $\nu_{rl,\mathbb{R}^8}(C)\cong \nu_{l,S^7}(L)$ for all $r>0$ we can consider $(r,l)\mapsto (r,r^{\lambda-1} v(l))$ as a normal vector field on the cone. Note also that the induced Euclidean metric on the normal bundle of $C$ in $\mathbb{R}^8$ takes the form $r^2h$, where $h$ is the metric on the normal bundle of $L$ in $S^7$ induced from the round metric on $S^7$.

Let $\{e_1,e_2,e_3\}$ denote an orthonormal frame for $TL$ with dual coframe $\{e^1,e^2,e^3\}$, and denote by $\Phi_0$ the Euclidean Cayley form on $\mathbb{R}^8$ and $\varphi$ the nearly parallel $G_2$-structure on $S^7$ defined in \eqref{eqn:nearpar}. We compute that
\begin{align*}
 D_0 (r^{\lambda-1}v)&=\pi_7\left(dr\wedge \left(\nabla^\perp_{\frac{\partial}{\partial r}}r^{\lambda-1}v\right)^\flat\right)+\sum_{i=1}^3\pi_7\left(re^i\wedge (\nabla_{\frac{e_i}{r}}^\perp r^{\lambda-1}v)^\flat\right) \\
&=\lambda r^{\lambda-2} dr\wedge v^\flat +\lambda r^{\lambda-2}\Phi_0\left(\frac{\partial}{\partial r},v,\,\cdot\,,\,\cdot\,\right) \\
&+\sum_{i=1}^3\left(r^{\lambda-1}e^i\wedge \left(\nabla^\perp_{e_i}v\right)^\flat+r^{\lambda-3}\Phi_0(e_i,\nabla^\perp_{e_i}v,\,\cdot\,,\,\cdot\,)\right),
\end{align*}
since $\nabla_{\frac{\partial}{\partial r}}^\perp v=r^{-1}v$ as the metric on the normal bundle is of the form $r^2h$. Using the definition of $\varphi$ in \eqref{eqn:nearpar}, we find that
\begin{align*}
 D_0 (r^{\lambda-1}v)&= \lambda r^{\lambda-2}dr \wedge v^\flat +\lambda r^{\lambda+1}\varphi(v,\,\cdot\,,\,\cdot\,) \\
&+\sum_{i=1}^3\left(r^{\lambda-1}e^i\wedge(\nabla^\perp_{e_i}v)^\flat +r^{\lambda}dr\wedge \varphi(e_i,\nabla^\perp_{e_i}v,\,\cdot\,) \right.\\
&\left.+r^{\lambda+1}*\varphi(e_i,\nabla^\perp_{e_i}v,\,\cdot\,,\,\cdot\,)\right).
\end{align*}
Now we wish to replace the musical isomorphism $\flat:\nu_{\mathbb{R}^8}(C)\to \nu^*_{\mathbb{R}^8}(C)$ with the musical isomorphism $\flat_L:\nu_{S^7}(L)\to \nu^*_{S^7}(L)$. Since the metric on $\nu_{\mathbb{R}^8}(C)$ is of the form $r^2h$, where $h$ is a metric on $\nu_{S^7}(L)$, we find that
\begin{align*}
  D_0 (r^{\lambda-1}v)&= \lambda r^{\lambda}dr \wedge v^{\flat_L} +\lambda r^{\lambda+1}\varphi(v,\,\cdot\,,\,\cdot\,) \\
&+\sum_{i=1}^3\left(r^{\lambda+1}e^i\wedge(\nabla^\perp_{e_i}v)^{\flat_L}+r^{\lambda}dr\wedge \varphi(e_i,\nabla^\perp_{e_i}v,\,\cdot\,)\right. \\
&\left. +r^{\lambda+1}*\varphi(e_i,\nabla^\perp_{e_i}v,\,\cdot\,,\,\cdot\,)\right).
\end{align*}
Notice that $E\cong \nu_{S^7}(L)$ via the map
\[
 \alpha \mapsto \left(\frac{\partial}{\partial r}\hook\alpha\right)^{\sharp_L},
\]
where $\sharp_L:\nu^*_{S^7}(L)\to \nu_{S^7}(L)$ is the musical isomorphism, with inverse map
\[
 v \mapsto \pi_7(dr\wedge v^{\flat_L}).
\]
Therefore we see that
\[
 r^{-1}D_0(r^{\lambda-1}v)=0 \iff \left(\frac{\partial}{\partial r}\hook r^{-1}D_0(r^{\lambda-1}v)\right)^{\sharp_L}=0.
\]
We find that
\[
 \left(\frac{\partial}{\partial r} \hook r^{-1} D_0(r^{\lambda-1}v)\right)^{\sharp_L}=r^{\lambda-2}\left(\lambda v+\varphi(e_i,\nabla^\perp_{e_i}v,\,\cdot\,)^\sharp_L\right).
\]
Since by definition,
\[
 D_Lv=e_i\times \nabla_{e_i}^\perp v=\varphi(e_i,\nabla^\perp_{e_i}v,\,\cdot\,)^{\sharp_L},
\]
we see that $\lambda\in \mathcal{D}_D$ if, and only if, there exists $v\in C^\infty(\nu_{S^7}(L))$ such that
\[
 D_Lv=-\lambda v.
\]
\end{proof}

\subsubsection{Exceptional weights for the operator \texorpdfstring{$\infop$}{infop}}

Let $N$ be a CS complex surface with rate $\mu$ and cone $C$ inside a Calabi--Yau four-fold $M$, and write $\hat{N}$ for its nonsingular part.
In order to prove an analogous result to Proposition \ref{prop:evalprobcay} for the operator
\begin{equation}\label{eqn:cxopbad}
 \infop:C^\infty_0(\nu^{1,0}_M(\hat{N})\oplus \Lambda^{0,2}\hat{N}\otimes \nu^{1,0}_M(\hat{N}))\to C^\infty_0(\Lambda^{0,1}\hat{N}\otimes \nu^{1,0}_M(\hat{N})),
\end{equation}
we will need some preliminary facts about complex cones. 
\begin{defn}\label{defn:complexlink}
 Let $C$ be a complex cone in $\mathbb{C}^{n+1}$, with real link $L:=C\cap S^{2n+1}$. Consider the Hopf projection $p:S^{2n+1}\to \mathbb{C}P^n$. Define the \emph{complex link} $\Sigma$ of $C$ to be the image of $L$ under the Hopf projection, i.e., $\Sigma:=p(L)\subseteq \mathbb{C}P^n$.
\end{defn}

The real link of a complex cone $C$ is a circle bundle over the complex link of $C$. Thinking of $L$ as $S^1\times \Sigma$, we can find a globally defined vector field on $L$ that we can think of as being tangent to $S^1$ in this product.
\begin{defn}\label{defn:reeb}
Let $C$ be a complex cone in $\mathbb{C}^{n+1}$, and denote by $J$ the complex structure on $\mathbb{C}^n$. The \emph{Reeb} vector field is defined to be 
\[
 \xi:=J\left(r\frac{\partial}{\partial r}\right).
\]
Notice that $|\xi|_L=1$.
\end{defn}

If $p|_L:L\to \Sigma$ is the restriction of the Hopf projection to $L$, then at each $l\in L$, $\xi_l$ spans the kernel of $d\pi|_l:T_lL\to T_{p(l)}\Sigma$.

\begin{defn}\label{defn:hor}
Let $C$ be a complex cone in $\mathbb{C}^{n+1}$ with real link $L$. Let $\alpha$ be a $p$-form on $L$. We say that $\alpha$ is \emph{horizontal} if $\xi\hook \alpha=0$, where $\xi$ is the Reeb vector field. Denote by $\Lambda^p_hL$ the vector bundle of horizontal $p$-forms on $L$. Denote by $d_h$ the projection of the exterior derivative onto horizontal forms. 

By definition of the Reeb vector field, we see if $J$ is the complex structure on $\mathbb{C}^{n+1}$ then $J(\Lambda^1_hL)\subseteq \Lambda^1_hL$. So we have a well-defined splitting $\Lambda^1_hL=\Lambda^{1,0}_hL\oplus \Lambda^{0,1}_hL$ of one-forms into the $\pm i$ eigenspaces of $J$. Define the operator $\bar\partial_h$ on functions to be the projection of $d_h$ onto horizontal $(0,1)$-forms.
\end{defn}

With these definitions, we may characterise the set of exceptional weights for the operator \eqref{eqn:cxopbad} in terms of an eigenproblem on the link of a cone.
\begin{prop}\label{prop:evalprobcx}
 Let $N$ be a CS complex surface at $\hat{x}$ with rate $\mu$ and cone $C$ inside a Calabi--Yau four-fold $M$. Write $\hat{N}:=N\backslash\{\hat{x}\}$. Let $\mathcal{D}$ denote the set of $\lambda\in\mathbb{R}$ for which 
\begin{equation}\label{eqn:cxopbad2}
 \infop:L^p_{k+1,\lambda}(\nu^{1,0}_M(\hat{N})\oplus \Lambda^{0,2}\hat{N}\otimes \nu^{1,0}_M(\hat{N}))\to L^p_{k,\lambda-1}(\Lambda^{0,1}\hat{N}\otimes \nu^{1,0}_M(\hat{N})),
\end{equation}
is not Fredholm.
Let $L$ denote the real link of $C$. Then $\lambda\in \mathcal{D}$ if, and only if, there exist $v\in C^\infty(\nu^{1,0}_{S^7}(L))$ and $w\in C^\infty(\Lambda^{0,1}_hL\otimes \nu^{1,0}_{S^7}(L))$ so that
\begin{align}\label{eqn:cxev1}
 \bar\partial_h v&=(\lambda+2-i\nabla_\xi)w, \\ \label{eqn:cxev2}
\bar\partial^*_h w&=\frac{1}{2}(\lambda+i\nabla_{\xi})v,
\end{align}
where $\xi$ is the Reeb vector field on $L$. Here $\nabla$ acts on $\Lambda^{0,1}_hL$ as the Levi-Civita connection of the metric on $L$ and on $\nu^{1,0}_{S^7}(L)$ as the normal part of the Levi-Civita connection on $S^7$.
\end{prop}
\begin{proof}
 Similarly to the proof of Proposition \ref{prop:evalprobcay}, if $\rho$ is a radius function for $N$ then we can see that 
\[
 \bar\partial+\bar\partial^*\rho^2,
\]
on $\hat{N}$ is asymptotically translation invariant to
\[
 \bar\partial_C+\bar\partial^*_Cr^2,
\]
on the cone $C$ where this time we take a metric on $\nu_{\mathbb{C}^4}(C)$ that is independent of $r$. If $v\in C^\infty(\nu_{S^7}(L)\otimes \mathbb{C})$ we can think of $r^\mu v$ as a complexified normal vector field on $C$, and moreover the complex structure $J$ on $\mathbb{C}^4$ induces a splitting
\[
 \nu_{S^7}(L) \otimes \mathbb{C}=\nu^{1,0}_{S^7}(L)\oplus \nu^{0,1}_{S^7}(L),
\]
of the complexified normal bundle of $L$ in $S^7$ into holomorphic and antiholomorphic parts (the $i$ and $-i$ eigenspaces of $J$ respectively). Also, by definition of the Reeb vector field, if we take $\theta\in C^\infty(\Lambda^1L)$ to be the dual one-form to $\xi$ we have that $dr-ir\theta$ is a $(0,1)$-form on $C$. Since $\Lambda^2C\cong \Lambda^2L\oplus dr\wedge \Lambda^1L$, we can see that a $(0,2)$-form on $C$ must be of the form
\[
r^\mu(dr-ir\theta)\wedge w,
\]
where $w\in C^\infty(\Lambda^{0,1}_hL)$. By Proposition \ref{prop:lmco} in combination with the discussion in \cite[pg 416]{MR837256}, we deduce that $\lambda\in \mathcal{D}$ if, and only if, there exists $v\in C^\infty(\nu_{S^7}(L))$ and $w\in C^\infty(\Lambda^{0,1}_hL\otimes \nu^{1,0}_{S^7}(L))$ so that
\[
 \bar\partial_C(r^\lambda v)+\bar\partial^*_C\left(r^{\lambda+2}\left(\frac{dr}{r}-i\theta\right)\wedge w\right)=0,
\]
where $\theta$ is dual to the Reeb vector field $\xi$.
We can calculate that
\[
 d_C(r^\lambda v)=\lambda r^{\lambda-1}dr\otimes v +r^\lambda \theta\otimes\nabla_{\xi}v+r^\lambda d_hv,
\]
and therefore
\begin{equation}\label{eqn:bpc}
 \bar\partial_C(r^\lambda v)= r^\lambda \frac{1}{2}\left(\frac{dr}{r}-i\theta\right)\otimes(\lambda+i\nabla_\xi)v+r^\lambda\bar\partial_h v.
\end{equation}
We also have that
\begin{align*}
 \bar\partial^*_C\left(\left(\frac{dr}{r}-i\theta\right)\wedge r^{\lambda+2}w\right)&=-\frac{\partial}{\partial r}\hook \nabla_{\frac{\partial}{\partial r}}\left(\left(\frac{dr}{r}-i\theta\right)\wedge r^{\lambda+2}w\right) \\
-\frac{1}{r^2} \xi\hook \nabla_\xi\left(\left(\frac{dr}{r}-i\theta\right)\wedge r^{\lambda+2}w\right)
&-r^\lambda\left(\frac{dr}{r}-i\theta\right)\bar\partial^*_hw,
\end{align*}
where since $w$ is a horizontal $(0,1)$-form we see that any term gained from applying $\bar\partial^*_h$ to $r^{-1}dr-i\theta$ must be a multiple of $w$ at each point and therefore will vanish under exterior product with $w$.
We have that
\[
 -\frac{\partial }{\partial r} \hook \nabla_{\frac{\partial}{\partial r}}\left(\left(\frac{dr}{r}-i\theta\right)\wedge r^{\lambda+2}w\right)=-(\lambda+1)r^\lambda w,
\]
and
\[
 -\xi\hook \nabla_\xi\left(\left(\frac{dr}{r}-i\theta\right)\wedge r^{\lambda+2}w\right)=-r^{\lambda+2} w+ir^{\mu+2} \nabla_\xi w,
\]
since $\nabla_{\xi}dr=r\theta$ where $\nabla$ is the Levi-Civita connection of the cone metric. We deduce that
\begin{equation}\label{eqn:bpsc}
 \bar\partial^*_C\left(\left(\frac{dr}{r}-i\theta\right)\wedge r^{\lambda+2}w\right)=-r^\lambda(\lambda+2-i\nabla_\xi)w-r^\lambda\left(\frac{dr}{r}-i\theta\right)\bar\partial^*_h w.
\end{equation}
Equating \eqref{eqn:bpc} and minus \eqref{eqn:bpsc}, we find that $\lambda\in \mathcal{D}$ if, and only if, there exist $v\in C^\infty(\nu_{S^7}^{1,0}(L))$ and $w\in C^\infty(\Lambda^{0,1}_hL\otimes \nu^{1,0}_{S^7}(L))$ satisfying 
\begin{align*}
 \bar\partial_h v &=(\lambda+2-i\nabla_\xi)w, \\
\bar\partial^*_h w&=\frac{1}{2}(\lambda+i\nabla_\xi)v,
\end{align*}
as claimed.
\end{proof}

\subsubsection{An eigenproblem on the complex link}\label{ss:trick1}

In Proposition \ref{prop:evalprobcx} we characterised the set of exceptional weights $\mathcal{D}$ for which the operator \eqref{eqn:cxopbad2} is not Fredholm in terms of an eigenproblem on the real link of a complex cone $C$. In this section we will introduce a trick used by Lotay \cite[\S 6]{MR2981841} to study an eigenvalue problem on the link of a coassociative cone which is a circle bundle over a complex curve in $\mathbb{C}P^2$. This will allow us to give an equivalent eigenvalue problem to \eqref{eqn:cxev1}-\eqref{eqn:cxev2} on the real link of $C$ completely in terms of operators and vector bundles on the complex link of $C$.

Let $C$ be a complex cone in $\mathbb{C}^4$ with real link $L\subseteq S^7$ and complex link $\Sigma\subseteq \mathbb{C}P^3$. Suppose we have a problem of the following form: Find all of the functions $f$ on $L$ that satisfy
\begin{equation}\label{eqn:prob}
 \mathcal{L}_{\xi} f=imf, \quad \bar\partial_h f=0,
\end{equation}
for some $m\in \mathbb{Z}$, where $\xi$ is the Reeb vector field on $C$.

We would like to understand the relationship between the operator $\bar\partial_h$ on the real link of $C$ and $\bar\partial_\Sigma$ on the complex link $C$.

\begin{defn}
 Call a function, horizontal vector field or horizontal differential form $f$ on $L$ \emph{basic} if 
\[
 \mathcal{L}_{\xi}f=0.
\]
\end{defn}
Basic functions, forms and vector fields are special because they are in one-one correspondence with functions, forms and vector fields on $\Sigma$. It follows from \cite[Lem 1]{MR0200865} that $\bar\partial_h$ acting on basic functions, forms or vector fields on L is equivalent to $\bar\partial_\Sigma$ acting on functions, forms or vector fields on $\Sigma$. In Problem \eqref{eqn:prob}, when $m\ne 0$, $f$ is not basic. However, a simple trick allows us to pretend that $f$ is basic. 

By the definition of the complex link, we may identify the cone $C$ with the vector bundle $\mathcal{O}_{\mathbb{C}P^3}(-1)|_\Sigma$, that is, the tautological line bundle over $\mathbb{C}P^3$ restricted to $\Sigma$. This is then a trivial (real) line bundle over $L$ and therefore has a global section given by the map $x\mapsto s(x)=x$ for $x\in L$. It is easy to see that $\mathcal{L}_{\xi}s=is$, and therefore 
\[
 f\otimes s^{-m},
\]
is a section of the vector bundle $\mathcal{O}_{\mathbb{C}P^3}(m)|_\Sigma$ satisfying 
\[
 \mathcal{L}_{\xi}(f\otimes s^{-m})=0,
\]
and therefore pushes down to a well-defined section of the vector bundle $\mathcal{O}_{\mathbb{C}P^3}(m)|_\Sigma$. Since $\mathcal{O}_{\mathbb{C}P^3}(m)|_{\Sigma}$ is a trivial line bundle over $L$, we can still consider $f\otimes s^{-m}$ as a function on $L$. Therefore we can rephrase Problem \eqref{eqn:prob} as: Find all basic sections $\tilde{f}$ of $\mathcal{O}_{\mathbb{C}P^3}(m)|_\Sigma\to L$ satisfying
\[
 \bar\partial_h\tilde{f}=0.
\]
This is now equivalent to finding the sections $\tilde{f}$ of $\mathcal{O}_{\mathbb{C}P^3}(m)|_\Sigma\to \Sigma$ that satisfy
\[
 \bar\partial_\Sigma\tilde{f}=0.
\]
Therefore we have reduced Problem \eqref{eqn:prob} to asking: How many holomorphic sections of the line bundle $\mathcal{O}_{\mathbb{C}P^3}(m)|_\Sigma$ are there?

This problem is easily solved using the Hirzebruch--Riemann--Roch Theorem \cite[Thm 5.1.1]{MR2093043}.
\begin{thm}[Hirzebruch--Riemann--Roch]\label{thm:rr}
 Let $\Sigma$ be a Riemann surface and let $E$ be a vector bundle over $\Sigma$. Denote by $h^0(\Sigma,E)$ the dimension of the space of holomorphic sections of $E$. Let $K_\Sigma$ denote the canonical bundle of $\Sigma$. Then
\[
 h^0(\Sigma,E)=h^0(\Sigma,E^*\otimes K_\Sigma)+\textnormal{deg}(E)+\textnormal{rk}(E)(1-g),
\]
where $\textnormal{deg}(E)$ is the degree of the vector bundle $E$, $\textnormal{rk}(E)$ is the rank of the vector bundle and $g$ is the genus of $\Sigma$.
\end{thm}
We will now apply the trick that we described above to rephrase the eigenvalue problem \eqref{eqn:cxev1}-\eqref{eqn:cxev2} on the real link of a cone as an eigenvalue problem on the complex link on a cone.
\begin{prop}\label{prop:evcxlink}
 Let $C$ be a complex cone in $\mathbb{C}^4$ with real link $L$ and complex link $\Sigma$. Then given $\lambda\in \mathbb{R}$ and $m\in \mathbb{Z}$, pairs $v\in C^\infty(\nu_{\mathbb{C}P^3}^{1,0}(\Sigma)\otimes \mathcal{O}_{\mathbb{C}P^3}(m)|_\Sigma)$ and $w\in C^\infty(\Lambda^{0,1}\Sigma\otimes \nu_{\mathbb{C}P^3}^{1,0}(\Sigma)\otimes \mathcal{O}_{\mathbb{C}P^3}(m)|_{\Sigma})$ so that
\begin{align}\label{eqn:evprobcxlink1}
\bar\partial_{\Sigma}v &=(\lambda+3+m)w, \\ \label{eqn:evprobcxlink2}
 \bar\partial^*_\Sigma w&=\frac{1}{2}(\lambda-1-m)v,
\end{align}
are in a one-one correspondence with pairs $\tilde{v}\in C^\infty(\nu^{1,0}_{S^7}(L))$ and $\tilde{w}\in C^\infty(\Lambda^{0,1}_hL\otimes \nu^{1,0}_{S^7}(L))$ satisfying
\[
 \mathcal{L}_\xi \tilde{v}=im\tilde v, \quad \mathcal{L}_\xi\tilde{w}=im\tilde{w},
\]
where $\xi$ is the Reeb vector field, and the eigenvalue problem \eqref{eqn:cxev1}-\eqref{eqn:cxev2}.
\end{prop}
\begin{proof}
 We can pull back $v$ and $w$ to basic sections of $\nu^{1,0}_{S^7}(L)\otimes \mathcal{O}_{\mathbb{C}P^3}(m)|_\Sigma$ and $\Lambda^{0,1}_hL\otimes \nu^{1,0}_{S^7}(L)\otimes \mathcal{O}_{\mathbb{C}P^3}(m)|_\Sigma$ over $L$ respectively. As mentioned above, these sections are in one-one correspondence with sections $\tilde v$ and $\tilde w$ of $\nu^{1,0}_{S^7}(L)$ and $\Lambda^{0,1}_hL\otimes \nu^{1,0}_{S^7}(L)$ respectively satisfying
\begin{equation}\label{eqn:liederiv}
 \mathcal{L}_\xi \tilde v=im\tilde v, \quad \mathcal{L}_\xi \tilde w=im \tilde w.
\end{equation}
So we see that $v$ and $w$ are in one-one correspondence with $\tilde v$ and $\tilde w$ satisfying \eqref{eqn:liederiv}, and $\tilde v$ and $\tilde w$ satisfy
\begin{align*}
 \bar\partial_h \tilde v&=(\lambda+3+m)\tilde w, \\
\bar\partial^*_h \tilde w&=\frac{1}{2}(\lambda-1-m)\tilde v.
\end{align*}
Finally, by \cite[Lemma 3, \S 5]{MR0200865}, we see that any horizontal vector field $X$ on $S^7$ viewed as a circle bundle over $\mathbb{C}P^3$ satisfies
\[
 \text{horizontal part}(\nabla_X \xi)=JX.
\]
and so for any vector field of type $(1,0)$, we have that
\[
 \mathcal{L}_\xi v=\nabla_\xi v-\nabla_v \xi=\nabla_\xi v -iv.
\]
Therefore \eqref{eqn:liederiv} implies that
\[
 \nabla_\xi \tilde v= i(m+1)\tilde v, \quad \nabla_\xi \tilde w=i(m+1)\tilde w,
\]
and therefore
\begin{align*}
 \bar\partial_h \tilde v&= (\lambda+2 -i\nabla_\xi)\tilde w, \\
\bar\partial^*_h \tilde w&= \frac{1}{2}(\lambda+i\nabla_\xi)\tilde v,
\end{align*}
as required.
\end{proof}
\subsection{Dimension of the moduli space of complex deformations of a CS complex surface}\label{sec:indcomp}

In this section, we will deduce a version of the Atiyah--Patodi--Singer index theorem for operators on conically singular manifolds. We will then apply this result to prove Theorem \ref{thm:indcomp}, an index formula for the operator \eqref{eqn:cxopind}, which allows us to compare the dimension of the moduli space of CS complex deformations of a conically singular complex surface to what we will think of as the dimension of the moduli space of all complex deformations of a CS complex surface in a Calabi--Yau four-fold based on Kodaira's theorem \cite[Thm 1]{MR0133841} on deformations of complex submanifolds of complex varieties.

\subsubsection{The Atiyah--Patodi--Singer index theorem for conically singular manifolds}
The Atiyah--Patodi--Singer index theorem is predominantly for a certain type of elliptic operator on a manifold with boundary. However, as a corollary to the main theorem, they also prove an index theorem for translation invariant operators on a manifold with a cylindrical end, which we quote here.
\begin{thm}[{\cite[Thm 3.10 \& Cor 3.14]{MR0397797}}]\label{thm:apsind}
 Let 
 \[
 A:C^\infty(E)\to C^\infty(F)
 \]
 be a linear elliptic first order translation invariant differential operator on a manifold $\hat{X}$ with a cylindrical end $Y\times (0,\infty)$ that takes the special form
 \[
 A=\sigma\left(\frac{\partial}{\partial u}+B\right),
 \]
 on $Y\times(0,\infty)$, where $u$ is the inward normal coordinate, $\sigma:E|_Y\to F|_Y$ is a bundle isomorphism and $B$ is a self adjoint elliptic operator on $Y$. Then 
\begin{equation}\label{eqn:indcylend}
 \textnormal{ind}_{L^2} A=\int_{\hat{X}} \alpha_0(x) \, dx-\frac{h+\eta(0)}{2}+h_\infty(F),
\end{equation}
where $h,\eta, \alpha_0$ and $h_\infty(F)$ are defined as follows:
\begin{enumerate}[(i)]
 \item $\alpha_0(x)$ is the constant term in the asymptotic expansion (as $t\to 0$) of
\[
 \sum e^{-t\mu'}|\phi'_\mu(x)|^2-\sum e^{-t\mu''}|\phi''_\mu(x)|^2,
\]
where $\mu',\phi'_\mu$ denote the eigenvalues and eigenfunctions of $A^*A$ on the double of $X$, and $\mu'',\phi''_\mu$ are the corresponding objects for $AA^*$.
\item $h=\textnormal{dim } \textnormal{Ker } B=$ multiplicity of the $0$-eigenvalue of $B$.
\item $\eta(s)=\sum_{\lambda\ne 0}(\textnormal{sign }\lambda) |\lambda|^{-s}$, where $\lambda$ runs over the eigenvalues of $B$.
\item $h_\infty(F)$ is the dimension of the subspace of $\textnormal{Ker } B$ consisting of limiting values of extended $L^2$ sections $f$ of $F$ satisfying $A^*f=0$.
\end{enumerate}
Here we call $f$ an \emph{extended $L^2$-section} of $E$ if $f\in L^2_\text{loc}(E)$ and on the cylindrical end of $\hat{X}$, for large $t$, $f$ takes the form
\[
 f(y,t)=g(y,t)+f_\infty(y),
\]
for $g\in L^2(E)$ and $f_\infty\in \textnormal{Ker }B$.
\end{thm}

We will now explain how we can apply the Atiyah--Patodi--Singer index theorem \ref{thm:apsind} to elliptic operators on conically singular manifolds.

We first give a technical result that relates the adjoint of a differential operator on a conically singular manifold to the adjoint of the related asymptotically translation invariant operator acting on the conformally equivalent manifold with cylindrical end.
\begin{lem}\label{lem:adj}
 Let $M$ be an $m$-dimensional conically singular manifold at $\hat{x}$ and let $\rho$ be a radius function for $M$. Write $\hat{M}:=M\backslash\{\hat{x}\}$, and $g$ for the metric on $\hat{M}$. Let 
\[
 A:C^\infty_0(T^q_s\hat{M})\to C^\infty_0(T^{q'}_{s'}\hat{M}),
\]
be a linear first order differential operator on $\hat{M}$ and suppose there exists $\lambda\in \mathbb{R}$ so that
\[
 \tilde{A}:=\rho^{\lambda+s'-q'}A\rho^{q-s},
\]
is an asymptotically translation invariant operator. Then the formal adjoint of the operator $\tilde{A}$ (with respect to the metric $\rho^{-2}g$)
\[
 \tilde{A}^*:C^\infty_0(T^{q'}_{s'}\hat{M})\to C^\infty_0(T^q_s\hat{M}),
\]
is of the form
\[
 \tilde{A}^*=\rho^{s-q+m}A^*\rho^{\lambda-s'+q'-m},
\]
where 
\[
 A^*:C^\infty_0(T^{q'}_{s'}\hat{M})\to C^\infty_0(T^q_s\hat{M}),
\]
is the formal adjoint of $A$ with respect to $g$.

Moreover, using the notation of Definitions \ref{defn:weightac} and \ref{defn:weightcs}, the kernel of
\begin{equation}\label{eqn:adjcyl}
 \tilde{A}^*:W^p_{k+1,\mu}(T^{q'}_{s'}\hat{M})\to W^p_{k,\mu}(T^q_s\hat{M}),
\end{equation}
is isomorphic to the kernel of
\begin{equation}\label{eqn:adjcs}
 A^*:L^p_{k+1,\mu+\lambda-m}(T^{q'}_{s'}\hat{M})\to L^p_{k,\mu-m}(T^q_s\hat{M}),
\end{equation}
for any $\mu\in \mathbb{R}$, $k\in \mathbb{N}$ and $1<p<\infty$.
\end{lem}
\begin{proof}
 Let $v\in C^\infty_0(T^{q}_s\hat{M})$ and $w\in C^\infty_0(T^{q'}_{s'}\hat{M})$. Then
\begin{align*}
 \int_{\hat{M}}\langle \tilde{A}v,w\rangle_{\rho^{-2}g} \text{ vol}_{\rho^{-2}g}&= \int_{\hat{M}}\langle \rho^{\lambda+s'-q'}A\rho^{q-s}v,w\rangle_{\rho^{-2}g}\text{ vol}_{\rho^{-2}g} \\
&= \int_{\hat{M}}\rho^{2q'-2s'}\langle \rho^{\lambda+s'-q'}A\rho^{q-s}v,w\rangle_{g}\;\rho^{-m}\text{vol}_{g} \\
&=\int_{\hat{M}}\langle \rho^{\lambda-s'+q'}A\rho^{q-s}v,w\rangle_g \;\rho^{-m}\text{vol}_g \\
&=\int_{\hat{M}}\langle A \rho^{q-s}v, \rho^{\lambda-m-s'+q'}w\rangle_g \text{ vol}_g \\
&=\int_{\hat{M}} \langle v, \rho^{q-s}A^*\rho^{\lambda-m-s'+q'}w\rangle_g\text{ vol}_g \\
&= \int_{\hat{M}}\rho^{2q-2s}\langle v,\rho^{s-q+m}A^*\rho^{\lambda-m-s'+q'}w\rangle_{g}\;\rho^{-m}\text{vol}_{g} \\
&=\int_{\hat{M}}\langle v , \rho^{s-q+m}A^*\rho^{\lambda-m-s'+q'}w\rangle_{\rho^{-2}g}\text{ vol}_{\rho^{-2}g},
\end{align*}
where we have used that $A^*$ is the formal adjoint of $A$ with respect to the metric $g$, which shows that
\[
 \tilde{A}^*:=\rho^{s-q+m}A^*\rho^{\lambda-m-s'+q'},
\]
is the formal adjoint of $\tilde{A}$ with respect to the metric $\rho^{-2}g$. By Lemma \ref{lem:weightisom}
\[
 \rho^{\lambda-m-s'+q'}:W^p_{k+1,\mu}(T^{q'}_{s'}\hat{M})\to L^p_{k+1,\mu+\lambda-m}(T^{q'}_{s'}\hat{M}),
\]
is an isomorphism and so by definition of $\tilde{A}^*$ and $A^*$ the kernels of \eqref{eqn:adjcyl} and \eqref{eqn:adjcs} are isomorphic.
\end{proof}

We may now deduce the following proposition from Theorem \ref{thm:apsind} and Lemma \ref{lem:adj} to give an index theorem for operators on conically singular submanifolds.
\begin{prop}\label{prop:apscs}
 Let $M$ be an $m$-dimensional conically singular manifold at $\hat{x}$ with radius function $\rho$. Let $T^q_s\hat{M}$ be the vector bundle of $(s,q)$-tensors on $\hat{M}:=M\backslash \{\hat{x}\}$. Let 
\[
 A:C^\infty_0(T^q_s\hat{M})\to C^\infty_0(T^{q'}_{s'}\hat{M}),
\]
be a first order linear elliptic differential operator so that
\[
 \tilde{A}:=\rho^{\lambda+s'-q'} A\rho^{q-s},
\]
is asymptotically translation invariant to $\tilde{A}_\infty$ for some $\lambda\in \mathbb{R}$. Then for $\mu\in \mathbb{R}\backslash \mathcal{D}$, given in Proposition \ref{prop:lmco}, the index of 
\begin{equation}\label{eqn:l2inda}
 A:L^2_{k+1,\mu}(T^q_s\hat{M})\to L^2_{k,\mu-\lambda}(T^{q'}_{s'}\hat{M}),
\end{equation}
differs by a constant from the index $\textnormal{ind}_\mu A_\infty$ of 
\begin{equation}\label{eqn:l2indatrans}
 A_\infty:=r^{q'-s'-\lambda}\tilde{A}_\infty r^{s-q}:L^2_{k+1,\mu}(T^q_s\hat{M})\to L^2_{k,\mu-\lambda}(T^{q'}_{s'}\hat{M}),
\end{equation}
which satisfies
\begin{equation}\label{eqn:indformulacs}
 \textnormal{ind}_\epsilon A_\infty=\int_{\hat{M}} \alpha_0(x) dx -\frac{h+\eta(0)}{2},
\end{equation}
for $\epsilon>0$ chosen so that $(0,\epsilon]\cap \mathcal{D}=\emptyset$ and we use the notation of Theorem \ref{thm:apsind} for the terms on the right hand side of \eqref{eqn:indformulacs} (and these terms are defined for the translation invariant operator $\tilde{A}_\infty$).
\end{prop}
\begin{proof}
By Proposition \ref{prop:lmco}, we know that $A$ and $\tilde{A}$ have the same kernel and cokernel when acting on weighted Sobolev spaces, and moreover, the index of these operators differ from the index of $\tilde{A}_\infty$ by a constant independent of the weight.

Since $\tilde{A}_\infty$ is translation invariant, we can apply Theorem \ref{thm:apsind} to $\tilde{A}_\infty$. Let $\text{Ker}_\mu\, \tilde{A}_\infty$ and $\text{Ker}_\mu\, \tilde{A}_\infty^*$ denote the kernels of 
\begin{align*}
 \tilde{A}_\infty:W^2_{k+1,\mu}(T^{q}_{s}\hat{M})&\to W^2_{k,\mu}(T^{q'}_{s'}\hat{M}), \\
\tilde{A}_\infty^*:W^2_{k+1,\mu}(T^{q'}_{s'}\hat{M})&\to W^2_{k,\mu}(T^{q}_{s}\hat{M}),
\end{align*}
respectively, where $\tilde{A}^*_\infty$ is the formal adjoint of $\tilde{A}_\infty$ with respect to the metric $\rho^{-2}g$, where $g$ is the metric on $\hat{M}$.
Then Theorem \ref{thm:apsind} yields that
\begin{equation}\label{eqn:csind1}
 \text{dim Ker}_0 \, \tilde{A}_\infty -\text{dim Ker}_0\, \tilde{A}_\infty^*=\int_{\hat{M}} \alpha_0(x) dx -\frac{h+\eta(0)}{2}+h_\infty(T^{q'}_{s'}\hat{M}).
\end{equation}
By definition of $\tilde{A}_\infty$, $\text{Ker}_0\, \tilde{A}_\infty\cong \text{Ker}_0\, A_\infty$, where $\text{Ker}_\mu\, A_\infty$ denotes the kernel of \eqref{eqn:l2indatrans}, and by Lemma \ref{lem:adj}, $\text{Ker}_0 \, \tilde{A}_\infty^*\cong \text{Ker}_{\lambda-m}\, A_\infty^*$, where $A^*_\infty$ is the formal adjoint of $A_\infty$ with respect to the metric $g$ and $\text{Ker}_\mu \, A_\infty^*$ denotes the kernel of
\[
 A_\infty^*:L^2_{k+1,\mu}(T^{q'}_{s'}\hat{M})\to L^2_{k,\mu-\lambda}(T^q_s\hat{M}).
\]
So we see that
\begin{equation}\label{eqn:csind2}
 \text{dim Ker}_0 \, A_\infty -\text{dim Ker}_{\lambda-m}\, A_\infty^*=\int_{\hat{M}} \alpha_0(x) dx -\frac{h+\eta(0)}{2}+h_\infty(T^{q'}_{s'}\hat{M}).
\end{equation}
Denote by $\mathcal{D}$ the subset of $\mathbb{R}$ for which $\mu\in \mathcal{D}$ if, and only if, \eqref{eqn:l2indatrans} is not Fredholm. Then we might have a problem equating
\[
  \text{dim Ker}_0 \, A_\infty -\text{dim Ker}_{\lambda-m}\, A_\infty^*=\text{ind}_0 \,A_\infty,
\]
since if $0\in \mathcal{D}$ then $\text{ind}_0 \,A_\infty$ may not be defined. Take $\epsilon>0$ so that
\[
 (0,\epsilon]\cap \mathcal{D}=\emptyset.
\]
Then $\text{ind}_\epsilon\, A_\infty$ is well-defined. Since $\epsilon>0$, we have that
\[
 \text{Ker}_\epsilon \,A_\infty\subseteq \text{Ker}_0 \,A_\infty,
\]
where $\text{Ker}_\mu\, A_\infty$ denotes the kernel of \eqref{eqn:l2indatrans}. It is claimed that
\[
 \text{Ker}_\epsilon \,A_\infty=\text{Ker}_0 \,A_\infty.
\]
To see this, suppose that $\alpha \in \text{Ker}_0 A_\infty$. Then by elliptic regularity, $\alpha$ is smooth, and by definition of weighted norm on $L^2_{k+1,0}(T^q_s\hat{M})$ $\alpha$ must decay to zero as $r\to 0$ and so we must have that $\alpha=\mathcal{O}(r^{\epsilon'})$ for some $\epsilon'>0$. Taking $\epsilon'$ smaller if necessary we can guarantee that $\mathcal{D}\cap (0,\epsilon']=\emptyset$. The rate of decay of $\alpha$ allows us to deduce that $\alpha\in L^2_{k+1,\epsilon''}(T^q_s\hat{M})$ where $0<\epsilon''<\epsilon'$. But then we are done, since there is no exceptional weight between $\epsilon$ and $\epsilon''$, and so \cite[Lem 7.1]{MR837256} says that $\text{Ker}_\epsilon A_\infty=\text{Ker}_{\epsilon'}A_\infty$. Notice that this tells us that the function $\mu\mapsto \text{dim Ker}_\mu A_\infty$ is upper semi-continuous at zero.

Since $\epsilon>0$
\[
 \text{Ker}_{\lambda-m}\, A^*_\infty\subseteq \text{Ker}_{-\epsilon+\lambda-m}\, A^*_\infty.
\]
The above argument also shows that the function $\mu \mapsto \text{dim Ker}_\mu A^*_\infty$ is upper semi-continuous (in particular at $\mu=\lambda-m$) and so the set
\[
 \text{Ker}_{-\epsilon+\lambda-m} \, A^*_\infty\backslash \text{Ker}_{\lambda-m}\, A^*_\infty,
\]
is nonempty, but its elements are exactly the limiting sections of the extended $L^2$-sections of $T^{q'}_{s'}\hat{M}$.
Therefore
\[
 \text{dim Ker}_{-\epsilon+\lambda-m}\, A^*_\infty \backslash \text{Ker}_{\lambda-m}\, A^*_\infty =h_\infty(T^{q'}_{s'}\hat{M}),
\]
i.e., exactly the dimension of the space of limiting sections of extended $L^2$-sections of $T^{q'}_{s'}\hat{M}$. This allows us to deduce that
\[
 \text{Ker}_0 \, A_\infty -\text{Ker}_{\lambda-m}\, A_\infty-h_\infty(T^{q'}_{s'}(\hat{M})) =\text{ind}_\epsilon \, A_\infty.
\]
Applying this to \eqref{eqn:csind2} we find that
\begin{equation}\label{eqn:csind3}
 \text{ind}_{\epsilon}\, A_\infty=\int_{\hat{M}} \alpha_0(x) dx -\frac{h+\eta(0)}{2},
\end{equation}
as claimed.
\end{proof}

\subsubsection{An application of the APS index theorem}

Having discussed in the previous section the set of exceptional weights $\mathcal{D}$ for the operator \eqref{eqn:cxopbad} in more detail, we will apply the Atiyah--Patodi--Singer index theorem to the operator $\infop$ to compare the dimension of the space of CS complex deformations of a CS complex surface in a Calabi--Yau four-fold to what we might expect to be the dimension of the space of all complex deformations of the complex surface from Kodaira's theorem \cite[Theorem 1]{MR0133841}.
\begin{thm}\label{thm:indcomp}
 Let $N$ be a CS complex surface at $\hat{x}$ with cone $C$ and rate $\mu\in (1,2)\backslash\mathcal{D}$, where $\mathcal{D}$ is the set of exceptional weights defined in Proposition \ref{prop:lmco}, inside a Calabi-Yau four-fold $M$. Write $\hat{N}:=N\backslash\{\hat{x}\}$. Let, for $k>4/p+1$,
\begin{equation}\label{eqn:infopaps}
 \infop:L^p_{k+1,\mu}(\nu^{1,0}_{M}(\hat{N})\oplus \Lambda^{0,2}\hat{N}\otimes \nu^{1,0}_M(\hat{N}))\to L^p_{k,\mu-1}(\Lambda^{0,1}\hat{N}\otimes \nu^{1,0}_M(\hat{N})),
\end{equation}
and denote the index of this operator by 
\[
 \textnormal{ind}_\mu(\infop).
\]
Then
\begin{equation}\label{eqn:indcomp}
 \chi(N,\nu^{1,0}_M(N))=\textnormal{ind}_{\mu}(\infop)+\sum_{\lambda\in(0,\mu)\cap\mathcal{D}}d(\lambda)+\frac{d(0)+\eta(0)}{2},
\end{equation}
where $\chi(N,\nu^{1,0}_M(N))$ is the holomorphic Euler characteristic of $\nu^{1,0}_M(N)$, $\mathcal{D}$ is the set of $\lambda\in \mathbb{R}$ for which \eqref{eqn:evprobcxlink1}-\eqref{eqn:evprobcxlink2} has a nontrivial solution and then $d(\lambda)$ is the dimension of the solution space, $\eta$ is the $\eta$-invariant which we can now define to be
\begin{equation}\label{eqn:neweta}
 \eta(s):=\sum_{0\ne\lambda\in  \mathcal{D}} d(\lambda)\frac{\textnormal{sign}(\lambda)}{|\lambda|^{s}}.
\end{equation}
\end{thm}
\begin{rem}
 We interpret this as follows. The term $\chi(N,\nu^{1,0}_M(N))$ is interpreted as the dimension of the space of all complex deformations of $N$ in $M$, since this is what we can expect if Kodaira's theorem \cite[Theorem 1]{MR0133841} remains valid for complex varieties. Theorem \ref{thm:cscxcaydefs} tells us that $\text{ind}_\mu(\infop)$ is the expected dimension of the space of CS Cayley deformations of $N$ in $M$ (which by Proposition \ref{prop:nocaydefs} we can interpret as the expected dimension of the space of CS complex deformations of $N$ in $M$, although Theorem \ref{thm:cxcsdefs} tells us that in fact this should be equal to just the dimension of the kernel of \eqref{eqn:infopaps}, which is what we expect to happen generically anyway). The term $d(1)$ represents deformations of $N$ that have a different tangent cone to $N$ at $\hat{x}$.
 \end{rem}
\begin{proof}
 This follows from Proposition \ref{prop:apscs}, since in this case
\[
 \int_N \alpha_0(x) \text{ vol}=\chi(N,\nu^{1,0}_M(N)),
\]
from \cite[Thm 1.6]{MR0290318}.
\end{proof}

\section{Calculations}\label{sec:calc}

In this section we will calculate some of the quantities studied in this article for some examples.

In Section \ref{sec:conedefs}, we will consider deformations of two-dimensional complex cones in $\mathbb{C}^4$, both as a Cayley submanifold and a complex submanifold of $\mathbb{C}^4$. In particular, we will consider Cayley deformations of the cone that are themselves cones. The (real) link of such a complex cone is an associative submanifold of $S^7$ with its nearly parallel $G_2$-structure inherited from the Euclidean $Spin(7)$-structure on $\mathbb{C}^4$, and so deforming the cone as a complex or Cayley cone in $\mathbb{C}^4$ is equivalent to deforming the link of the cone as an associative submanifold. Homogeneous associative submanifolds of $S^7$ were classified by Lotay \cite{MR3004102}, using the classification of homogeneous submanifolds of $S^6$ of Mashimo \cite{MR794670}. The deformation theory of these submanifolds was studied by Kawai \cite{MR3672213}, who explicitly calculated the dimension of the space of infinitesimal associative deformations of these explicit examples using techniques from representation theory. Motivated by these calculations, in Section \ref{ss:calc} we will apply the analysis of the earlier sections to compute the dimension of the space of infinitesimal Cayley conical deformations of the complex cones with these links, and check that these calculations match. We will be able to see explicitly which infinitesimal deformations correspond to complex deformations of the cone and which are Cayley but not complex deformations. In particular we will see that complex infinitesimal deformations and Cayley infinitesimal deformations of a two-dimensional complex submanifold of a Calabi--Yau four-fold are not the same in general. Finally, in Section \ref{ss:etacalc} we will compute the $\eta$-invariant for a complex cone in $\mathbb{C}^4$.

\subsection{Cone deformations}\label{sec:conedefs}

Let $C$ be a two-dimensional complex cone in $\mathbb{C}^4$. Let $v$ be a normal vector field on $C$. If $v$ is sufficiently small, we can apply the tubular neighbourhood theorem for cones \ref{prop:tubnbhdcone} to identify $v$ with a deformation of $C$. Write $v=v_1\oplus v_2$, where $v_1\in C^\infty(\nu^{1,0}_{\mathbb{C}^4}(C))$ and $v_2\in C^\infty(\nu^{0,1}_{\mathbb{C}^4}(C))$. We know from Proposition \ref{prop:caylin2} that $v$ is an infinitesimal Cayley deformation of $C$ if, and only if,
\[
 \bar\partial v_1+\frac{1}{4}\bar\partial^*(v_2\hook\overline{\Omega}_0^\sharp)=0,
\]
where $\Omega_0$ is the standard holomorphic volume form on $\mathbb{C}^4$ and $\sharp$ denotes the musical isomorphism $\nu^{*0,1}_{\mathbb{C}^4}(C)\to \nu^{1,0}_{\mathbb{C}^4}(C)$. Moreover by Proposition \ref{prop:compdefs} $v$ is an infinitesimal complex deformation of $C$ if, and only if,
\[
 \bar\partial v_1= 0 = \bar\partial^*(v_2\hook \overline{\Omega}_0^\sharp).
\]
We would like to know what properties $v$ must have in order for the deformation of $C$ corresponding to $v$ to be a cone itself. By Proposition \ref{prop:tubnbhdcone}, in which we constructed the tubular neighbourhood of a cone, we constructed a map
\[
 \Xi_C:V_C\to T_C,
\]
where $V_C\subseteq \nu_{\mathbb{R}^8}(C)$ contains the zero section and $T_C\subseteq \mathbb{C}^4$ contains $C$. We constructed an action of $\mathbb{R}_+$ on $\nu_{\mathbb{C}^4}(C)$ satisfying $|t\cdot v|=t|v|$, and the map $\Xi_C$ satisfies
\[
 \Xi_C(tr,l,tr\cdot v(r,l))=t\Xi_C(r,l,v(r,l)).
\]
Therefore, to guarantee that $\Xi_C\circ v$ is a cone in $\mathbb{C}^4$, we must have that $v(r,l)=r\cdot \hat{v}(l)$, for some $\hat{v}\in C^\infty(\nu_{S^7}(L))$. In this case,
\[
 \Xi_C(r,l,v(r,l))=r\Xi_C(1,l,\hat{v}(l)),
\]
for all $r\in \mathbb{R}_+$. Choosing a metric on $\nu_{\mathbb{C}^4}(C)$ that is independent of $r$, we see that $r\cdot \hat{v}(l)=r\hat{v}(l)$.

Therefore the dimension of the space of infinitesimal conical Cayley deformations of $C$ is equal to the dimensions of the spaces of solutions to the eigenproblems \eqref{eqn:eprobcay} and \eqref{eqn:cxev1}-\eqref{eqn:cxev2} with $\lambda=1$. As remarked after the statement of Proposition \ref{prop:evalprobcay}, this particular eigenspace can be identified with the space of infinitesimal associative deformations of the link of the cone in $S^7$ with its nearly parallel $G_2$-structure. This problem was studied by Kawai \cite{MR3672213}, who computed the dimension of these spaces for a range of examples. In terms of the work done here, this is equivalent to solving the eigenproblem \eqref{eqn:eprobcay} when $\lambda=1$. We will study the eigenproblem \eqref{eqn:cxev1}-\eqref{eqn:cxev2} for the three examples of complex cones that were studied by Kawai in his paper. Our analysis will allow us to see directly the difference between the infinitesimal conical Cayley and complex deformations of a cone, and we hope that the complex geometry will make these calculations simpler.

\subsubsection{Example 1: \texorpdfstring{$L_1=S^3$}{L2inS3}}\label{ss:ex1}

The first example is the simplest, being just a vector subspace (with the zero vector removed). We take 
\[
 C_1:=\mathbb{C}^2\backslash\{0\}, \quad L_1:=S^3, \quad \Sigma_1:=\mathbb{C}P^1,
\]
where $C_1$ is the complex cone, $L_1$ is the real link of $C_1$ and $\Sigma_1$ is the complex link of $C_1$.
\begin{prop}[{\cite[\S6.4.1]{MR3672213}}]\label{prop:ex1}
 The space of infinitesimal associative deformations of $L_1$ in $S^7$ has dimension twelve.  
\end{prop}

\subsubsection{Example 2: \texorpdfstring{$L_2\cong SU(2)/ \mathbb{Z}_2$}{L2}}\label{ss:ex2}

Our second example is a little less trivial. Take
\[
 C_2:=\{(z_1,z_2,z_3,z_4)\in \mathbb{C}^4\,|\, z_4=0, z_1^2+z_2^2+z_3^2=0\}.
\]
Then it can be shown \cite[Ex 6.6]{MR3672213} that the link of $C_2$, $L_2$, is isomorphic to the quotient group $SU(2)/\mathbb{Z}_2$.

The complex link of $C_2$ is
\[
 \Sigma_2:=\{[z_0:z_1:z_2:z_3]\in\mathbb{C}P^3\,|\, z_0=0, z_1^2+z_2^2+z^2_3=0\}.
\]
\begin{prop}[{\cite[Cor 5.12]{MR2981841}, \cite[Prop 6.26]{MR3672213}}]\label{prop:ex2}
 The space of infinitesimal associative deformations of $L_2$ in $S^7$ has dimension twenty-two.
\end{prop}

\subsubsection{Example 3: \texorpdfstring{$L_3\cong SU(2)/\mathbb{Z}_3$}{L3}}\label{ss:ex3}
Our third example is the most complicated to state, but is certainly the most interesting.

Define the cone $C_3$ to be the cone over the submanifold $L_3$ of $S^7$ which is defined as follows: consider the following action of $SU(2)$ on $\mathbb{C}^4$
\[
 \begin{pmatrix}
\,z_1 \, \\
z_2 \\
z_3 \\
z_4
 \end{pmatrix}
\mapsto
\begin{pmatrix}
 a^3z_1+\sqrt{3}a^2b z_2+\sqrt{3}ab^2 z_3+b^3z_4 \\
-\sqrt{3}a^2\bar{b}z_1+a(|a|^2-2|b|^2)z_2+b(2|a|^2-|b|^2)z_3+\sqrt{3}\bar{a}b^2z_4 \\
\sqrt{3}a\bar{b}^2z_1-\bar{b}(2|a|^2-|b|^2)z_2+\bar{a}(|a|^2-2|b|^2)z_3+\sqrt{3}\bar{a}^2bz_4 \\
-\bar{b}^3z_1+\sqrt{3}\bar{a}\bar{b}^2z_2-\sqrt{3}\bar{a}^2\bar{b}z_3+\bar{a}^3z_4
\end{pmatrix},
\]
where $a,b\in\mathbb{C}$ satisfy $|a|^2+|b|^2=1$. We define $L_3$ to be the orbit of the above action around the point $(1,0,0,0)^T$, that is,
\[
L_3:=
 \begin{pmatrix}
  a^3 \\
-\sqrt{3}a^2\bar{b} \\
\sqrt{3} a\bar{b}^2 \\
-\bar{b}^3
 \end{pmatrix},
\]
where $a,b\in\mathbb{C}$ satisfy $|a|^2+|b|^2=1$. We see that for
\[
 \mathbb{Z}_3:=\left\{
\begin{pmatrix}
\,\zeta\, & 0 \\
0 & \,\bar{\zeta}\, 
\end{pmatrix}
\in SU(2) \,|\, \zeta^3=1 \right\},
\]
$L_3$ is invariant under the action of $\mathbb{Z}_3$, therefore $L_3\cong SU(2)/\mathbb{Z}_3$. The complex link of the cone $C_3$ over $L_3$ is
\[
 \Sigma_3:= \{[x^3:\sqrt{3}x^2y:\sqrt{3}xy^2:y^3]\in \mathbb{C}P^3\,|\, [x:y]\in \mathbb{C}P^1\},
\]
which is known as the \emph{twisted cubic} in $\mathbb{C}P^3$.

This is a particularly interesting example for the following reason \cite[Ex 5.8]{MR3004102}. Define $L_3(\theta)$ to be the orbit of the above group action around the point $(\cos\theta,0,0,\sin\theta)^T$. Then $L_3(\theta)$ is associative for $\theta\in [0,\frac{\pi}{4}]$. As noted above, $L_3(0)=L_3$ is the real link of a complex cone, however, $L_3(\frac{\pi}{4})$ is the link of a special Lagrangian cone. Therefore there exists a family of Cayley cones in $\mathbb{C}^4$, including both a complex cone and a special Lagrangian cone, that are related by a group action.

\begin{prop}[{\cite[\S6.3.2]{MR3672213}}]\label{prop:ex3}
 The space of infinitesimal associative deformations of $L_3(\frac{\pi}{4})$ in $S^7$ has dimension thirty.
\end{prop}

\subsection{Calculations}\label{ss:calc}

We will now study the eigenvalue problem \eqref{eqn:cxev1}-\eqref{eqn:cxev2} with $\lambda=1$ for $C_1$, $C_2$ and $C_3$ defined above. Recall that by Proposition \ref{prop:evcxlink} we can study the eigenproblem \eqref{eqn:evprobcxlink1}-\eqref{eqn:evprobcxlink2} with $\lambda=1$ on the complex link instead to make our calculations easier. We first explain how to count infinitesimal conical complex deformations and infinitesimal conical Cayley but non complex deformations of a complex cone.

\begin{prop}\label{prop:caycxdef}
 Let $C$ be a complex cone in $\mathbb{C}^4$ with real link $L$ and complex link $\Sigma$. Infinitesimal complex conical deformations of $C$ in $\mathbb{C}^4$ are given by holomorphic sections of $\nu^{1,0}_{\mathbb{C}P^3}(\Sigma)$. Infinitesimal Cayley conical deformations of $C$ that are not complex are given by $v\in C^\infty(\nu^{1,0}_{\mathbb{C}P^3}(\Sigma)\otimes\mathcal{O}_{\mathbb{C}P^3}(m)|_{\Sigma})$ satisfying
\begin{equation}\label{eqn:caylapev}
 \Delta_{\bar\partial_\Sigma} v=-\frac{1}{2}m(4+m)v,
\end{equation}
where $-4< m<0$.
\end{prop}
\begin{proof}
 We know that infinitesimal complex deformations $C$ will lie in the kernel of $\bar\partial_C$ or $\bar\partial^*_C$. Recall that these spaces are isomorphic and so we expect them to have the same dimension. Examining the proof of Proposition \ref{prop:evalprobcx} and comparing to Proposition \ref{prop:evcxlink}, we see that infinitesimal complex deformations of $C$ are given by holomorphic sections of $\nu^{1,0}_{\mathbb{C}P^3}(\Sigma)\otimes \mathcal{O}_{\mathbb{C}P^3}(\lambda-1)|_\Sigma$, and antiholomorphic sections of $\Lambda^{0,1}\Sigma\otimes \nu^{1,0}_{\mathbb{C}P^3}(\Sigma)\otimes \mathcal{O}_{\mathbb{C}P^3}(-3-\lambda)$. Since infinitesimal conical deformations of $C$ will correspond to $\lambda=1$ here, we see that infinitesimal complex conical deformations of $C$ correspond to holomorphic sections of
\[
 \nu^{1,0}_{\mathbb{C}P^3}(\Sigma),
\]
and antiholomorphic sections of
\[
 \Lambda^{0,1}\Sigma\otimes \nu^{1,0}_{\mathbb{C}P^3}(\Sigma)\otimes \mathcal{O}_{\mathbb{C}P^3}(-4)|_{\Sigma}\cong \nu^{*1,0}_{\mathbb{C}P^3}(\Sigma),
\]
by the adjunction formula \cite[Prop 2.2.17]{MR2093043} since $K_{\mathbb{C}P^3}|_{\Sigma}=\mathcal{O}_{\mathbb{C}P^3}(-4)|_{\Sigma}$. So we see that infinitesimal conical complex deformations of $C$ arise from holomorphic sections of the holomorphic normal bundle of the complex link in $\mathbb{C}P^3$. The dimension of the space of infinitesimal conical complex deformations of $C$ is then equal to the real dimension (or twice the complex dimension) of the space of holomorphic sections of the holomorphic normal bundle of the complex link.
 
Finally, we see that any remaining infinitesimal conical Cayley deformations of $C$ must satisfy the eigenproblem \eqref{eqn:evprobcxlink1}-\eqref{eqn:evprobcxlink2} with $\lambda=1$ and $m\ne 0,-4$. Applying $\bar\partial^*_{\Sigma}$ to \eqref{eqn:evprobcxlink1} and using \eqref{eqn:evprobcxlink2}, we see that the remaining infinitesimal conical Cayley deformations of $C$ are given by  $v\in C^\infty(\nu^{1,0}_{\mathbb{C}P^3}(\Sigma)\otimes \mathcal{O}_{\mathbb{C}P^3}(m)|_\Sigma)$ satisfying
\[
 \Delta_{\bar\partial_\Sigma} v =-\frac{1}{2} m(4+m)v.
\]
\end{proof}

While we can apply the Hirzebruch--Riemann--Roch theorem \ref{thm:rr} to count holomorphic sections of holomorphic vector bundles, solving eigenproblems for the Laplacian acting sections of vector bundles such as \eqref{eqn:caylapev} is somewhat more difficult, especially since the degree of the line bundle we consider appears in the eigenvalue itself. Such problems have been studied, however, and we will make use of the following result of L\'opez Almorox and Tejero Prieto on eigenvalues of the $\bar\partial_\Sigma$-Laplacian acting on sections of holomorphic line bundles over $\mathbb{C}P^1$ equipped with a metric of constant scalar curvature.
\begin{thm}[{\cite[Thm 5.1]{MR2241738}}]\label{thm:evtwistdirac}
 Let $K$ be a Hermitian line bundle over $\Sigma$, where $\Sigma$ is $\mathbb{C}P^1$ with metric of constant scalar curvature $\kappa$ equipped with a unitary harmonic connection $\nabla_K$ of curvature $F^{\nabla_K}=-iB\omega_\Sigma$ for some $B\in\mathbb{R}$. Then the spectrum of the operator
\[
 2\bar\partial^*_\Sigma \bar\partial_\Sigma:C^\infty(K)\to C^\infty(K),
\]
 is the set
\[
 \left\{\lambda_q=\frac{\kappa}{2}\left[(q+a)^2+(q+a)|\textnormal{deg }K+1|\right] \,|\, q\in \mathbb{N}\cup \{0\}\right\},
\]
where $a=0$ if $\textnormal{deg }K\ge 0$, $a=1$ otherwise.

The space of eigensections of $2\bar\partial^*_\Sigma\bar\partial_\Sigma$ with eigenvalue $\lambda_q$ is identified with the space of holomorphic sections of
\[
 K^{-q}_\Sigma\otimes K,
\]
when $\textnormal{deg } K\ge 0$, or of holomorphic sections of 
\[
 K^{-q}_\Sigma\otimes K^{-1},
\]
when $\textnormal{deg } K <0$.
Therefore the multiplicity of  $\lambda_q$ is
\[
 m(\lambda_q)=1+|\textnormal{deg } K|+2q.
\]
\end{thm}

\subsubsection{Example 1: \texorpdfstring{$L_1=S^3$}{l1s3}}

To calculate the dimension of the space of infinitesimal conical Cayley deformations of the cone $C_1=\mathbb{C}^2$, which as real link $L_1=S^3$ and complex link $\Sigma_1=\mathbb{C}P^1$, we will apply Proposition \ref{prop:caycxdef}. We first calculate the dimension of the space of holomorphic sections of
\[
 \nu^{1,0}_{\mathbb{C}P^3}(\Sigma_1)=\mathcal{O}_{\mathbb{C}P^3}(1)|_\Sigma\oplus \mathcal{O}_{\mathbb{C}P^3}(1)|_\Sigma,
\]
which by the Hirzebruch--Riemann--Roch theorem \ref{thm:rr} has dimension four. Therefore, the dimension of the space of infinitesimal conical complex deformations of $C_1$ is eight.

Now we study the eigenproblem
\begin{equation}\label{eqn:lapex1}
 \Delta_{\bar\partial_\Sigma}v =-\frac{1}{2}m(4+m),
\end{equation}
for $v\in C^\infty(\nu^{1,0}_{\mathbb{C}P^3}(\Sigma_1)\otimes \mathcal{O}_{\mathbb{C}P^3}(m)|_\Sigma)=C^\infty(\mathcal{O}_{\mathbb{C}P^3}(m+1)|_\Sigma\oplus\mathcal{O}_{\mathbb{C}P^3}(m+1)|_\Sigma)$ and $-4<m<0$. We can apply Theorem \ref{thm:evtwistdirac} to solve \eqref{eqn:lapex1} as long as the connection on $\mathcal{O}_{\mathbb{C}P^3}(m+1)|_\Sigma\oplus\mathcal{O}_{\mathbb{C}P^3}(m+1)|_\Sigma$ takes the form
\[
 \begin{pmatrix}
  \nabla_1 & 0 \\
0 &\nabla_2
 \end{pmatrix},
\]
where $\nabla_i$ are connections on $\mathcal{O}_{\mathbb{C}P^3}(m+1)|_\Sigma$. This is the case here, as can be seen from the relation between the connection on the normal bundle of $\Sigma_1$ in $\mathbb{C}P^3$ and the connection on the normal bundle of $L_1$ in $S^7$ (see \cite[Lem 1]{MR0200865}) and the fact that the normal bundle of $L_1$ in $S^7$ is trivial.

Therefore, by Theorem \ref{thm:evtwistdirac}, solving \eqref{eqn:lapex1} reduces to solving the algebraic equation
\[
 -m(4+m)=4((q+a)^2+(q+a)|m+2|),
\]
for $m\in \mathbb{Z}$ and $q\in \mathbb{N}\cup \{0\}$ (since the scalar curvature of $\Sigma_1$ is eight) with $a=0$ if $m\ge -1$ and $a=1$ if $m\le -2$. It can be checked that this has solution $(q,a,m)=(0,1,-2)$, and so by Theorem \ref{thm:evtwistdirac} the dimension of eigensections of \eqref{eqn:lapex1} has dimension $2\times 2=4$. So we have a total of twelve infinitesimal conical Cayley deformations of $C$ in $\mathbb{C}^4$.

We sum this up in a proposition.
\begin{prop}\label{prop:ex1soln}
 The real dimension of the space of infinitesimal conical Cayley deformations of $C_1$ in $\mathbb{C}^4$ is twelve. The real dimension of the space of infinitesimal conical complex deformations of $C_1$ in $\mathbb{C}^4$ is eight. 
\end{prop}
\begin{rem}
Recall that the stabiliser of a Cayley plane in $\mathbb{R}^8$ is isomorphic to $(SU(2)\times SU(2)\times SU(2))/\mathbb{Z}_2$ and that the dimension of $Spin(7)/((SU(2)\times SU(2)\times SU(2))/\mathbb{Z}_2)$ is twelve. The stabiliser of a two-dimensional complex plane in $\mathbb{C}^4$ is isomorphic to $U(2)\times U(2)$, and the dimension of $U(4)/(U(2)\times U(2))$ is equal to eight.
\end{rem}

\subsubsection{Example 2: \texorpdfstring{$L_2\cong SU(2)/\mathbb{Z}_2$}{l2}}

We now calculate the dimension of the space of infinitesimal conical Cayley deformations of the cone $C_2$ in $\mathbb{C}^4$ with link $L_2\cong SU(2)/\mathbb{Z}_2$ and complex link $\Sigma_2$ as defined in Section \ref{ss:ex2}. Again by Proposition \ref{prop:caycxdef} we compute the dimension of the space of holomorphic sections of 
\[
 \nu^{1,0}_{\mathbb{C}P^3}(\Sigma_2)=\mathcal{O}_{\mathbb{C}P^3}(1)|_{\Sigma}\oplus\mathcal{O}_{\mathbb{C}P^3}(2)|_\Sigma,
\]
which by the Hirzebruch--Riemann--Roch theorem \ref{thm:rr} has dimension eight, and so we deduce that the space of infinitesimal conical complex deformations of $C_2$ has dimension sixteen.

Again since the normal bundle of $L_2$ in $S^7$ is trivial we may apply Theorem \ref{thm:evtwistdirac} to solve the eigenproblem
\begin{equation}\label{eqn:lapex2}
 \Delta_{\bar\partial_\Sigma} v=-\frac{1}{2}m(m+4),
\end{equation}
for $v\in C^\infty(\mathcal{O}_{\mathbb{C}P^3}(m+1)|_\Sigma\oplus \mathcal{O}_{\mathbb{C}P^3}(m+2)|_\Sigma)$ with $-4<m<0$. This reduces again to solving the equations for $m\in \mathbb{Z}$ and $q\in \mathbb{N}\cup \{0\}$
\[
 -m(m+4)=2((q+a)^2+(q+a)|2m+3|),
\]
with $a=0$ for $m\ge -1$ and $a=1$ otherwise, which has solution $(q,a,m)=(0,1,-2)$ and
\[
 -m(m+4)=2((q+a)^2+(q+a)|2m+5|),
\]
with $a=0$ for $m\ge -2$  and $a=1$ otherwise, which has solution $(q,a,m)=(1,0,-2)$. Therefore by Theorem \ref{thm:evtwistdirac} the dimension of the space of solutions to \eqref{eqn:lapex2} has dimension $3+3=6$. Therefore, the dimension of the space of infinitesimal conical Cayley deformations of $C_2$ in $\mathbb{C}^4$ is twenty-two.

\begin{prop}\label{prop:ex2soln}
 The real dimension of the space of infinitesimal conical Cayley deformations of $C_2$ in $\mathbb{C}^4$ is twenty-two. The real dimension of the space of infinitesimal conical complex deformations of $C_2$ in $\mathbb{C}^4$ is sixteen.
\end{prop}
\begin{rem}
The dimension of $Spin(7)/SU(4)$ is six, which implies that the six Cayley but not complex infinitesimal conical deformations of $C_2$ are just rigid motions induced by the action of $Spin(7)$ on $\mathbb{R}^8$.
\end{rem}
\subsubsection{Example 3: \texorpdfstring{$L_3\cong SU(2)/\mathbb{Z}_3$}{l3}}

Finally, we compute the dimension of the space of infinitesimal conical Cayley deformations of $C_3$ in $\mathbb{C}^4$, which has real link $L_3\cong SU(2)/\mathbb{Z}_3$ and complex link $\Sigma_3$ as defined in Section \ref{ss:ex3}.

The dimension of the space of holomorphic sections of
\[
 \nu^{1,0}_{\mathbb{C}P^3}(\Sigma_3)=\mathcal{O}_{\Sigma_3}(5)\oplus \mathcal{O}_{\Sigma_3}(5),
\]
where $\mathcal{O}_{\Sigma_3}(n)$ denotes the line bundle of degree $n$ over $\Sigma_3$. By Hirzebruch--Riemann--Roch, Theorem \ref{thm:rr}, this space has dimension twelve, and so the dimension of the space of infinitesimal conical complex deformations of $C_3$ in $\mathbb{C}^4$ has dimension twenty-four.

So it remains to find $v\in C^\infty(\mathcal{O}_{\Sigma_3}(3m+5)\oplus \mathcal{O}_{\Sigma_3}(3m+5))$ satisfying
\begin{equation}\label{eqn:lapex3}
 \Delta_{\bar\partial_\Sigma}v=-\frac{1}{2}m(4+m).
\end{equation}
Unfortunately, for this example we cannot directly apply Theorem \ref{thm:evtwistdirac} to this problem, so we must find a different way to solve \eqref{eqn:lapex3}. We will do this by constructing a moving frame for $L_3$.

\begin{prop}[{\cite[\S 6.3.2]{MR3672213}}]\label{prop:frame2}
 There exists an orthonormal frame of $L_3$, denoted $\{e_1,e_2,e_3\}$, where $Je_2=e_3$ and $e_1$ is the Reeb vector field. We have that
\[
 [e_1,e_2]=-\frac{2}{3}e_3, \quad [e_1,e_3]=\frac{2}{3}e_2, \quad [e_2,e_3]=-2e_1.
\]
\end{prop}
We extend this to a frame of $S^7$ as follows.
\begin{lem}\label{lem:frame}
 There exist orthonormal frames $\{e_1,e_2,e_3\}$ of $L_3$ and $\{f_4,f_5,$ $f_6,f_7\}$ of $\nu_{S^7}(L_3)$ such that the structure equations of Proposition \ref{prop:jase1} take the following form:
\begin{align*}
 dx&=e_1\omega_1+e_2\omega_2+e_3\omega_3+f_4\eta_4+f_5\eta_5+f_6\eta_6+f_7\eta_7, \\
de_1&=-\omega_1x-\omega_3e_2+\omega_2e_3-\eta_5 f_4+\eta_4 f_5-\eta_7f_6+\eta_6f_7, \\
de_2&=-\omega_2x+\omega_3e_1+\frac{\omega_1}{3}e_3+\frac{2}{\sqrt{3}}\omega_2f_4+\frac{2}{\sqrt{3}}\omega_3 f_5, \\
de_3&=-\omega_3 x-\omega_2 e_1-\frac{\omega_1}{3}e_2-\frac{2}{\sqrt{3}}\omega_3f_4+\frac{2}{\sqrt{3}}\omega_2f_5, \\
df_4&=-x\eta_4+\eta_5e_1-\frac{2}{\sqrt{3}}\omega_2e_2+\frac{2}{\sqrt{3}}\omega_3e_3-\frac{\omega_1}{3}f_5+\omega_2f_6+\omega_3f_7, \\
df_5&=-x\eta_5-\eta_4e_1-\frac{2}{\sqrt{3}}\omega_3e_2-\frac{2}{\sqrt{3}}\omega_2e_3+\frac{\omega_1}{3}f_4-\omega_3f_6+\omega_2f_7, \\
df_6&=-x\eta_6+\eta_7e_1-\omega_2f_4+\omega_3f_5-\omega_1f_7, \\
df_7&=-x\eta_7-\eta_6e_1-\omega_3f_4-\omega_2f_5+\omega_1f_6,
\end{align*}
where $Je_2=e_3, Jf_4=f_5, Jf_6=f_7$, $\{\omega_1,\omega_2,\omega_3\}$ is an orthonormal coframe of $L_3$ ($\omega_i(e_j)=\delta_{ij}$) and $\{\eta_4,\eta_5,\eta_6,\eta_7\}$ is an orthonormal coframe of the normal bundle of $L_2$ in $S^7$ ($\eta_a(f_b)-\delta_{ab}$). Further, the second structure equations of Proposition \ref{prop:jase2} are also satisfied.
\end{lem}
\begin{proof}
 Let $\nabla$ denote the Levi-Civita connection of $L_3$. Again we take $\alpha_2=\omega_2$ and $\alpha_3=\omega_3$ as we may by Proposition \ref{prop:jase3}. We see that since, using the structure equations given in \ref{prop:jase1},
\[
-\alpha_1(e_1)e_3-e_3 =\nabla_{e_1}e_2-\nabla_{e_2}e_1=[e_1,e_2]=-\frac{2}{3}e_3,
\]
we must have that $\alpha_1=-\frac{\omega_1}{3}$. We check that
\[
 -\frac{1}{3}e_2+e_2 =\nabla_{e_1}e_3-\nabla_{e_3}e_1=[e_1,e_3]=\frac{2}{3}e_2,
\]
and
\[
 -e_1-e_1=\nabla_{e_2}e_3-\nabla_{e_3}e_2=[e_2,e_3]=-2e_1.
\]
Now Equation \eqref{eqn:struc23} tells us that we must have that
\[
 -2(\beta^4_2\wedge \beta^5_2+\beta^6_2\wedge \beta^7_2)=-\frac{8}{3}\omega_2\wedge \omega_3.
\]
So we take $\beta^4_2=\frac{2}{\sqrt{3}}\omega_2$ and $\beta^5_2=\frac{2}{\sqrt{3}}\omega_3$, $\beta^6_2=\beta^7_2=0$ and this is satisfied.
To ensure that Equation \eqref{eqn:struc24} is satisfied, we seek $\gamma$ so that
\begin{align*}
 d\beta^4_2&=\frac{2}{\sqrt{3}}d\omega_2=-\frac{4}{3\sqrt{3}}\omega_1\wedge \omega_3=-\frac{2}{\sqrt{3}}\omega_1\wedge \omega_3+\frac{1}{\sqrt{3}}\gamma_1\wedge \omega_3, \\
d\beta^5_2&=\frac{2}{\sqrt{3}}d\omega_3=\frac{4}{3\sqrt{3}}\omega_1\wedge \omega_2=\frac{2}{\sqrt{3}}\omega_1\wedge \omega_2-\frac{1}{\sqrt{3}}\gamma_1\wedge \omega_2, \\
d\beta^6_2&=0=\frac{1}{\sqrt{3}}\gamma_3\wedge \omega_3-\frac{1}{\sqrt{3}}\gamma_2\wedge \omega_2, \\
d\beta^7_2&=0=-\frac{1}{\sqrt{3}}\gamma_3\wedge\omega_2-\frac{1}{\sqrt{3}}\gamma_2\wedge \omega_2.
\end{align*}
From this we see that we must have that $\gamma_1=\frac{2}{3}\omega_1$, and $\gamma_2=a\omega_2$ and $\gamma_3=a\omega_3$. To determine $a$, we check Equation \eqref{eqn:struc25}, which tells us that we must have
\[
 -\frac{1}{3}d\omega_1=-\frac{2}{3}\omega_2\wedge \omega_3=\frac{a^2}{2}\omega_2\wedge \omega_3-\frac{8}{3}\omega_2\wedge \omega_3,
\]
and therefore we must have $a=2$. It can be checked that the remaining parts of Equation \eqref{eqn:struc25} are satisfied with $\gamma=(\frac{2}{3}\omega_1,2\omega_2,2\omega_3)$. Therefore we choose $\{f_4,f_5,f_6,f_7\}$ so that the above choices of $\gamma,\beta$ and $\alpha$ hold, and so the equations claimed hold.
\end{proof}

We have that $\{f_4-if_5,f_6-if_7\}$ is a frame for the holomorphic tangent bundle of $L_3$ in $S^7$. We have that
\begin{align*}
 \nabla_{e_1}(f_4-if_5)&=-\frac{1}{3}f_5-\frac{i}{3}f_4=-\frac{i}{3}(f_4-if_5), \\
\nabla_{e_1}(f_6-if_7)&=-f_7-if_6=-i(f_6-if_7).
\end{align*}
However, 
\begin{align*}
 (\nabla_{e_2}^\perp+i\nabla_{e_3}^\perp)(f_4-if_5)&=0, \\
(\nabla_{e_2}^\perp+i\nabla_{e_3}^\perp)(f_6-if_7)&=-2(f_4-if_5), \\
(\nabla_{e_2}^\perp-i\nabla_{e_3}^\perp)(f_4-if_5)&=2(f_6-if_7), \\
(\nabla_{e_2}^\perp-i\nabla_{e_3}^\perp)(f_6-if_7)&=0,
\end{align*}
and so we see explicitly that the connection on the normal bundle of $L_3$ in $S^7$ is not in a nice diagonal form as we had before.
Since we have a moving frame of $S^7$, we will return to considering the eigenvalue problem \eqref{eqn:cxev1}-\eqref{eqn:cxev1}. Writing a section of $\nu^{1,0}_{S^7}(L_3)$ as
\[
 g_1(f_4-if_5)+g_2(f_6-if_7),
\]
where $g_1,g_2$ are functions on $L_3$ and sections of $\Lambda^{0,1}_hL\otimes \nu^{1,0}_{S^7}(L_3)$ as
\[
 \alpha_1\otimes (f_4-if_5)+\alpha_2\otimes (f_6-if_7),
\]
where $\alpha_1,\alpha_2$ are sections of $\Lambda^{0,1}_hL$, we seek $g_1,g_2\in C^\infty(L_3)$ and $\alpha_1,\alpha_2\in C^\infty(\Lambda^{0,1}_hL)$ satisfying
\begin{align*}
 \bar\partial_h g_1-g_2(\omega_2-i\omega_3)&=\left(\frac{8}{3}-i\nabla_{e_1}\right)\alpha_1, \\
\bar\partial^*_h\alpha_1&=\frac{1}{2}\left(\frac{4}{3}+i\nabla_{e_1}\right) g_1,
\end{align*}
and
\begin{align*}
 \bar\partial_h g_2&=(2-i\nabla_{e_1})\alpha_2, \\
\bar\partial^*_h\alpha_2+2(e_2\hook \alpha_1)&=\frac{1}{2}(2+i\nabla_{e_1})g_2.
\end{align*}
We must have that 
\[
 g_2(\omega_2-i\omega_3)=a \alpha_1,
\]
for some $a\in \mathbb{C}$ (since if $\alpha_1=0$ then we find infinitesimal conical complex deformations of $C_3$), and so we may instead study the eigenvalue problems
\begin{align}\label{eqn:ev0}
 \bar\partial_h g_1&=\left(\frac{8}{3}-i\nabla_{e_1}+a\right)\alpha_1, \\\label{eqn:ev00}
\bar\partial^*_h\alpha_1&=\frac{1}{2}\left(\frac{4}{3}+i\nabla_{e_1}\right)g_1,
\end{align}
and
\begin{align}\label{eqn:ev1}
 \bar\partial_h g_2&=(2-i\nabla_{e_1})\alpha_2, \\ \label{eqn:ev2}
\bar\partial^*_h \alpha_2&=\frac{1}{2}\left(2+\frac{4}{a}+i\nabla_{e_1}\right)g_2.
\end{align}
Using the structure equations given in Lemma \ref{lem:frame}, we see that the problem \eqref{eqn:ev1}-\eqref{eqn:ev2} is equivalent to the eigenproblem
\begin{align}\label{eqn:ev3}
 \bar\partial_h(g_2(\omega_2-i\omega_3))&=\left(\frac{8}{3}-i\nabla_{e_1}\right)\alpha_2\otimes (\omega_2-i\omega_3), \\ \label{eqn:ev4}
\bar\partial^*_h(\alpha_2\otimes (\omega_2-i\omega_3))&=\frac{1}{2}\left(\frac{4}{3}+\frac{4}{a}+i\nabla_{e_1}\right)g_2(\omega_2-i\omega_3),
\end{align}
where we consider $g_2(\omega_2-i\omega_3)$ as a $\Lambda^{0,1}_hL$-valued function, which becomes
\begin{align}\label{eqn:ev000}
 a\bar\partial_h \alpha_1&=\left(\frac{8}{3}-i\nabla_{e_1}\right)\alpha_2, \\ \label{eqn:ev0000}
\bar\partial^*_h\alpha_2&=\frac{a}{2}\left(\frac{4}{3}+\frac{4}{a}+i\nabla_{e_1}\right)\alpha_1,
\end{align}
where now $\alpha_2$ is a section of $\Lambda^{0,1}_hL\otimes \Lambda^{0,1}_hL$.
Supposing that
\[
 \mathcal{L}_{e_1}g_1=img_1, \;\mathcal{L}_{e_1}\alpha_1=im\alpha_1,
\]
for $3m\in \mathbb{Z}$ we see that in order for the eigenproblem \eqref{eqn:ev000}-\eqref{eqn:ev0000} to make sense we must have 
\[
 \mathcal{L}_{e_1}\alpha_2=im\alpha_2.
\]
Write $\mathcal{O}_{\Sigma_3}(d)$ for the degree $d$ line bundle over $\Sigma_3$. Then as explained in Section \ref{ss:trick1}, we may replace the eigenvalue problems \eqref{eqn:ev0}-\eqref{eqn:ev00}-\eqref{eqn:ev000}-\eqref{eqn:ev0000} with seeking $g_1\in C^\infty(\mathcal{O}_{\Sigma_3}(3m)),$ and $\alpha_1\in C^\infty(\mathcal{O}_{\Sigma_3}(3m+2)),\alpha_2\in C^\infty(\mathcal{O}_{\Sigma}(3m+4))$ satisfying
\begin{align}\label{eqn:fev1}
 \bar\partial_{\Sigma_3}g_1=\left(\frac{8}{3}+a+m\right)\alpha_1, \\ \label{eqn:fev2}
\bar\partial^*_{\Sigma_3}\alpha_1=\frac{1}{2}\left(\frac{4}{3}-m\right)g_1,
\end{align}
and
\begin{align}\label{eqn:fev3}
 a\bar\partial_{\Sigma_3}\alpha_1=\left(\frac{8}{3}+m\right)\alpha_2, \\ \label{eqn:fev4}
\bar\partial^*_{\Sigma_3}\alpha_2=\frac{a}{2}\left(\frac{4}{3}+\frac{4}{a}-m\right)\alpha_1.
\end{align}
We find that $\alpha_1$ must simultaneously satisfy the following two eigenproblems: applying $\bar\partial_{\Sigma_3}$ to \eqref{eqn:fev2} and using \eqref{eqn:fev1} we find that
\begin{equation}\label{eqn:lap1}
 \bar\partial_{\Sigma_3}\bar\partial^*_{\Sigma_3} \alpha_1 =\frac{1}{2}\left(\frac{8}{3}+a+m\right)\left(\frac{4}{3}-m\right)\alpha_1,
\end{equation}
and applying $\bar\partial_{\Sigma_3}^*$ to \eqref{eqn:fev3} and using \eqref{eqn:fev4} we have that
\begin{equation}\label{eqn:lap2}
 \bar\partial^*_{\Sigma_3}\bar\partial_{\Sigma_3}\alpha_1=\frac{1}{2}\left(\frac{8}{3}+m\right)\left(\frac{4}{3}+\frac{4}{a}-m\right)\alpha_1.
\end{equation}
Applying the formula \cite[Lem 2.1, 2.2]{MR2241738}
\[
 \bar\partial_{\Sigma_3}\bar\partial^*_{\Sigma_3}\alpha=\bar\partial^*_{\Sigma_3}\bar\partial_{\Sigma_3}\alpha+\frac{2}{3}(3m+2)\alpha,
\]
where $\alpha$ is a section of $\mathcal{O}_{\Sigma_3}(3m+2)$, we see that
\begin{align*}
 \bar\partial^*_{\Sigma_3}\bar\partial_{\Sigma_3}\alpha_1&=\frac{1}{2}\left(\frac{8}{3}+m\right)\left(\frac{4}{3}+\frac{4}{a}-m\right)\alpha_1, \\
&=\frac{1}{2}\left[\left(\frac{8}{3}+a+m\right)\left(\frac{4}{3}-m\right)+\frac{4}{3}(3m+2)\right]\alpha_1,
\end{align*}
for $\alpha_1\in C^\infty(\mathcal{O}_{\Sigma_3}(3m+2))$. Therefore $a\in \mathbb{C}$ must satisfy
\[
 \left(\frac{8}{3}+m\right)\left(\frac{4}{3}+\frac{4}{a}-m\right)=\left(\frac{8}{3}+a+m\right)\left(\frac{4}{3}-m\right)-\frac{4}{3}(3m+2).
\]
Solving this equation for $a$, we find that for $m\ne 4/3$
\[
 a_\pm=\frac{4m+\frac{8}{3}\pm 8}{2(\frac{4}{3}-m)},
\]
which simplifies to
\[
 a_+=\frac{6m+16}{4-3m}, \;a_-=-2.
\]
First considering $a=a_+$ we apply Theorem \ref{thm:evtwistdirac} to see that
\[
 \frac{1}{2}\left(\frac{8}{3}+m\right)\left(\frac{4}{3}-m+\frac{4(4-3m)}{6m+16}\right),
\]
is an eigenvalue of $\bar\partial^*_{\Sigma_3}\bar\partial_{\Sigma_3}$ acting on sections of $\mathcal{O}_{\Sigma_3}(3m+2)$ if, and only if, $m=-2/3$. In this case there are five $\alpha_1\in C^\infty(\mathcal{O}_{\Sigma_3}(0))$ satisfying
\[
 \Delta_{\bar\partial_{\Sigma_3}}\alpha_1=4\alpha_1.
\]
Taking $g_1=\bar\partial^*_{\Sigma_3}\alpha_1$ and $\alpha_2=\bar\partial_{\Sigma_2}\alpha_1$ completes this solution to the eigenproblem \eqref{eqn:fev1}-\eqref{eqn:fev2}-\eqref{eqn:fev3}-\eqref{eqn:fev4}.

Secondly, when $a=a_-=-2$ Theorem \ref{thm:evtwistdirac} tells us that
\[
 \frac{1}{2}\left(\frac{8}{3}+m\right)\left(-\frac{2}{3}-m\right),
\]
is an eigenvalue of $\bar\partial^*_{\Sigma_3}\bar\partial_{\Sigma_3}$ acting on sections of $\mathcal{O}_{\Sigma_3}(3m+2)$ if, and only if, $m=-2/3$, in which case we seek functions $\alpha_1$ on $\Sigma_3$ satisfying
\[
 \Delta_{\bar\partial_{\Sigma_3}}\alpha_1=0.
\]
Since $\Sigma_3$ is compact, $\alpha_1$ must be holomorphic and further constant. Taking $g_1=\alpha_2=0$ completes our analysis. 

Finally, we check the case that $m=4/3$. In this case, for the eigenvalues
\[
 \frac{1}{2}\left(\frac{8}{3}+\frac{4}{3}\right)\left(\frac{4}{a}\right)=-\frac{4}{6}(4+2),
\]
we must have $a=-2$. However, in this case, the eigenvalue is equal to $-4$, which is negative and therefore not a possible eigenvalue of $\bar\partial^*_{\Sigma_3}\bar\partial_{\Sigma_3}$ on sections of $\mathcal{O}_{\Sigma_3}(6)$.

We have found a total of six infinitesimal conical Cayley deformations of $C_3$ that are not complex.

\begin{prop}\label{prop:ex3soln}
 The real dimension of the space of infinitesimal conical Cayley deformations of $C_3$ in $\mathbb{C}^4$ is thirty. The real dimension of the space of infinitesimal conical complex deformations of $C_3$ in $\mathbb{C}^4$ is twenty-four.
\end{prop}
\begin{rem}
Similarly to Proposition \ref{prop:ex2soln} we have six infinitesimal conical Cayley deformations of $C_3$ which are not complex, which again implies that these deformations are just rigid motions.
\end{rem}

\subsection{Calculating the \texorpdfstring{$\eta$}{eta}-invariant for an example}\label{ss:etacalc}
The final calculation in this article is to compute the $\eta$-invariant of the Atiyah--Patodi--Singer index theorem \ref{thm:apsind} for one of the examples we considered in Section \ref{sec:conedefs}. This will help us to calculate (what we expect to be) the codimension of the space of conically singular complex CS deformations of a CS complex surface $N$ at $C$ with rate $\mu$ in a Calabi--Yau manifold $M$ inside the space of all complex deformations of $N$, for a certain cone $C$ in $\mathbb{C}^4$, using Theorem \ref{thm:indcomp}.

We consider our simplest example of a two-dimensional complex cone in $\mathbb{C}^4$ which is $C_1=\mathbb{C}^2$. Denote by $\Sigma_1$ the complex link of $C_1$, i.e., $\Sigma_1=\mathbb{C}P^1$. Proposition \ref{prop:evcxlink} told us that the exceptional weights $\lambda\in \mathbb{R}$ satisfy an eigenproblem, and to calculate the $\eta$-invariant we must first find the dimension of the space of solutions to \eqref{eqn:evprobcxlink1}-\eqref{eqn:evprobcxlink2} for each $\lambda\in \mathbb{R}$. Setting $w=0$ in \eqref{eqn:evprobcxlink1}-\eqref{eqn:evprobcxlink2}, we seek holomorphic sections of $\nu^{1,0}_{\mathbb{C}P^3}(\Sigma_1)\otimes \mathcal{O}_{\mathbb{C}P^3}(\lambda-1)|_{\Sigma_1}=\mathcal{O}_{\mathbb{C}P^3}(\lambda)|_{\Sigma_1}\oplus\mathcal{O}_{\mathbb{C}P^3}(\lambda)|_{\Sigma_1}$, for $\lambda\in \mathbb{N}\cup \{0\}$, which by the Hirzebruch--Riemann--Roch theorem \ref{thm:rr} has dimension $2(\lambda+1)$. Similarly, setting $v=0$ in \eqref{eqn:evprobcxlink1}-\eqref{eqn:evprobcxlink2}, we seek antiholomorphic sections of $\mathcal{O}_{\mathbb{C}P^3}(-\lambda)|_{\Sigma_1}\oplus\mathcal{O}_{\mathbb{C}P^3}(-\lambda)|_{\Sigma_1}$, which again have dimension $2(\lambda+1)$.

It remains to compute the multiplicity of $\lambda$ as an eigenvalue of 
\begin{equation}\label{eqn:lapevex1}
2\bar\partial_{\Sigma_1}^*\bar\partial_{\Sigma_1} v=(\lambda-1-m)(\lambda+3+m)v,
\end{equation}
where $v$ is a section of $\mathcal{O}_{\mathbb{C}P^3}(m+1)|_{\Sigma_1}\oplus\mathcal{O}_{\mathbb{C}P^3}(m+1)|_{\Sigma_1}$ and $\lambda\ne 1+m$ or $-3-m$. Theorem \ref{thm:evtwistdirac} tells us that this is equivalent to solving the algebraic equation
\[
(\lambda-1-m)(\lambda+3+m)=4[q^2+q|m+2|],
\]
where $q$ is a positive integer.

It can be computed that the multiplicity of integer $\lambda>0$ as an eigenvalue of \eqref{eqn:lapevex1} is $2\lambda(\lambda+1)$ and the multiplicity of integer $\lambda<-2$ as an eigenvalue of \eqref{eqn:lapevex1} is $2(\lambda+2)(\lambda+1)$. So we have that
\begin{align*}
\eta(s)&=4\sum_{\lambda=1}^\infty \frac{\lambda+1}{\lambda^s}+2\sum_{\lambda=1}^\infty\frac{\lambda(\lambda+1)}{\lambda^s}-2\sum_{\lambda=3}^\infty\frac{(-\lambda+2)(-\lambda+1)}{\lambda^{s}}, \\
&=12\sum_{\lambda=1}^\infty \lambda^{1-s},
\end{align*}
and so 
\[
\eta(0)=12\zeta(-1)=-1,
\]
where $\zeta$ is the Riemann zeta function.

We have that the multiplicity of the zero eigenvalue in this case in four.  So we have found that
\[
\frac{\eta(0)+h}{2}=\frac{3}{2}.
\]

\appendix

\section{Structure equations of \texorpdfstring{$Spin(7)$}{Spin(7)}}
We will here give the structure equations of $S^7$ adapted to an associative submanifold of $S^7$. To do this, we will consider the sphere $S^7$ as the group quotient $Spin(7)/G_2$, that is, we can consider $Spin(7)$ as the $G_2$ frame bundle over $S^7$. Bryant \cite[Prop 1.1]{MR664494} first wrote down the structure equations of $Spin(7)$, but we will quote them in the following useful form given by Lotay \cite[\S4]{MR3004102}.
\begin{prop}[{\cite[Prop 4.2]{MR3004102}}]
We may write the Lie algebra $\mathfrak{spin}(7)$ of the Lie group $Spin(7)\subseteq Gl(n,\mathbb{R})$ as
\begin{align*}
 \mathfrak{spin}(7)= \left\{\left.
\begin{pmatrix}
 0 & -\omega^T & -\eta^T \\
\omega & [\alpha] & -\beta^T-\frac{1}{3}\{\eta\}^T \\
\eta & \beta+\frac{1}{3}\{\eta\} & \frac{1}{2}[\alpha-\omega]_+ +\frac{1}{2}[\gamma]_-
\end{pmatrix}
 \right| \right.&\left.
\begin{matrix}
 \omega,\alpha,\gamma\in M_{3\times 1}(\mathbb{R}), \\
\eta\in M_{4\times 1}(\mathbb{R}),\\
\beta\in M_{4\times 3}(\mathbb{R}),
\end{matrix} \right. \\
\left. 
\begin{matrix}
 \beta^4_1+\beta^7_2+\beta^6_3=0,   \\ 
\beta^6_1-\beta^5_2-\beta^4_3=0, 
\end{matrix}
\right.&\left.
\begin{matrix}
 \beta^5_1+\beta^6_2-\beta^7_3=0, \\
\beta^7_1-\beta^4_2+\beta^5_3=0.
\end{matrix}
\right\},
\end{align*}
where
\[
 [(x,y,z)^T]:=
\begin{pmatrix}
 0 & z & -y \\
-z & 0 & x \\
y & -x & 0
\end{pmatrix},
\]
\[
 [(x,y,z)^T]_\pm:=
\begin{pmatrix}
 0 & -x &-y & \pm z \\ 
x & 0 & z & \pm y \\
y& -z & 0 & \mp x \\
\mp z &\mp y & \pm x & 0
\end{pmatrix},
\]
and 
\[
 \{(p,q,r,s)^T\}:=
\begin{pmatrix}
  -q & -r & s \\
p& s& r \\
-s& p& -q \\
r& -q &-p
\end{pmatrix}.
\]
\end{prop}
Now that we have the structure equations for $Spin(7)$, we may construct a moving frame for $S^7$ adapted to an associative three-fold. 
If we let $g:Spin(7)\to Gl(8,\mathbb{R})$ be the map taking $Spin(7)$ to the identity component of the Lie subgroup of $Gl(8,\mathbb{R})$ which has Lie algebra $\mathfrak{spin}(7)$, then we can write $g=(x\, e \, f)$, where for $p\in Spin(7)$ we have that $x(p)\in M_{8\times 1}(\mathbb{R}), e(p)=(e_1(p),e_2(p),e_3(p))\in M_{8\times 3}(\mathbb{R})$ and $f(p)=(f_4(p),f_5(p),f_6(p),$ $f_7(p))\in M_{8\times 4}(\mathbb{R})$. We can choose our frame so that $x$ represents a point of our associative three-fold $L$, $e$ is an orthonormal frame for $L$ and $\omega$ is an orthonormal coframe for $L$. Therefore $f$ is an orthonormal frame for the normal bundle of $L$ in $S^7$, $\eta$ an orthonormal coframe. Then since the Maurer-Cartan form $\phi=g^{-1}dg$ takes values in $\mathfrak{spin}(7)$, we can write
\[
 \phi:=
\begin{pmatrix}
 0 & -\omega^T & -\eta^T \\
\omega & [\alpha] & -\beta^T-\frac{1}{3}\{\eta\}^T \\
\eta & \beta+\frac{1}{3}\{\eta\} & \frac{1}{2}[\alpha-\omega]_+ +\frac{1}{2}[\gamma]_-
\end{pmatrix}.
\]
This yields the following results
\begin{prop}[{\cite[Prop 4.3]{MR3004102}}]\label{prop:jase1}
 Use the notation above. On the adapted frame bundle of an associative three-fold $L$ in $S^7$, $x:L\to S^7$ and $\{e_1,e_2,e_3,$ $f_4,f_5,f_6,f_7\}$ is a local oriented orthonormal basis for $TA\oplus NA$, so the first structure equations are
\begin{align*}
 dx&=e \omega; \\
de&=-x\omega^T+e[\alpha]+f\beta; \\
df&=-e\beta^T+\frac{1}{2}f([\alpha-\omega]_++[\gamma]_-).
\end{align*}
\end{prop}
\begin{prop}[{\cite[Prop 4.4]{MR3004102}}]\label{prop:jase2}
 Use the notation above. On the adapted frame bundle of an associative three-fold in $S^7$, there exists a local tensor of functions $h=h^a_{jk}=h^a_{kj}$, for $1\le j,k \le 3$ and $4\le a \le 7$, such that the second structure equations are
\begin{align}\label{eqn:struc21}
 d\omega&=-[\alpha]\wedge \omega; \\ \label{eqn:struc22}
\beta&=h\omega; \\ \label{eqn:struc23}
d[\alpha]&=-[\alpha]\wedge [\alpha]+\omega\wedge \omega^T+\beta^T\wedge \beta; \\ \label{eqn:struc24}
d\beta&=-\beta\wedge [\alpha]-\frac{1}{2}([\alpha-\omega]_++[\gamma]_-)\wedge \beta; \\ \label{eqn:struc25}
\frac{1}{2}d([\alpha-\omega]_++[\gamma]_-)&=-\frac{1}{4}[\alpha-\omega]_+\wedge [\alpha-\omega]_+-\frac{1}{4}[\gamma]_-\wedge [\gamma]_-+\beta\wedge \beta^T.
\end{align}
\end{prop}
Notice that $[\alpha]$ is the Levi-Civita connection of $L$ and $\frac{1}{2}([\alpha-\omega]_++[\gamma]_-)$ defines the induced connection on the normal bundle of $L$ in $S^7$. We have that $h$ defines the second fundamental form $\mathbf{II}_L\in C^\infty(S^2T^*L;\nu(L))$ of $L$ in $S^7$, writing
\[
 \mathbf{II}_L:=h^a_{jk}f_a\otimes \omega_j\omega_k.
\]
Since the associative submanifolds of $S^7$ that we are considering are $S^1$-bundles over complex curves, we may reduce the structure equations of $L$.
\begin{prop}[{\cite[Ex 4.9]{MR3004102}}]\label{prop:jase3}
 Let $L$ be the link of complex cone $C$ in $\mathbb{C}^4$. Then we can choose a frame of $TS^7|_L$ such that
\[
 \alpha_2=\omega_2, \; \alpha_3=\omega_3 \text{ and } \beta^4_1=\beta^5_3=\beta^6_3=\beta^7_3=0.
\]
This implies that $\beta^4_3=-\beta^5_2,\beta^5_3=\beta^4_2, \beta^6_3=-\beta^7_2$ and $\beta^7_3=\beta^6_2$. Here $e_1$ defines the direction of the circle fibres of $L$ over the complex link $\Sigma$ of $C$.
\end{prop}
\begin{proof}
 This follows from supposing that the complex structure of $\mathbb{C}^4$ acts on $C$ as follows:
\[
 Jx=e_1; \quad Je_2=e_3; \quad Jf_4=f_5; \quad Jf_6=f_7.
\]
\end{proof}

\textbf{Acknowledgements.} I would like to thank Jason Lotay for his help, guidance and feedback on this project. I would also like to thank Alexei Kovalev, Yng-Ing Lee and Julius Ross for comments on my PhD thesis, from which this work is taken. This work was supported by the UK Engineering and Physical Sciences Research Council (EPSRC) grant EP/H023348/1 for the University of Cambridge Centre for Doctoral Training, the Cambridge Centre for Analysis.

\bibliographystyle{plain}
\bibliography{csbib}

\end{document}